\documentclass[11pt]{article}
\usepackage{amsmath,fullpage,amssymb,amsthm,hyperref,enumitem,mathtools,stmaryrd}
\usepackage{tikz,color}
\usetikzlibrary{calc,patterns,decorations.markings}

\newtheorem{theorem}{Theorem}[section]
\newtheorem{proposition}[theorem]{Proposition}
\newtheorem{lemma}[theorem]{Lemma}
\newtheorem{corollary}[theorem]{Corollary}

\theoremstyle{definition}
\newtheorem{example}[theorem]{Example}  
\newtheorem{definition}[theorem]{Definition}   
 
\theoremstyle{remark}

\newcommand\W{\mathcal W}	
\newcommand\E{\mathcal E}	
\newcommand\N{\mathcal N}	
\newcommand\D{\mathcal D}	
\newcommand\C{\mathcal C}	
\newcommand\CS{\Lambda}	
\newcommand\CSg{\mathcal L}	
\newcommand\WW{W}	
\newcommand\bij{\phi}		
\newcommand\x{\mathbf x}	
\newcommand\y{\mathbf y}	
\newcommand\s{\mathtt s}	
\newcommand\w{\mathtt w}	
\newcommand\e{\mathtt e}	
\newcommand\OW{\mathcal O}	

\newcommand\bars{\bar{s}} 
\newcommand\bdry{\partial}

\newcommand\Z{\mathbb Z}
\DeclareMathOperator\SCT{SCT}	
\DeclareMathOperator\SSCT{SSCT}	
\DeclareMathOperator\SYT{SYT}	
\DeclareMathOperator\OCT{OCT}	
\newcommand\ins{R}
\DeclareMathOperator\CRS{CRS}
\DeclareMathOperator\CRSK{CRSK}
\DeclareMathOperator\wt{wt}
\DeclareMathOperator\RS{RS}
\DeclareMathOperator\ssh{ssh}
\newcommand\blam{\bar\lambda}
\newcommand\lam{\lambda}
\newcommand\balpha{\bar\alpha}
\newcommand\barij{\overline{\langle i,j\rangle}}
\renewcommand\ij{\langle i,j\rangle}
\newcommand{\bT}{\overline T}	
\newcommand{\bU}{\overline U}
\newcommand{\bP}{\overline P}
\newcommand{\bQ}{\overline Q}
\newcommand{\tP}{\tilde P}
\newcommand{\tQ}{\tilde Q}
\newcommand\hatch[2]{\fill[pattern=north east lines,pattern color=gray] (#1,#2) rectangle (#1+1,#2+1);}
\newcommand\halfhatch[2]{\fill[pattern=north east lines,pattern color=gray] (#1,#2) rectangle (#1+1,#2+.5);}
\newcommand\bright[2]{\fill[pattern=grid,pattern color=green] (#1,#2) rectangle (#1+1,#2+1);}
\newcommand\halfbright[2]{\fill[pattern=grid,pattern color=green] (#1,#2) rectangle (#1+1,#2+.5);}
\newcommand\newcell[2]{\fill[pink] (#1,#2) rectangle (#1+1,#2+1);}

\def\north{-- ++(0,1)}
\def\east{-- ++(1,0)}
\def\south{-- ++(0,-1)}
\def\west{-- ++(-1,0)}
\newcommand{\add}[1]{\overset{#1}{\rightarrow}}
\newcommand\rholl{\rho^{\raisebox{1pt}[2pt][0pt]{$\llcorner$}}} 
\newcommand\rhoul{\rho^{\raisebox{-3pt}[2pt][0pt]{$\ulcorner$}}}
\newcommand\rholr{\rho^{\raisebox{1pt}[2pt][0pt]{$\lrcorner$}}} 
\newcommand\rhour{\rho^{\!\!\raisebox{-3pt}[2pt][0pt]{$\urcorner$}}} 
\newcommand{\card}[1]{\left|#1\right|}
\newcommand{\bracket}[1]{\llbracket#1\rrbracket} 
\newcommand{\mn}[1]{\bar{#1}}
\DeclareMathOperator\evac{evac}
\renewcommand\S{\mathcal S}	

\newcommand{\state}[5]{
\draw[color=#5] (#1:#2) circle (1);
\foreach \i in {1,...,#3} {
\filldraw[color=#5,fill=white] (#1:#2) ++(270+180/#3-2*360/#3-\i*360/#3:1) circle (.2);
}
\foreach \i in #4 {
\filldraw[color=#5] (#1:#2) ++(270+180/#3-2*360/#3-\i*360/#3:1) circle (.2);
}}

\newcommand{\arrow}[2]{
\draw[->,shorten <=.55cm,shorten >=.55cm] #1--#2;
}

\newcommand{\arrows}[2]{
\draw[->,thick,shorten <=1mm,shorten >=1mm] #1--#2;
}

\title{Cylindric growth diagrams, walks in simplices,\\ and exclusion processes}

\author{Sergi Elizalde\thanks{Department of Mathematics, Dartmouth College, Hanover, NH 03755, USA. {\tt sergi.elizalde@dartmouth.edu}.}}

\date{}

\begin{document}

\maketitle

\begin{abstract}
We establish bijections between three classes of
combinatorial objects that have been studied in very different contexts: lattice walks in simplicial regions as introduced by Mortimer--Prellberg, standard cylindric tableaux as introduced by Gessel--Krattenthaler and Postnikov, and sequences of states in the totally asymmetric simple exclusion process. This perspective allows us to translate symmetries from one setting into another, revealing unexpected properties of these objects.

Specifically, we show that a recent bijection of Courtiel, Elvey Price and Marcovici between certain simplicial walks with forward and backward steps is equivalent to a special case of a cylindric analogue of the Robinson--Schensted--Knuth correspondence.
Originally defined by Neyman by iterating an insertion operation, we provide an alternative description of this correspondence
by introducing a cylindric version of Fomin's growth diagrams.
This natural description elucidates the symmetry obtained when switching the insertion an recording tableaux, and it 
allows us to interpret the above walks as oscillating cylindric tableaux.
\end{abstract}

\noindent {\bf Keywords:} cylindric tableau, growth diagram, RSK, differential poset, TASEP, constrained walk, oscillating tableau


\section{Introduction}

The first goal of this paper is to explore a connection between three seemingly unrelated combinatorial objects that have appeared in the literature. The first object are certain lattice walks in simplicial regions that were introduced by Mortimer and Prellberg~\cite{mortimer_number_2015}, who used the kernel method to show that they have some surprising enumerative properties, later proved bijectively by Courtiel, Elvey Price and Marcovici~\cite{courtiel_bijections_2021}.
The second object are standard cylindric tableaux, which are a special case of Gessel and Krattenthaler's cylindric partitions~\cite{gessel_cylindric_1997}, later studied by Postnikov in connection to Gromov--Witten invariants~\cite{postnikov_affine_2005}, and recently used by Huh et al.\ to obtain affine analogues of bounded Littlewood identities~\cite{huh_bounded_2025}. Standard cylindric tableaux can be thought of as standard Young tableaux with certain restrictions involving the entries of the first and last rows.
The third object are sequences of states in the totally asymmetric simple exclusion process (TASEP) on a cycle~\cite{ferrari_stationary_2007,liggett_stochastic_1999}; more specifically, walks in the directed graph (with its loops removed) underlying the Markov chain of this process. 

We will show that these three objects are in bijection with each other. In addition to providing new enumerative results, a consequence of the bijections is that some symmetries that are natural in one setting, such as conjugation of standard cylindric tableaux, translate into less obvious involutions in other settings, such as a certain duality between walks on simplices of different dimensions. 
Another remarkable consequence is that we can use a cylindric analogue of the Robinson--Schensted correspondence, due to Neyman~\cite{neyman_cylindric_2015}, to deduce a recent result of Courtiel, Elvey Price and Marcovici~\cite{courtiel_bijections_2021} stating that the number of the above simplicial walks starting at a given point does not change if the direction of the steps is reversed.

Mortimer and Prellberg~\cite{mortimer_number_2015}, as well as  Courtiel et al.~\cite{courtiel_bijections_2021}, also considered more general walks where each step can be taken in either forward or backward direction. When applied to these walks, our bijections produce  oscillating cylindric tableaux, which are cylindric analogues of the well-studied oscillating tableaux \cite{sundaram_tableaux_1990,roby_applications_1991}, and sequences of states in the symmetric simple exclusion process (SSEP), where particles can jump in either direction.

The second goal of this paper is to introduce a cylindric analogue of Fomin's growth diagrams. 
Growth diagrams were first devised in~\cite{fomin_generalized_1986,fomin_duality_1994,fomin_schensted_1995} in order to generalize the Robinson--Schensted (RS) correspondence~\cite{stanley_enumerative_1999}, a celebrated bijection between permutations and pairs of standard Young tableaux of the same shape. 
Over the years, growth diagrams have proved to be very useful in various settings (see e.g.~\cite{krattenthaler_growth_2006}), and they have been extended \cite[Ch.~4]{roby_applications_1991} in order to provide an alternative description of the more general Robinson--Schensted--Knuth (RSK) correspondence, a bijection between matrices with non-negative integer entries and pairs of semistandard Young tableaux of the same shape.
Growth diagrams are built in terms of local rules which describe how to label the vertices of a certain grid by elements of a differential poset~\cite{stanley_differential_1988}. In the basic case, this is the poset of partitions ordered by containment of their shapes,
known as Young's lattice. Even though the analogous lattice of cylindric shapes is not technically a differential poset according to~\cite{stanley_differential_1988}, 
we will show that it can still be used to define a cylindric analogue of growth diagrams, as well as a generalization to the semistandard case. 

Cylindric growth diagrams, as we will call them, have several applications. First, they provide a very natural description of Neyman's cylindric analogue of the RSK correspondence~\cite{neyman_cylindric_2015}, which was originally defined in terms of an insertion operation inspired in Sagan and Stanley's insertion on skew tableaux~\cite{sagan_robinson-schensted_1990}.
In particular, our description using cylindric growth diagrams elucidates the symmetry of the correspondence, similarly to how Fomin's original growth diagrams explain
the symmetry of the classical RS correspondence.

Second, cylindric growth diagrams, in the symmetric case and when restricted to standard cylindric tableaux, provide the ``right'' setting to understand the general case of a bijection of Courtiel et al.~\cite{courtiel_bijections_2021} for simplicial walks with a given starting point and an arbitrary pattern of forward and backward steps. Compared to the ad-hoc description from~\cite{courtiel_bijections_2021}, which is based on a ``convergent rewriting system'' and then rephrased in terms of tilings of square using a certain set of tiles, 
our description arises naturally once we interpret the walks as oscillating cylindric tableaux. In particular, the local rules of the cylindric growth diagrams determine the set of available tiles.

The paper is structured as follows. In Section~\ref{sec:background}, we describe the three combinatorial objects of interest: lattice walks in simplicial regions, standard (and more generally, semistandard) cylindric tableaux, and sequences of states in the TASEP. In Section~\ref{sec:connections}, we describe bijections between these objects, and we study some of their properties. In Section~\ref{sec:CRS}, we use a cylindric analogue of the RS correspondence, 
described by Neyman~\cite{neyman_cylindric_2015} in terms of  insertion operations on cylindric tableaux, to give a bijection between forward and backward simplicial walks starting at a given point.

In Section~\ref{sec:growth}, we introduce the notion of cylindric growth diagrams, and use them to provide a natural and more symmetric description of the cylindric analogue of the RSK correspondence.
In the standard case, the forward and backward rules for cylindric growth diagrams resemble the classical ones, except that the row indices are periodic and that no cells can be filled with crosses. In the semistandard case, our general construction is vaguely reminiscent of a description 
by Berenstein and Kirillov~\cite{berenstein_robinson--schensted--knuth_2001} of a variant of RSK (as a bijection between matrices of nonnegative integers and plane partitions) in terms of piecewise linear transformations. 
In Section~\ref{sec:OCT}, we describe bijections on cylindric oscillating tableaux using growth diagrams, which translate to bijections for simplicial walks with  a given starting point and an arbitrary pattern of forward and backward steps.
Finally, in Section~\ref{sec:further}, we discuss some possible directions for further exploration.

\section{Background}\label{sec:background}

\subsection{Walks in simplices}\label{sec:walks}

We assume throughout the paper that $d,L\ge1$. Define the following simplicial section of the $d$-dimensional integer lattice:
$$\Delta_{d,L}=\{(x_1,x_2,\dots,x_d)\in \mathbb{N}^d:x_1+x_2+\dots+x_d=L\},$$
where $\mathbb{N}$ is the set of non-negative integers.
For $i\in[d]=\{1,2,\dots,d\}$, let $e_i$ be the unit vector whose $i$th coordinate equals $1$ and whose other coordinates equal $0$, and let $s_i=e_{i+1}-e_{i}$,\footnote{This indexing differs from the notation in~\cite{courtiel_bijections_2021} in that it is shifted by one, but it is more convenient for our correspondences with other combinatorial objects.} with the convention $e_{d+1}=e_1$. See Figure~\ref{fig:DeltadL} for an example. We will consider walks in $\Delta_{d,L}$ with steps $s_i$ for $i\in[d]$. Sometimes it will be convenient to think of these as walks in the directed graph whose vertices are the points in $\Delta_{d,L}$, and whose edges are ordered pairs $(\x,\y)$ such that $\y-\x=s_i$ for some $i\in[d]$ (we sometimes label such edges with $s_i$ to refer to them). We denote this directed graph by $\D_{d,L}$. The graphs $\D_{3,2}$ and $\D_{3,3}$ are shown in Figures~\ref{fig:32} and~\ref{fig:33}, respectively. Let $C=(L,0,\dots,0)\in \Delta_{d,L}$ denote one of the corners of the simplex. Walks starting at $C$ correspond to words over the alphabet $\{s_1,s_2,\dots,s_d\}$ such that, in every prefix, the number of steps $s_i$ is no less than the number of steps $s_{i+1}$, for each $i\in[d-1]$, and the number of steps $s_d$ plus $L$ is no less than the number of steps $s_{1}$. This idea will be explored in Section~\ref{sec:connections} to obtain a bijection between walks in $\D_{d,L}$ and certain tableaux.

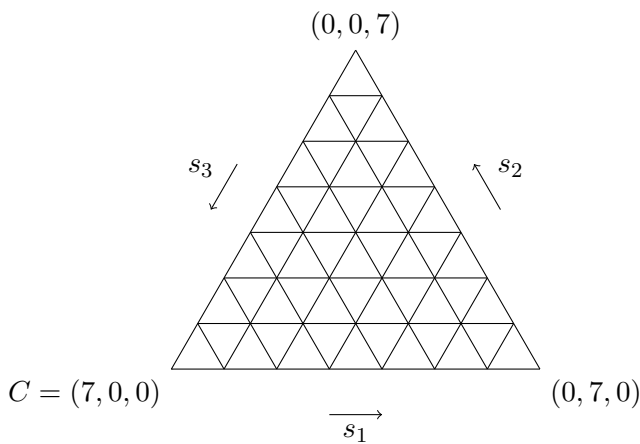
\begin{figure}[htb]
\centering
\newcommand*\rows{7}
\begin{tikzpicture}[scale=0.7]
    \foreach \row in {0, 1, ...,\rows} {
        \draw ($\row*(0.5, {0.5*sqrt(3)})$) -- ($(\rows,0)+\row*(-0.5, {0.5*sqrt(3)})$);
        \draw ($\row*(1, 0)$) -- ($(\rows/2,{\rows/2*sqrt(3)})+\row*(0.5,{-0.5*sqrt(3)})$);
        \draw ($\row*(1, 0)$) -- ($(0,0)+\row*(0.5,{0.5*sqrt(3)})$);
    }
	\draw (1,0) {} coordinate (vecta);
    \draw ($(-.5,{.5*sqrt(3)})$) {} coordinate (vectb);
    \draw ($(-.5,{-.5*sqrt(3)})$) {} coordinate (vectc);
    \draw (0,0) node[below left] {$C=(7,0,0)$} coordinate(O);
    \draw ($\rows*(vecta)$) node[below right] {$(0,7,0)$}  coordinate(R);
    \draw ($-\rows*(vectc)$) node[above] {$(0,0,7)$};
    \draw[->] ($\rows/2*(vecta)+(vectc)$) -- node[below] {$s_1$} ($\rows/2*(vecta)-(vectb)$);
    \draw[->] ($-\rows/2*(vectc)+(vectb)$) -- node[above left] {$s_3$} ($-\rows/2*(vectc)-(vecta)$);
    \draw[->] ($(R)+\rows/2*(vectb)+(vecta)$) -- node[above right] {$s_2$} ($(R)+\rows/2*(vectb)-(vectc)$);
\end{tikzpicture}
   \caption{The simplicial region $\Delta_{3,7}$.}
   \label{fig:DeltadL}    
\end{figure}

For a point $\x\in\Delta_{d,L}$, let $\W^n_{d,L}(\x)$ be the set of $n$-step walks in $\Delta_{d,L}$ starting at $\x$ with steps $s_i$ for $i\in[d]$; equivalently, $n$-step walks in the graph $\D_{d,L}$ starting at $\x$.
The enumeration of these walks is relatively straightforward for $d=2$ \cite[Prop.\ 1, Cor.\ 2]{mortimer_number_2015}, but it becomes much more interesting for $d=3$, where the problem was solved by Mortimer and Prellberg \cite[Thm.~3]{mortimer_number_2015}. In the special case where the starting point is the corner $C$, they showed that, surprisingly, these walks are equinumerous with Motzkin paths of bounded height. 

Recall that a Motzkin path of length $n$ is a path in $\Z^2$ from $(0,0)$ to $(n,0)$, with steps $(1,1)$, $(1,0)$ and $(1,-1)$, not going below the $x$-axis. Let $M_{n,h}$ denote the number of Motzkin paths of length $n$ and height at most $h$, that is, not going above the line $y=h$. Additionally, let $M'_{n,h}$ be the number of Motzkin paths of length $n$ and height at most $h$ that do not have any $(1,0)$ steps on the line $y=h$. For $n\ge0$ and $L\ge1$, let 
\begin{equation}\label{eq:anL} a_{n,L}=\begin{cases} M_{n,h} & \text{if }L=2h+1,\\
M'_{n,h} & \text{if }L=2h.
\end{cases}
\end{equation}

\begin{theorem}[\cite{mortimer_number_2015}]\label{thm:mortimer_number_2015}
For $n\ge0$ and $L\ge1$, $$|\W^n_{3,L}(C)|=a_{n,L}.$$
\end{theorem}

Mortimer and Prellberg's proof~\cite{mortimer_number_2015} uses the kernel method to solve a functional equation. A complicated bijective proof of Theorem~\ref{thm:mortimer_number_2015} was later given by Courtiel et al.~\cite{courtiel_bijections_2021}.
For $d=4$, the following generating function enumerating walks starting at a corner $C$ is also given in~\cite{courtiel_bijections_2021}.

\begin{theorem}[{\cite[Cor.\ 39]{courtiel_bijections_2021}}]\label{thm:d4}
For all $L\ge1$,
$$\sum_{n\ge0}|\W^n_{4,L}(C)|\,t^n=\frac{1}{(L+4)^2}\sum_{\substack{1\le j<k\le L+3\\ 2\nmid j,k}}^{L+4}\frac{(\zeta^k+\zeta^{-k}-\zeta^j-\zeta^{-j})(2+\zeta^j+\zeta^{-j})(2+\zeta^k+\zeta^{-k})}{1-(\zeta^j+\zeta^{-j}+\zeta^k+\zeta^{-k})t},$$
where $\zeta=e^{\frac{i \pi}{L+4}}$.
\end{theorem}

\newcommand{\dA}{7}
\newcommand{\dB}{4}

\begin{figure}[htb]
\centering
\begin{tikzpicture}
\begin{scope}[scale=.8]
\coordinate (a) at (2,0);
\coordinate (b) at (1,1.732);

\foreach \x in {0,...,2}{
	\foreach \y in {\x,...,2}{
		\fill ($\x*(a)+2*(b)-\y*(b)$) circle (.1);
}}

\fill[red] (0,0) circle (.1);
\fill[red] ($2*(a)$) circle (.1);
\fill[red] ($2*(b)$) circle (.1);
\fill[blue] (a) circle (.1);
\fill[blue] (b) circle (.1);
\fill[blue] ($(a)+(b)$) circle (.1);

\foreach \x in {0,...,1}{
	\foreach \y in {\x,...,1}{
		\arrows{($\x*(a)+2*(b)-\y*(b)$)}{($\x*(a)+(b)-\y*(b)$)}
		\arrows{($\x*(b)+(a)-\y*(a)$)}{($\x*(b)+2*(a)-\y*(a)$)}
		\arrows{($\x*(b)+2*(a)-\y*(a)$)}{($\x*(b)+(b)+(a)-\y*(a)$)}
}}
\draw (0,0) node[below] {$C=(2,0,0)$};
\draw ($2*(a)$) node[below] {$(0,2,0)$};
\draw ($2*(b)$) node[above] {$(0,0,2)$};
\draw (0,3.5) node {$\D_{3,2}$};
\end{scope}

\begin{scope}[shift={(7.2,0.2)},scale=0.42]
\state{90}{\dA}{5}{{1,2,3}}{red}
\state{90}{\dB}{5}{{2,4,5}}{blue}
\state{90+360/5}{\dA}{5}{{1,4,5}}{red}
\state{90+360/5}{\dB}{5}{{2,3,5}}{blue}
\state{90+2*360/5}{\dA}{5}{{2,3,4}}{red}
\state{90+2*360/5}{\dB}{5}{{1,3,5}}{blue}
\state{90+3*360/5}{\dA}{5}{{1,2,5}}{red}
\state{90+3*360/5}{\dB}{5}{{1,3,4}}{blue}
\state{90+4*360/5}{\dA}{5}{{3,4,5}}{red}
\state{90+4*360/5}{\dB}{5}{{1,2,4}}{blue}
\foreach \i in {0,...,4}{
	\coordinate (a1) at (90+\i*360/5:\dA);  
	\coordinate (b1) at (90+\i*360/5:\dB);
	\coordinate (a2) at (90+360/5+\i*360/5:\dA);
	\coordinate (b2) at (90+360/5+\i*360/5:\dB);
	\arrow{(a1)}{(b2)}
	\arrow{(b1)}{(a2)}
	\arrow{(b1)}{(b2)}
}
\draw (-5,6) node {$\E_{3,2}$};
\end{scope}
\begin{scope}[shift={(0.5,-2.2)},scale=0.42]
\state{0}{0}{5}{{1,2,3}}{red}
\state{0}{5}{5}{{2,3,5}}{blue}
\draw[->] (1.4,-.3)--(3.6,-.3);
\draw[->] (3.6,.3)--(1.4,.3);
\draw[->] (5.3,1.5) to[out=60,in=120,distance=2cm] (4.7,1.5);
\draw (2,2) node {$\N_{3,2}$};
\end{scope}

\begin{scope}[shift={(-5,-1.7)},scale=.9]
\foreach \x in {(0,0),(3,1),(0,2),(3,3),(0,4)}
{\fill[red] \x circle (.1);}
\foreach \x in {(2,0),(1,1),(2,2),(1,3),(2,4)}
{\fill[blue] \x circle (.1);}
\foreach \y in {0,2}{
		\arrows{(2,\y)}{(1,\y+1)}
		\arrows{(1,\y+1)}{(2,\y+2)}
\foreach \x in {0,2}{
		\arrows{(\x,\y)}{(\x+1,\y+1)}
		\arrows{(\x+1,\y+1)}{(\x,\y+2)}
}}
\draw (1,-.3) node {$\vdots$};
\draw (1,4.3) node {$\vdots$};
\draw (0,0) node[left,scale=.9] {$[0 0 0]$};
\draw (1,1) node[left,xshift=-2,scale=.9] {$[1 0 0]$};
\draw (0,2) node[left,scale=.9] {$[2 0 0]$};
\draw (1,3) node[left,xshift=-2,scale=.9] {$[2 1 0]$};
\draw (0,4) node[left,scale=.9] {$[2 2 0]$};
\draw (2,0) node[right,xshift=2,scale=.9] {$[1 0 \mn1]$};
\draw (3,1) node[right,scale=.9] {$[1 1 \mn1]$};
\draw (2,2) node[right,xshift=2,scale=.9] {$[1 1 0]$};
\draw (3,3) node[right,scale=.9] {$[1 1 1]$};
\draw (2,4) node[right,xshift=2,scale=.9] {$[2 1 1]$};
\draw (1,5) node {$\CSg_{3,2}$};
\end{scope}
\end{tikzpicture}
\caption{The graphs $\CSg_{3,2}$, $\D_{3,2}$ and $\E_{3,2}$, and the multigraph $\N_{3,2}$. 
}
\label{fig:32}
\end{figure} 

\newcommand{\dAA}{10}
\newcommand{\dD}{13}

\begin{figure}[htb]
\centering
\begin{tikzpicture}
\begin{scope}[scale=.8]
\coordinate (a) at (2,0);
\coordinate (b) at (1,1.732);

\foreach \x in {0,...,3}{
	\foreach \y in {\x,...,3}{
		\fill ($\x*(a)+3*(b)-\y*(b)$) circle (.1);
}}

\fill[red] (0,0) circle (.1);
\fill[red] ($3*(a)$) circle (.1);
\fill[red] ($3*(b)$) circle (.1);
\fill[blue] (a) circle (.1);
\fill[blue] ($2*(a)+(b)$) circle (.1);
\fill[blue] ($2*(b)$) circle (.1);
\fill[green] (b) circle (.1);
\fill[green] ($2*(b)+(a)$) circle (.1);
\fill[green] ($2*(a)$) circle (.1);
\fill[brown] ($(b)+(a)$) circle (.1);

\foreach \x in {0,...,2}{
	\foreach \y in {\x,...,2}{
		\arrows{($\x*(a)+3*(b)-\y*(b)$)}{($\x*(a)+2*(b)-\y*(b)$)}
		\arrows{($\x*(b)+2*(a)-\y*(a)$)}{($\x*(b)+3*(a)-\y*(a)$)}
		\arrows{($\x*(b)+3*(a)-\y*(a)$)}{($\x*(b)+(b)+2*(a)-\y*(a)$)}
}}
\draw (0,0) node[below] {$C=(3,0,0)$};
\draw ($3*(a)$) node[below] {$(0,3,0)$};
\draw ($3*(b)$) node[above] {$(0,0,3)$};
\draw (0,4.5) node {$\D_{3,3}$};
\end{scope}

\begin{scope}[shift={(2.5,-8.5)},scale=0.42]
\state{0}{0}{6}{{2,3,4}}{red} 
\state{0}{5}{6}{{1,3,4}}{blue} 
\state{90}{5}{6}{{3,4,6}}{green}  
\state{45}{7.071}{6}{{2,4,6}}{brown} 
\arrow{(0,0)}{(5,0)}
\arrow{(0,5)}{(0,0)}
\arrow{(0,5)}{(5,5)}
\arrow{(5,5)}{(5,0)}
\arrow{(5.3,5)}{(5.3,0)}
\arrow{(4.7,5)}{(4.7,0)}
\arrow{(5.15,0.15)}{(0.15,5.15)}
\arrow{(4.85,-0.15)}{(-0.15,4.85)}
\draw (2,7) node {$\N_{3,3}$};
\end{scope}

\begin{scope}[shift={(-5,-5)},scale=0.42]
\state{180}{\dD}{6}{{2,4,6}}{brown}
\state{180}{\dAA}{6}{{1,2,3}}{red}
\state{180}{\dA}{6}{{3,4,5}}{red}
\state{180}{\dB}{6}{{1,5,6}}{red}
\state{180+360/6}{\dAA}{6}{{2,3,6}}{blue} 
\state{180+360/6}{\dA}{6}{{2,4,5}}{blue}
\state{180+360/6}{\dB}{6}{{1,4,6}}{blue}
\state{180+2*360/6}{\dAA}{6}{{2,3,5}}{green}
\state{180+2*360/6}{\dA}{6}{{1,3,6}}{green}
\state{180+2*360/6}{\dB}{6}{{1,4,5}}{green}
\state{180+3*360/6}{\dD}{6}{{1,3,5}}{brown}
\state{180+3*360/6}{\dAA}{6}{{2,3,4}}{red}
\state{180+3*360/6}{\dA}{6}{{1,2,6}}{red}
\state{180+3*360/6}{\dB}{6}{{4,5,6}}{red}
\state{180+4*360/6}{\dAA}{6}{{1,3,4}}{blue}
\state{180+4*360/6}{\dA}{6}{{1,2,5}}{blue}
\state{180+4*360/6}{\dB}{6}{{3,5,6}}{blue}
\state{180+5*360/6}{\dAA}{6}{{1,2,4}}{green}
\state{180+5*360/6}{\dA}{6}{{3,4,6}}{green}
\state{180+5*360/6}{\dB}{6}{{2,5,6}}{green}
\foreach \i in {0,2,3,5}{
	\coordinate (aa1) at (180+\i*360/6:\dAA);
	\coordinate (a1) at (180+\i*360/6:\dA);
	\coordinate (b1) at (180+\i*360/6:\dB);
	\coordinate (aa2) at (180+360/6+\i*360/6:\dAA);
	\coordinate (a2) at (180+360/6+\i*360/6:\dA);
	\coordinate (b2) at (180+360/6+\i*360/6:\dB);
	\arrow{(aa1)}{(aa2)}
	\arrow{(a1)}{(a2)}
	\arrow{(b1)}{(b2)}
}
\foreach \i in {1,4}{
	\coordinate (aa1) at (180+\i*360/6:\dAA);
	\coordinate (a1) at (180+\i*360/6:\dA);
	\coordinate (b1) at (180+\i*360/6:\dB);
	\coordinate (aa2) at (180+360/6+\i*360/6:\dAA);
	\coordinate (a2) at (180+360/6+\i*360/6:\dA);
	\coordinate (b2) at (180+360/6+\i*360/6:\dB);
	\arrow{(aa1)}{(aa2)}
	\arrow{(aa1)}{(a2)}
	\arrow{(a1)}{(aa2)}
	\arrow{(a1)}{(b2)}
	\arrow{(b1)}{(a2)}
	\arrow{(b1)}{(b2)}
}
\foreach \i in {2,5}{
	\coordinate (aa1) at (180+\i*360/6:\dAA);
	\coordinate (a1) at (180+\i*360/6:\dA);
	\coordinate (b1) at (180+\i*360/6:\dB);
	\coordinate (d) at (180+360/6+\i*360/6-3:\dD+.5);
	\arrow{(aa1)}{(d)}
	\arrow{(a1)}{(d)}
	\arrow{(b1)}{(d)}
}
\foreach \i in {1,4}{
	\coordinate (aa1) at (180+\i*360/6:\dAA);
	\coordinate (a1) at (180+\i*360/6:\dA);
	\coordinate (b1) at (180+\i*360/6:\dB);
	\coordinate (d) at (180-360/6+\i*360/6+3:\dD+.5);
	\arrow{(d)}{(aa1)}
	\arrow{(d)}{(a1)}
	\arrow{(d)}{(b1)}
}
\draw (-11,7) node {$\E_{3,3}$};
\end{scope}

\begin{scope}[shift={(-9,0.2)},scale=.9]
\foreach \x in {(0,0),(4,0),(6,0),(0,3),(4,3),(6,3)}
{\fill[red] \x circle (.1);}
\foreach \x in {(2,0),(2,3)}
{\fill[brown] \x circle (.1);}
\foreach \x in {(1,1),(3,1),(5,1),(1,4),(3,4),(5,4)}
{\fill[blue] \x circle (.1);}
\foreach \x in {(1,2),(3,2),(5,2)}
{\fill[green] \x circle (.1);}

\draw (3,-.3) node {$\vdots$};
\draw (3,4.6) node {$\vdots$};
\draw (0,0) node[left,scale=.9] {$[0 0 0]$};
\draw (0,3) node[left,scale=.9] {$[1 1 1]$};
\draw (2,0) node[left,xshift=-2,scale=.9] {$[1 0 \mn1]$};
\draw (2,3) node[left,xshift=-2,scale=.9] {$[2 1 0]$};
\draw (1,1) node[left,xshift=-2,scale=.9] {$[1 0 0]$};
\draw (1,4) node[left,xshift=-2,scale=.9] {$[2 1 1]$};
\draw (3,1) node[left,xshift=-5,scale=.9] {$[1 1 \mn1]$};
\draw (3,4) node[left,xshift=-2,scale=.9] {$[2 2 0]$};
\draw (1,2) node[left,xshift=-2,scale=.9] {$[1 1 0]$};
\draw (3,2) node[left,xshift=-2,scale=.9] {$[2 0 0]$};
\draw (5,2) node[right,xshift=2,scale=.9] {$[2 1 \mn1]$};
\draw (5,1) node[right,scale=.9] {$[2 0 \mn1]$};
\draw (5,4) node[right,scale=.9] {$[3 1 0]$};
\draw (4,0) node[right,xshift=2,scale=.9] {$[2 \mn1 \mn1]$};
\draw (4,3) node[right,xshift=2,scale=.9] {$[3 0 0]$};
\draw (6,0) node[right,scale=.9] {$[1 1 \mn2]$};
\draw (6,3) node[right,scale=.9] {$[2 2 \mn1]$};
\arrows{(0,0)}{(1,1)} \arrows{(0,3)}{(1,4)}
\arrows{(2,0)}{(1,1)} \arrows{(2,3)}{(1,4)} 
\arrows{(2,0)}{(3,1)} \arrows{(2,3)}{(3,4)} 
\arrows{(2,0)}{(5,1)} \arrows{(2,3)}{(5,4)} 
\arrows{(4,0)}{(5,1)} \arrows{(4,3)}{(5,4)}
\arrows{(6,0)}{(3,1)} \arrows{(6,3)}{(3,4)}
\arrows{(1,1)}{(1,2)}\arrows{(1,1)}{(3,2)}
\arrows{(3,1)}{(1,2)}\arrows{(3,1)}{(5,2)}
\arrows{(5,1)}{(3,2)}\arrows{(5,1)}{(5,2)}
\arrows{(1,2)}{(0,3)}\arrows{(1,2)}{(2,3)}
\arrows{(3,2)}{(2,3)}\arrows{(3,2)}{(4,3)}
\arrows{(5,2)}{(2,3)}\arrows{(5,2)}{(6,3)}

\draw (1.5,4.8) node {$\CSg_{3,3}$};
\end{scope}
\end{tikzpicture}
\caption{The graphs $\CSg_{3,3}$, $\D_{3,3}$ and $\E_{3,3}$, and the multigraph $\N_{3,3}$. 
}
\label{fig:33}
\end{figure} 

Another important result of Courtiel et al.\ is \cite[Thm.~8]{courtiel_bijections_2021} (see also \cite[Thm.\ 34]{courtiel_bijections_2021} for the case of arbitrary $d$), which can be formulated as follows.

\begin{theorem}[\cite{courtiel_bijections_2021}]\label{thm:courtiel_bijections_2021}
For any $\x\in\Delta_{d,L}$, there is a bijection between the set $\W_{d,L}^n(\x)$ of $n$-step walks in $\D_{d,L}$ starting at $\x$ and the set of $n$-step walks in $\D_{d,L}$ ending at $\x$.
\end{theorem}

In Section~\ref{sec:CRS}, we will provide an alternative construction of this bijection based on a cylindric analogue of the RS correspondence.

Mortimer and Prellberg~\cite{mortimer_number_2015}, and later Courtiel et al.~\cite{courtiel_bijections_2021}, also consider walks in $\Delta_{d,L}$ which, in addition to steps $s_i$, can have steps $\bars_i\coloneqq -s_i$ for $i\in[d]$. 
We think of these as walks in the directed graph $\D_{d,L}$ where edges may be used in both directions, and we call them 
{\em oscillating walks}. Given a point $\x\in\Delta_{d,L}$, let $\OW^n_{d,L}(\x)$ be the set of $n$-step oscillating walks in $\Delta_{d,L}$ starting at~$\x$.
Following~\cite{courtiel_bijections_2021}, we call steps $s_i$ {\em forward steps}, and steps $\bars_i$ {\em backward steps}. 
More generally, we can define oscillating walks in any directed multigraph by allowing the edges to be followed in the forward or backward direction.
The {\em type} of an $n$-step oscillating walk in a directed multigraph is the word $w\in\{+,-\}^n$ whose $r$th entry is a $+$ or a $-$ if the $r$th step of the walk is a forward or a backward step, respectively.
Oscillating walks of type $+^n$ are simply called {\em walks}.
For $w\in\{+,-\}^n$, let $\OW^w_{d,L}(\x)$ be the set of oscillating walks in $\D_{d,L}$ of type $w$ starting at $\x$.
In \cite[Thm.\ 6 \& 34]{courtiel_bijections_2021}, Courtiel et al.\ prove the following.

\begin{theorem}[\cite{courtiel_bijections_2021}]\label{thm:osc_walks}
Let $\x\in\Delta_{d,L}$ and $n\ge0$. For any $w,w'\in\{+,-\}^n$, there is a bijection between $\OW^w_{d,L}(\x)$ and $\OW^{w'}_{d,L}(\x)$. 
\end{theorem}

In Section~\ref{sec:growth}, we will use cylindric growth diagrams to provide an alternative proof of this result. In the special case $w=+^n$ and $w'=-^n$, Theorem~\ref{thm:osc_walks} reduces to Theorem~\ref{thm:courtiel_bijections_2021}, since walks starting at $\x$ that consist of backward steps are equivalent, by reversal, to walks ending at $\x$ that consist of forward steps.

\subsection{Standard cylindric tableaux}\label{sec:SCT}

Cylindric partitions were introduced by Gessel and Krattenthaler in~\cite{gessel_cylindric_1997}, as plane partitions satisfying certain constraints between the entries of the first and the last row. A particularly interesting special case of them, called $(0,1)$-cylindric partitions in~\cite{gessel_cylindric_1997}, is equivalent to {\em semistandard cylindric tableaux}, as defined by Postnikov in~\cite{postnikov_affine_2005}.
Postnikov showed that these tableaux are related to the 3-point Gromov--Witten invariants, which are the structure constants of the quantum cohomology ring of the Grassmannian. Specifically, he introduced a generalization of Schur functions, as generating functions for semistandard cylindric tableaux of a given shape, and he showed that, when these are expanded in terms of ordinary Schur functions, the Gromov--Witten invariants appear as coefficients in this expansion. These generalized Schur functions, under the name of {\em cylindric skew Schur functions}, have been further studied by 
McNamara~\cite{mcnamara_cylindric_2006} from the perspective of Schur-positivity.

For positive integers $d$ and $L$, define the cylinder $\C_{d,L}$ to be the quotient $\Z^2/(-d,L)\Z$. Elements of $\C_{d,L}$ are equivalence classes $\langle i,j\rangle =(i,j)+(-d,L)\Z$, where $i,j\in\Z$.  We borrow this notation from~\cite{postnikov_affine_2005,neyman_cylindric_2015}, although we use $L$ instead of $d+L$ for the second index.
As in~\cite{postnikov_affine_2005,neyman_cylindric_2015}, we draw points $(i,j)\in\Z^2$ as unit squares on the plane with vertices $(i-1,j-1),\allowbreak(i-1,j),\allowbreak(i,j-1),\allowbreak(i,j)$, with the unusual convention that the positive $x$-axis points downward and the positive $y$-axis points to the right, to be consistent with the English notation for Young diagrams. With this $90^\circ$-rotation of the usual Cartesian coordinates, $(i,j)$ represents the square in row $i$ and column $j$,
where row indices increase from top to bottom, and column indices increase from left to right. 
We use the term {\em cell} to refer to an equivalence class $\langle i,j\rangle$ of squares.

A {\em cylindric shape\footnote{Such an object is called a cylindric partition in \cite{neyman_cylindric_2015}, but we avoid this term because in \cite{gessel_cylindric_1997} it is used to mean something different, namely the cylindric version of a plane partition} of period $(d,L)$} is a doubly infinite weakly decreasing sequence of integers, $\lambda=(\lambda_i)_{i\in\Z}$, such that $\lambda_i=\lambda_{i+d}+L$ for all $i\in\Z$. 
We often write it as $\lambda=[\lambda_1,\lambda_2,\dots,\lambda_d]$, since these $d$ integers uniquely determine a cylindric shape, provided that 
$\lambda_d+L\ge \lambda_1\ge\lambda_2\ge\dots\ge\lambda_d$. In some of the figures we will omit the commas and use bars to denote negative numbers.
Denote by $\CS_{d,L}$ the set of cylindric shapes of period $(d,L)$. We will use the term {\em shape} to mean {\em cylindric shape} throughout the paper.
For $\lambda\in\CS_{d,L}$, the region $\{(i,j)\in\Z^2:j\le\lambda_i\}$ is a union of equivalence classes, which we denote by 
$Y_\lambda=\{\langle i,j\rangle\in\C_{d,L}:j\le\lambda_i\}$, and we call it the {\em Young diagram} of $\lambda$. We often identify $\lambda$ with $Y_\lambda$ when we talk about adding or removing cells to $\lambda$. The boundary of $Y_\lambda$ determines a closed lattice path (also called a {\em loop}) in $\C_{d,L}$ with steps $\s=(1,0)$  and $\w=(0,-1)$ (south and west in our orientation, respectively). We denote the doubly infinite periodic sequence of steps of this path by 
\begin{equation}\label{eq:bdry}
\bdry_\lambda=(\w^{\lambda_0-\lambda_{1}}\s\w^{\lambda_{1}-\lambda_{2}}\s\dots\w^{\lambda_{d-1}-\lambda_{d}}\s)^\infty,
\end{equation}
where exponentiation indicates repetition.
For example, for the shape $\mu=[3,1,0]\in\CS_{3,4}$, the upper (\textcolor{red}{red}) path in Figure~\ref{fig:SCT} has step sequence $\bdry_\mu=(\w\s\w\w\s\w\s)^\infty$.

\begin{figure}[htb]
\centering
\begin{tikzpicture}[scale=0.5]
 \fill[yellow!25] (-1,-.5)--(-1,0) \east \north\east\north\east\east\north\east \north\east\north\east\east\north\east -- ++(0,.5)-- ++(3,0) -- ++(0,-.5) \west\west\south\south\west\west\south\west\west\south\south\west\west\south\west\west -- ++(0,-0.5) -- ++(-2,0);
 \fill[yellow] (0,0) \north\east\north\east\east\north\east\east\south\south\west\west\south\west\west\west;
 \draw[very thin,dotted] (-1,0)  grid (11,6);
 \draw[red, ultra thick] (-1,-.5)--(-1,0) \east \north\east\north\east\east\north\east \north\east\north\east\east\north\east -- ++(0,.5);
 \draw[blue, ultra thick] (1,-.5)--(1,0) \east\east \north\east\east\north\north\east\east \north\east\east\north\north\east\east -- ++(0,.5);
 \draw[thick,<->,olive] (-.2,0)-- node[left] {$d$} (-.2,3);
 \draw[thick,<->,olive] (0,3.2)-- node[above] {$L$} (4,3.2); 
 \draw (1.5,1.5) node {1}; \draw (1.5+4,1.5+3) node {\footnotesize1};
 \draw (3.5,2.5) node {2}; \draw (3.5+4,2.5+3) node {\footnotesize2};  \draw (3.5-4,2.5-3) node {\footnotesize2};
 \draw (.5,.5) node {3}; \draw (.5+4,.5+3) node {\footnotesize3}; \draw (.5+8,.5+6) node {\footnotesize3};
 \draw (1.5,.5) node {4}; \draw (1.5+4,.5+3) node {\footnotesize4}; \draw (1.5+8,.5+6) node {\footnotesize4};
 \draw (2.5,1.5) node {5}; \draw (2.5+4,1.5+3) node {\footnotesize5};
 \draw (2.5,.5) node {6}; \draw (2.5+4,.5+3) node {\footnotesize6}; \draw (2.5+8,.5+6) node {\footnotesize6};
 \draw (4.5,2.5) node {7}; \draw (4.5+4,2.5+3) node {\footnotesize7}; \draw (4.5-4,2.5-3) node {\footnotesize7};
 \draw (3.5,1.5) node {8}; \draw (3.5+4,1.5+3) node {\footnotesize8};
 \draw (4.5,1.5) node {9}; \draw (4.5+4,1.5+3) node {\footnotesize9};
 \draw[red] (2,2.5) node {$\mu$};
 \draw[blue] (5.5,2) node {$\lambda$};
 \draw[thick,->] (0,6.5)--(0,-1.5) node[below] {$x$};
\draw[thick,->] (-1,3)--(11.5,3) node[right] {$y$};

\begin{scope}[shift={(16,0)}]
 \fill[yellow!25] (-1,-.5)--(-1,0) \east \north\east\north\east\east\north\east \north\east\north\east\east\north\east -- ++(0,.5)-- ++(3,0) -- ++(0,-.5) \west\west\south\south\west\west\south\west\west\south\south\west\west\south\west\west -- ++(0,-0.5) -- ++(-2,0);
 \fill[yellow] (0,0) \north\east\north\east\east\north\east\east\south\south\west\west\south\west\west\west;
 \draw[very thin,dotted] (-1,0)  grid (11,6);
 \draw[red, ultra thick] (-1,-.5)--(-1,0) \east \north\east\north\east\east\north\east \north\east\north\east\east\north\east -- ++(0,.5);
 \draw[blue, ultra thick] (1,-.5)--(1,0) \east\east \north\east\east\north\north\east\east \north\east\east\north\north\east\east -- ++(0,.5);
 \draw[thick,<->,olive] (-.2,0)-- node[left] {$d$} (-.2,3);
 \draw[thick,<->,olive] (0,3.2)-- node[above] {$L$} (4,3.2); 
 \draw (1.5,1.5) node {1}; \draw (1.5+4,1.5+3) node {\footnotesize1};
 \draw (3.5,2.5) node {1}; \draw (3.5+4,2.5+3) node {\footnotesize1};  \draw (3.5-4,2.5-3) node {\footnotesize1};
 \draw (.5,.5) node {1}; \draw (.5+4,.5+3) node {\footnotesize1}; \draw (.5+8,.5+6) node {\footnotesize1};
 \draw (1.5,.5) node {2}; \draw (1.5+4,.5+3) node {\footnotesize2}; \draw (1.5+8,.5+6) node {\footnotesize2};
 \draw (2.5,1.5) node {2}; \draw (2.5+4,1.5+3) node {\footnotesize2};
 \draw (2.5,.5) node {5}; \draw (2.5+4,.5+3) node {\footnotesize5}; \draw (2.5+8,.5+6) node {\footnotesize5};
 \draw (4.5,2.5) node {4}; \draw (4.5+4,2.5+3) node {\footnotesize4}; \draw (4.5-4,2.5-3) node {\footnotesize4};
 \draw (3.5,1.5) node {2}; \draw (3.5+4,1.5+3) node {\footnotesize2};
 \draw (4.5,1.5) node {5}; \draw (4.5+4,1.5+3) node {\footnotesize5};
 \draw[red] (2,2.5) node {$\mu$};
 \draw[blue] (5.5,2) node {$\lambda$};
 \draw[thick,->] (0,6.5)--(0,-1.5) node[below] {$x$};
\draw[thick,->] (-1,3)--(11.5,3) node[right] {$y$};
\end{scope}
\end{tikzpicture}
\caption{A standard cylindric tableau (left) and a semistandard cylindric tableau (right) of shape $\lambda/\mu$, where $\lambda=[5,5,3]$ and $\mu=[3,1,0]$ are elements of $\CS_{3,4}$. Here $|\lambda/\mu|=9$.}
\label{fig:SCT}
\end{figure}
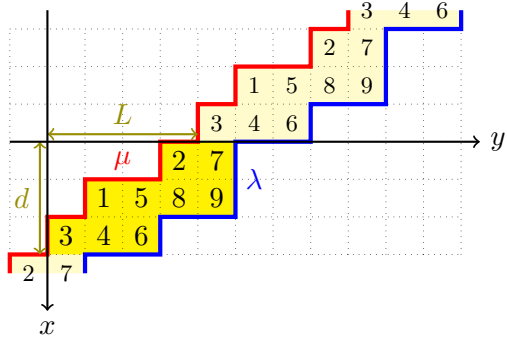 

Consider the partial order on $\CS_{d,L}$ defined by containment of their Young diagrams: given $\lambda,\mu\in\CS_{d,L}$, write $\mu\subseteq\lambda$ if $Y_\mu\subseteq Y_\lambda$; equivalently, if $\mu_i\le\lambda_i$ for all $i\in\Z$. For $\mu\subseteq\lambda$, we define the {\em cylindric Young diagram of shape $\lambda/\mu$} to be the set
$$Y_{\lambda/\mu}=Y_\lambda\setminus Y_\mu=\{\langle i,j \rangle\in\C_{d,L}: \mu_i< j\le \lambda_i\}.$$
Unlike $Y_\lambda$ or $Y_\mu$, the set $Y_{\lambda/\mu}$ is finite. We denote its cardinality by $|\lambda/\mu|=\sum_{i=1}^d (\lambda_i-\mu_i)$. Note that $Y_{\lambda/\mu}$ is the set of cells between the boundary paths of $Y_\mu$ and $Y_\lambda$. 
It will be convenient to also define $|\lambda|=\lambda_1+\dots+\lambda_d$, which may be negative.

A {\em semistandard cylindric tableau} of shape $\lambda/\mu$ is a function $T:Y_{\lambda/\mu}\to\{1,2,3,\dots\}$ such that $T(\langle i,j\rangle)\le T(\langle i+1,j\rangle)$ and $T(\langle i,j\rangle)<T(\langle i,j+1\rangle)$ for all $i,j$; equivalently, a filling of the cells in $Y_{\lambda/\mu}$ so that entries are weakly increasing along rows (from left to right) and strictly increasing along columns (from top to bottom). 
Such a tableau is {\em standard} if it is filled with the entries $1,2,\dots,|\lambda/\mu|$, each appearing exactly once. 
See the examples in Figure~\ref{fig:SCT}. 

Denote by $\SSCT_{d,L}(\lambda/\mu)$ (resp.\ $\SCT_{d,L}(\lambda/\mu)$) the set of semistandard (resp.\ standard) cylindric tableaux of period $(d,L)$ and shape $\lambda/\mu$.
If $T\in\SSCT_{d,L}(\lambda/\mu)$, we call $\lambda$ and $\mu$ the outer shape and the inner shape of $T$, respectively. 
Define the weight $\wt(T)=\prod_{i}x_i^{c_i}$, where $c_i$ is the number of times that the entry $i$ appears in $T$.
For example, if $T\in\SSCT_{3,4}([5,5,3]/[3,1,0])$ is the tableau from the right of Figure~\ref{fig:SCT}, then $\wt(T)=x_1^3x_2^3x_4x_5^2$.
Denote by $\SSCT_{d,L}^n(\cdot/\mu)$ (respectively, $\SSCT_{d,L}^n(\lambda/\cdot)$) the set of semistandard cylindric tableaux with $n$ cells and inner shape $\mu$ (respectively, outer shape $\lambda$), and define analogous notation for $\SCT$.

We will call $\lambda/\mu$ a {\em horizontal strip} if the cylindric Young diagram $Y_{\lambda/\mu}$ contains at most one cell in each column; equivalently, if $\mu_i\le\lambda_i\le\mu_{i-1}$ for all $i$. In any semistandard cylindric tableau, the set of cells containing a given number form a horizontal strip.

Let $\CSg_{d,L}$ be the infinite directed graph obtained from the Hasse diagram of the poset $(\CS_{d,L},\subseteq)$ by orienting its edges upwards; see Figures~\ref{fig:32} and~\ref{fig:33} for examples. In other words, this graph has vertex set $\CS_{d,L}$ and, for $\mu,\lambda\in\CS_{d,L}$, it has an edge from $\mu$ to $\lambda$ if $\lambda$ can be obtained from $\mu$ by adding a cell; equivalently, if $\lambda_i=\mu_i+1$ for some $i\in [d]$ and $\lambda_j=\mu_j$ for all $j\in[d]\setminus\{i\}$. 
Note that it is possible to add a cell to row $i$ of $\mu$ if and only if $\mu_{i-1}>\mu_i$. 

For $\lambda,\mu\in\CS_{d,L}$ with $\mu\subseteq\lambda$ and  $|\lambda/\mu|=n$, elements of $\SCT_{d,L}(\lambda/\mu)$ are in bijection with $n$-step walks from $\mu$ to $\lambda$ in the graph $\CSg_{d,L}$; equivalently, sequences of shapes $\mu=\lambda^{(0)}\subset\lambda^{(1)}\subset\dots\subset\lambda^{(n)}=\lambda$, where $\lambda^{(k)}\in \CS_{d,L}$ for all $k$, and each 
$\lambda^{(k)}$, for $1\le k\le n$, is obtained from 
$\lambda^{(k-1)}$ by adding a cell. 
Similarly, elements of $\SSCT_{d,L}(\lambda/\mu)$ with entries at most $N$ are in bijection with sequences of shapes $\mu=\lambda^{(0)}\subseteq\lambda^{(1)}\subseteq\dots\subseteq\lambda^{(N)}=\lambda$, where $\lambda^{(k)}\in \CS_{d,L}$ for all $k$, and $\lambda^{(k)}/\lambda^{(k-1)}$ is a horizontal strip for each $1\le k\le N$. 
For $T\in\SSCT_{d,L}(\lambda/\mu)$, we denote this sequence by 
\begin{equation}\label{eq:ssh}
\ssh(T)=(\lambda^{(0)},\lambda^{(1)},\dots),
\end{equation}
where $\lambda^{(k)}$ is the outer shape of the tableau consisting of the cells of $T$ with entries at most~$k$.

In analogy with the definition of oscillating tableaux, which have been widely studied~\cite{sundaram_tableaux_1990,roby_applications_1991,krattenthaler_bijections_2016}, 
we define {\em oscillating cylindric tableaux} to be sequences $\rho^0,\rho^1,\dots,\rho^n$ of cylindric shapes in $\CS_{d,L}$ where each shape is obtained from the previous one by either adding or removing one cell. One difference with the non-cylindric case is that it is always possible to remove a cell from a cylindric shape.
Oscillating cylindric tableaux can be viewed as oscillating walks in the graph $\CSg_{d,L}$.
For $\alpha,\beta\in\CS_{d,L}$, let $\OCT_{d,L}^n(\alpha,\cdot)$ denote the set of oscillating cylindric tableaux $\rho^0,\rho^1,\dots,\rho^n$ where $\rho^0=\alpha$, and let $\OCT_{d,L}^n(\alpha,\beta)$ denote the set of those where additionally $\rho^n=\beta$. The type of such an oscillating tableau is the word in $\{+,-\}^n$ whose $k$th entry is a $+$ if $\rho^{k-1}\subset\rho^k$, and a $-$ if $\rho^{k-1}\supset\rho^k$, for $1\le k\le n$.
Oscillating cylindric tableau of type $+^n$ can be identified with standard cylindric tableaux with $n$ cells via the correspondence $\ssh$ from Equation~\eqref{eq:ssh}. 
Given $w\in\{+,-\}^n$, denote by $\OCT_{d,L}^w(\alpha,\cdot)$ the set of elements of $\OCT_{d,L}^n(\alpha,\cdot)$ of type $w$, and define 
$\OCT_{d,L}^w(\alpha,\beta)$ similarly.

\subsection{Exclusion processes}\label{sec:TASEP}

In the {\em totally asymmetric simple exclusion process} (TASEP) on the cycle $\Z_N$, each of the sites $r\in\Z_N$ can either contain a particle or be empty. We denote a state of the system by $u=u_1u_2\dots u_N$, where $u_r=1$ if site $r$ contains a particle, and $u_r=0$ otherwise. The indices of $u$ are always interpreted modulo $N$. We draw such a state by placing $N$ beads around a circle representing the sites in clockwise order, starting and ending at the bottom of the circle, and filling in the beads corresponding to particles.
At each time step, a particle can jump to the next site in counterclockwise direction if this site is empty. 
The number of particles remains fixed through the process, let $d$ denote this number, and assume that $d\ge1$.

Typically, one associates transition probabilities to these particle jumps to define a Markov chain, see~\cite{ferrari_stationary_2007,liggett_stochastic_1999}. Here, however, we are more interested in the underlying directed graph $\E_{d,N-d}$ whose vertices are the $\binom{N}{d}$ states $u=u_1u_2\dots u_N$ containing $d$ particles, and whose edges correspond to valid jumps of a particle.
Specifically, $\E_{d,N-d}$ has an edge from state $u$ to state $v$ if and only if there exists $r\in[N]$ such that $u_{r-1}u_r=01$, $v_{r-1}v_r=10$,  and $u_t=v_t$ for all $t\in\Z_N\setminus\{r-1,r\}$, with indices taken modulo $N$. 
Note that $\E_{d,N-d}$ does not contain loops, unlike the Markov chain for the TASEP, where each of the $N$ sites is equally likely to {\it attempt} a jump, 
but it succeeds only if that site contains a particle and the next site counterclockwise is empty, and it stays in the same state otherwise. We will consider walks in the directed graph $\E_{d,N-d}$, which record sequences of valid jumps.

We also define the directed multigraph $\N_{d,N-d}$ as the quotient of $\E_{d,N-d}$ under the equivalence relation given by cyclic rotations.
Specifically, define a {\em necklace} to be an equivalence class of vertices of $\E_{d,N-d}$, where $u\sim v$ if there exists $t$ such that $u_r=v_{r+t}$ for all $r$, again with indices taken modulo $N$. The vertices of $\N_{d,N-d}$ are necklaces, for which we use the notation $\bracket{u}$. The number of edges from necklace $\bracket{u}$ to necklace $\bracket{v}$ is the number of states 
$\tilde{v}\in\bracket{v}$ for which $\E_{d,N-d}$ has an edge from $u$ to $\tilde{v}$. By cyclic symmetry, this number does not depend on the chosen representative from $\bracket{u}$. In other words, $\bracket{u}$ has an outgoing edge for each cyclic occurrence of the consecutive substring $01$ in the necklace, and the edge points to the necklace obtained from $\bracket{u}$ by replacing this occurrence with $10$.

The graph $\E_{3,L}$ and the multigraph $\N_{3,L}$ are shown in Figure~\ref{fig:32} for $L=2$, and in Figure~\ref{fig:33} for $L=3$.

The {\em symmetric simple exclusion process} (SSEP) is defined similarly, by allowing particles to jump in both directions (clockwise and counterclockwise), provided the next site is empty~\cite{spitzer_interaction_1970,liggett_stochastic_1999,lacoin_simple_2017}. In this case, sequences of valid jumps can be interpreted as walks in the graph $\E_{d,N-d}$ where edges can be taken in the forward or backward direction, i.e., oscillating walks.

\section{Connecting all three}\label{sec:connections}

In this section we make explicit the connection between walks in simplices, standard cylindric tableaux, and exclusion processes.
The idea behind the correspondence is that inserting a cell in row $i$ of a shape in $\CS_{d,L}$ translates to moving along a step $s_i$ in $\D_{d,L}$, and to the $i$th particle (with the appropriate indexing) of a state in $\E_{d,L}$ or $\N_{d,L}$ jumping to the next site in counterclockwise direction. Next we make this precise.

Different kinds of tableaux have appeared in the literature in connection to the TASEP~\cite{corteel_multiline_2020,mandelshtam_toric_2020}, where they are used to encode states. However, unlike in these previous appearances, our standard cylindric tableaux encode sequences of states (i.e., walks in the underlying directed graph) instead.

\subsection{Covering maps}

We start with some definitions from graph theory that will be useful when relating these three notions. 

\begin{definition}
Let $G$ be a directed multigraph with vertex set $V(G)$. 
For each vertex $u\in V(G)$, let $N_G^+(u)$ (respectively, $N_G^-(u)$) be the multiset of vertices where the multiplicity of each $v\in V(G)$ equals the number of edges from $u$ to $v$ (respectively, from $v$ to $u$) in~$G$. We call the multisets $N_G^+(u)$ and $N_G^-(u)$ the {\em out-neighbors} and {\em in-neighbors} of $u$, respectively.
\end{definition}

\begin{definition}
Let $G$ and $H$ be directed multigraphs. A {\em (bidirectional) covering map} from $G$ to $H$ is a surjective map $f:V(G)\to V(H)$ such that, for every $u\in V(G)$, $\psi$ restricts to a bijection between the multisets $N_G^+(u)$ and $N_G^+(f(u))$, and to a bijection  between the multisets $N_G^-(u)$ and $N_G^-(f(u))$.
\end{definition}

We will omit the word {\em bidirectional} for the sake of brevity.

\begin{example}\label{ex:EN}
Let $q$ be the quotient map that takes each vertex $u$ of $\E_{d,L}$ to its equivalence class $\bracket{u}$, as described in Section~\ref{sec:TASEP}. Let us check that $q$ is a covering map from $\E_{d,L}$ to $\N_{d,L}$.
Clearly, the map $q:V(\E_{d,L})\to V(\N_{d,L})$ is surjective. 
In addition, for any $u,v\in V(\E_{d,L})$, the number of edges from $\bracket{u}$ to $\bracket{v}$ in $\N_{d,L}$ is the number of edges in $\E_{d,L}$ from $u$ to vertices in the equivalence class $\bracket{v}$. Thus, $q$ restricts to a bijection between the multiset of out-neighbors of $u$ in $\E_{d,L}$ and the multiset of out-neighbors of $\bracket{u}$ in $\N_{d,L}$.
Similarly, $q$ restricts to a bijection between the multiset of in-neighbors of $u$ in $\E_{d,L}$ and the multiset of in-neighbors of $\bracket{u}$ in $\N_{d,L}$.
\end{example}

Recall the infinite graph $\CSg_{d,L}$ whose vertices are cylindric shapes in $\CS_{d,L}$, and the graph $\D_{d,L}$ whose vertices are points in the simplicial region $\Delta_{d,L}$.

\begin{lemma}\label{lem:PD}
The map $f:\CS_{d,L}\to\Delta_{d,L}$ defined by
$$f(\alpha)=(\alpha_0-\alpha_1,\alpha_1-\alpha_2,\dots,\alpha_{d-1}-\alpha_{d})$$
is a covering map from $\CSg_{d,L}$ to $\D_{d,L}$.
\end{lemma}

\begin{proof}
Let $\alpha\in\CS_{d,L}$ and $\x=f(\alpha)$. Then $\x\in\Delta_{d,L}$, since all its entries are nonnegative and their sum equals $\alpha_0-\alpha_{d}=L$. Clearly, $f$ is surjective.
For $i\in[d]$, a cell can be added to row $i$ of $\alpha$ if and only if $x_{i}=\alpha_{i-1}-\alpha_{i}>0$, which is the same condition that guarantees $\x+s_i\in\Delta_{d,L}$.  In addition, if $\beta$ is the shape obtained by adding this cell, then $f(\beta)=\x+s_i$. 
Thus, $f$ restricts to a bijection between the out-neighbors of $\alpha$ in $\CSg_{d,L}$ and the out-neighbors of $\x$ in $\D_{d,L}$.

Similarly, a cell can be removed from row $i$ of $\alpha$ if and only if $x_{i+1}=\alpha_{i}-\alpha_{i+1}>0$, which is the condition for $\x-s_i\in\Delta_{d,L}$. 
If $\beta$ is the shape obtained by removing this cell, then $f(\beta)=\x-s_i$. 
Thus, $f$ restricts to a bijection between the in-neighbors of~$\alpha$ and the in-neighbors of~$\x$.
\end{proof}

\begin{lemma}\label{lem:DN}
The map $g:\Delta_{d,L}\to V(\N_{d,L})$ defined on $\x=(x_1,x_2,\dots,x_d)\in\Delta_{d,L}$ by 
$$g(\x)=\bracket{0^{x_1}10^{x_2}1\dots 0^{x_d}1}$$
is a covering map from $\D_{d,L}$ to $\N_{d,L}$.
\end{lemma}

\begin{proof}
The necklace $g(\x)$ is a vertex of $\N_{d,L}$, since it contains $d$ ones and $L$ zeros. Clearly, the map on the vertices is surjective. 
Each edge leaving $\x$ in $\D_{d,L}$ corresponds to an index $i\in[d]$ such that $x_i>0$, which we can associate to the occurrence of the consecutive substring $01$ at the end of $0^{x_i}1$ in $g(\x)$.
The necklace $g(\x+s_i)$ is obtained from $g(\x)$ by replacing this occurrence of $01$ with $10$.
Thus, $g$ restricts to a bijection between the out-neighbors of $\x$ in $\D_{d,L}$ and the out-neighbors of $g(\x)$ in~$\N_{d,L}$.

Similarly, each edge entering $\x$ in $\D_{d,L}$ corresponds to an index $i\in[d]$ such that $x_{i+1}>0$ (with indices modulo $d$),
which we can associate to the (cyclic) occurrence of the consecutive substring $10$ at the beginning of $10^{x_{i+1}}$ in $g(\x)$.
The necklace $g(\x-s_i)$ is obtained from $g(\x)$ by replacing this occurrence of $10$ with $01$.
This gives a bijection between the in-neighbors of $\x$ and the in-neighbors of~$g(\x)$.
\end{proof}

\begin{lemma}\label{lem:PE}
Let $\rho$ denote the rotation operation on finite binary strings which moves the first entry of the string to the end.
The map $h:\CS_{d,L}\to V(\E_{d,L})$ defined by
$$h(\alpha)=\rho^{\alpha_d}(0^{\alpha_0-\alpha_1}10^{\alpha_1-\alpha_2}\dots0^{\alpha_{d-1}-\alpha_{d}}1)$$
is a covering map from $\CSg_{d,L}$ to $\E_{d,L}$.
\end{lemma}

\begin{proof}
For any $\alpha\in\CS_{d,L}$, the binary string $h(\alpha)$ has $d$ ones and $\alpha_0-\alpha_d=L$ zeros, so it is a vertex of $\E_{d,L}$. Additionally, every vertex of $\E_{d,L}$ is of the form $h(\alpha)$ for some $\alpha\in\CS_{d,L}$, so the map is surjective.
A cell can be added to row $i$ of $\alpha$ if and only if $\alpha_{i-1}-\alpha_{i}>0$, in which case the resulting shape $\beta$ satisfies 
$\beta_i=\alpha_i+1$ and $\beta_j=\alpha_j$ for $j\in[d]\setminus\{i\}$.
If $i\in[d-1]$, then 
$$h(\beta)=\rho^{\alpha_d}(0^{\alpha_0-\alpha_1}10^{\alpha_1-\alpha_2}\dots0^{\alpha_{i-1}-\alpha_i-1}10^{\alpha_{i}+1-\alpha_{i+1}}1\dots
0^{\alpha_{d-1}-\alpha_{d}}1),$$
and if $i=d$, then 
$$h(\beta)=\rho^{\alpha_d+1}(0^{\alpha_0+1-\alpha_1}10^{\alpha_1-\alpha_2}\dots0^{\alpha_{d-1}-\alpha_{d}-1}1)=\rho^{\alpha_d}(0^{\alpha_0-\alpha_1}10^{\alpha_1-\alpha_2}\dots0^{\alpha_{d-1}-\alpha_{d}-1}10).$$
In both cases, $h(\beta)$ is obtained from $h(\alpha)$ by replacing a cyclic occurrence of $01$ with $10$. Thus, $h$ restricts to a bijection between the out-neighbors of $\alpha$ and the out-neighbors of $h(\alpha)$.

A similar argument shows that $h$ also restricts to a bijection between the in-neighbors of $\alpha$ and the in-neighbors of $h(\alpha)$.
\end{proof}

Let us give two examples of the map $h$ from Lemma~\ref{lem:PE}. For $\alpha=[3,1,0]\in\CS_{3,4}$, we have $h(\alpha)=0100101$. For $\alpha=[5,5,3]\in\CS_{3,4}$, we have $h(\alpha)=\rho^{3}(0011001)=1001001$.  A diagram of the four covering maps defined above is shown in Figure~\ref{fig:diagram}. 
It is easy to see that $g\circ f=q\circ h$, since $$g(f(\alpha))=\bracket{0^{\alpha_0-\alpha_1}10^{\alpha_1-\alpha_2}1\dots 0^{\alpha_{d-1}-\alpha_d}1}=\bracket{\rho^{\alpha_d}(0^{\alpha_0-\alpha_1}10^{\alpha_1-\alpha_2}\dots0^{\alpha_{d-1}-\alpha_{d}}1)}=q(h(\alpha))$$
for all $\alpha\in\CS_{d,L}$. This is the same necklace that is obtained from the infinite periodic sequence $\bdry_\alpha$, defined in Equation~\eqref{eq:bdry}, by replacing $\w$ steps with $0$s and $\s$ steps with $1$s.

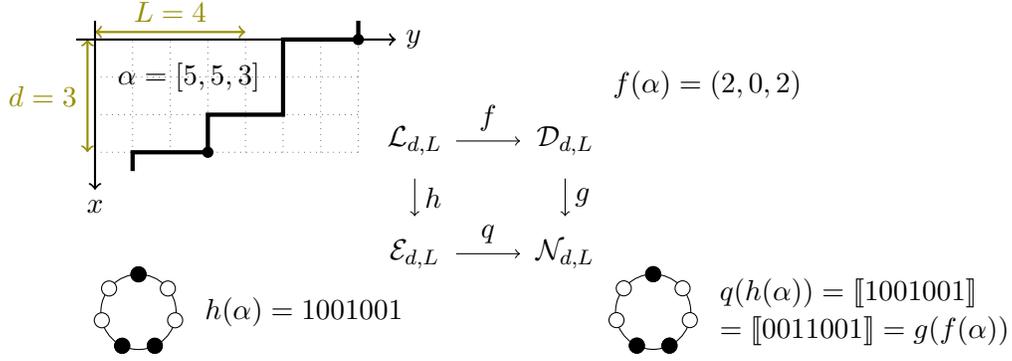
\begin{figure}[htb]
\centering
\begin{tikzpicture}[scale=0.5]

 \draw[very thin,dotted] (0,0)  grid (7,3);
 \draw[ultra thick] (1,-.5)--(1,0) \east\east \north\east\east\north\north\east\east -- ++(0,.5);
 \draw[thick,<->,olive] (-.2,0)-- node[left] {$d=3$} (-.2,3);
 \draw[thick,<->,olive] (0,3.2)-- node[above] {$L=4$} (4,3.2); 
 \draw (2.5,2) node {$\alpha=[5,5,3]$};
 \draw[thick,->] (0,3.5)--(0,-1); \draw (0,-1) node[below] {$x$};
\draw[thick,->] (-.5,3)--(8,3); \draw (8,3) node[right] {$y$};
\fill (3,0) circle (.15);
\fill (7,3) circle (.15);

\begin{scope}[shift={(-.5,.3)}]
\draw(9,0) node{$\CSg_{d,L}$};
\draw(13,0) node{$\D_{d,L}$};
\draw(9,-3) node{$\E_{d,L}$};
\draw(13,-3) node{$\N_{d,L}$};
\draw[->] (10.1,0)--node[above]{$f$} (11.8,0); 
\draw[->] (9,-1)--node[right]{$h$} (9,-2); 
\draw[->] (13,-1)--node[right]{$g$} (13,-2); 
\draw[->] (10.1,-3)--node[above]{$q$} (11.8,-3); 

\draw (14,1.5) node[right]{$f(\alpha)=(2,0,2)$} ;

\state{290}{4.84}{7}{{2,5,6}}{black}
\draw (290:4.84) ++(1.5,0) node[right]{$h(\alpha)=1001001$};

\state{343.5}{16}{7}{{2,5,6}}{black}
\draw (343.5:16) ++(1.5,0.5) node[right]{$q(h(\alpha))=\bracket{1001001}$};
\draw (343.5:16) ++(1.5,-0.5) node[right]{$=\bracket{0011001}=g(f(\alpha))$};
\end{scope}
\end{tikzpicture}
\caption{A diagram and an example of the covering maps $f,g,h,q$.}
\label{fig:diagram}
\end{figure}

In each of Figures~\ref{fig:32} and~\ref{fig:33}, the covering maps $q$ and $g$ are illustrated by the colors of the vertices: each vertex of $\N_{d,L}$ has the same color as each of its preimages.

Recall that an oscillating walk in a directed multigraph $G$ can take edges in the forward or backward direction, and that the type of an $n$-step oscillating walk is the word in $\{+,-\}^n$ that records the direction of the steps.

\begin{lemma}\label{lem:covering}
Let $G$ and $H$ be directed multigraphs for which there exists a covering map $\psi$ from $G$ to $H$. Then, for any $u\in V(G)$ and any $n\ge0$, 
$\psi$ induces a type-preserving bijection $\tilde{\psi}$ between $n$-step oscillating walks in $G$ starting at $u$ and $n$-step oscillating walks in $H$ starting at $\psi(u)$.
In addition, if the vertices of an oscillating walk in $G$ are $u,u_1,\dots,u_n$, then the vertices of the corresponding oscillating walk in $H$ are $\psi(u),\psi(u_1),\dots,\psi(u_n)$.
\end{lemma}

\begin{proof}
We will prove this by induction on $n$. The result trivially holds for $n=0$.
Now let $n\ge1$ and assume the $\psi$ induces a type-preserving bijection $\tilde{\psi}$ for $(n-1)$-step oscillating walks starting at any given vertex.

An $n$-step oscillating walk in $G$ starting at $u$ is determined the next vertex $v$, which can be any vertex in the multisets $N_G^+(u)$ or $N_G^-(u)$ depending on whether the first edges is in the forward or backward direction, together with an $(n-1)$-step oscillating walk starting at~$v$.

To construct the bijection for $n$-step oscillating walks that start with a forward edge, say from $u$ to $v$, we use the fact that $\psi$ restricts to a bijection between $N_G^+(u)$ and $N_H^+(\psi(u))$, and then apply the bijection between $(n-1)$-step oscillating walks in $G$ starting at $v$ and $(n-1)$-step oscillating walks in $H$ starting at $\psi(v)$.

Similarly, we can construct a bijection for oscillating walks that start with a backward edge by using the 
fact that $\psi$ restricts to bijection between $N_G^-(u)$ and $N_H^-(\psi(u))$.
\end{proof}

\subsection{Bijections between walks}\label{sec:bijections}

With the above setup, we can easily describe bijections between the three types of combinatorial objects from Section~\ref{sec:background}. We focus on the case of walks with forward steps, but the arguments for oscillating walks are very similar.

\begin{theorem}\label{thm:equivalence}
Fix $d,L\ge1$. Let $\alpha\in\CS_{d,L}$, let $\x=(x_1,x_2,\dots,x_d)\in\Delta_{d,L}$ where $x_i=\alpha_{i-1}-\alpha_{i}$ for $i\in[d]$, and let $u=0^{x_1}10^{x_2}1\dots 0^{x_d}1$. 
There are explicit bijections between the following sets:
\begin{enumerate}[label={(\alph*)}]
\item\label{it:a} The set $\SCT_{d,L}^n(\cdot/\alpha)$ of standard cylindric tableaux of period $(d,L)$ with $n$ cells and inner shape~$\alpha$.
\item\label{it:b} The set $\W^n_{d,L}(\x)$ of $n$-step walks in $\D_{d,L}$ starting at vertex $\x$.
\item\label{it:c} The set of $n$-step walks in $\E_{d,L}$ starting at state $u$.
\item\label{it:d} The set of $n$-step walks in $\N_{d,L}$ starting at state $\bracket{u}$.
\end{enumerate}
\end{theorem}

\newcommand{\arrowsl}[4]{
\draw[->,purple,ultra thick,shorten <=1mm,shorten >=1mm] #1--node[#4]{#3} #2;}

\newcommand{\arrowl}[4]{
\draw[->,purple,very thick,shorten <=.55cm,shorten >=.55cm] #1--node[#4]{#3} #2;}

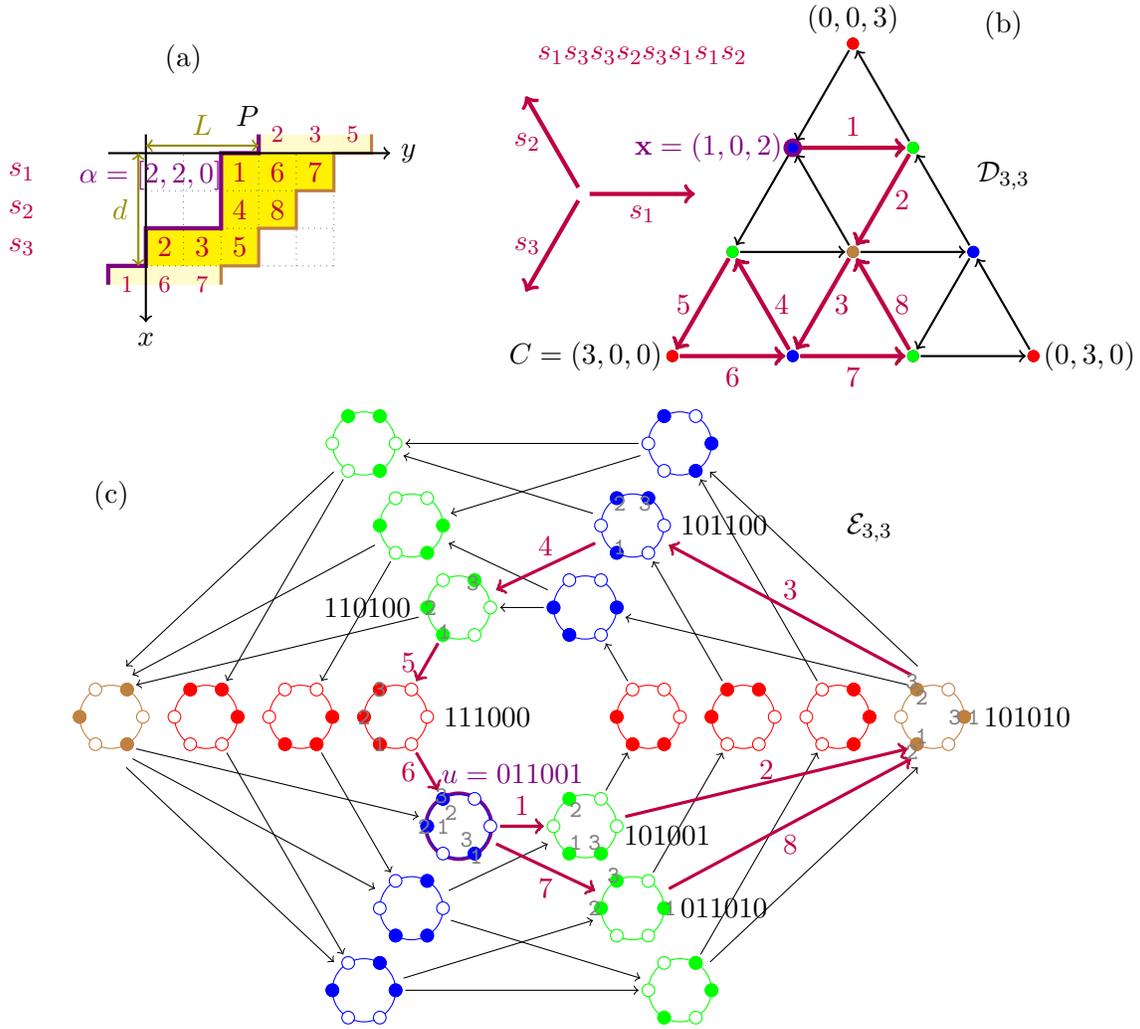
\begin{figure}
\centering
\begin{tikzpicture}
\begin{scope}[scale=0.5]
 \fill[yellow!25] (-1,-.5) rectangle (2,0);
 \fill[yellow!25] (3,3) rectangle (6,3.5);
 \fill[yellow] (0,0) \north\east\east\north\north\east \east\east \south\west\south\west\south \west\west\west;
 \draw[very thin,dotted] (0,0)  grid (5,3);
 \draw[violet, ultra thick] (-1,-.5)--++(0,.5)\east\north\east\east\north\north\east -- ++(0,.5);
 \draw[brown, very thick] (6,3.5)-- ++(0,-.5)\west\south\west\south\west\south\west -- ++(0,-0.5);
 \draw[thick,->] (0,3.5)--(0,-1.5) node[below] {$x$};
\draw[thick,->] (-1,3)--(6.5,3) node[right] {$y$};
 \draw[thick,<->,olive] (-.2,0)-- node[left] {$d$} (-.2,3);
 \draw[thick,<->,olive] (0,3.2)-- node[above] {$L$} (3,3.2); 
 \draw[purple] (2.5,2.5) node {1}; \draw[purple]  (2.5-3,2.5-3) node {\footnotesize1};
 \draw[purple]  (.5,.5) node {2}; \draw[purple]  (.5+3,.5+3) node {\footnotesize2};
 \draw[purple]  (1.5,.5) node {3}; \draw[purple]  (1.5+3,.5+3) node {\footnotesize3};
 \draw[purple]  (2.5,1.5) node {4}; 
 \draw[purple]  (2.5,.5) node {5}; \draw[purple]  (2.5+3,.5+3) node {\footnotesize5};
 \draw[purple]  (3.5,2.5) node {6}; \draw[purple]  (3.5-3,2.5-3) node {\footnotesize6};
 \draw[purple]  (4.5,2.5) node {7}; \draw[purple]  (4.5-3,2.5-3) node {\footnotesize7};
 \draw[purple]  (3.5,1.5) node {8}; 
 \draw[violet] (2.29,2.4) node[left] {$\alpha=[2,2,0]$};
\draw (2.7,4) node {$P$};
\draw[purple] (-3.3,2.5) node {$s_1$};
\draw[purple] (-3.3,1.5) node {$s_2$};
\draw[purple] (-3.3,.5) node {$s_3$};
\draw (1,5.5) node {\ref{it:a}};
\end{scope}

\begin{scope}[shift={(7,-1.2)},scale=.8]
\coordinate (a) at (2,0);
\coordinate (b) at (1,1.732);

\foreach \x in {0,...,3}{
	\foreach \y in {\x,...,3}{
		\fill ($\x*(a)+3*(b)-\y*(b)$) circle (.1);
}}

\fill[red] (0,0) circle (.1);
\fill[red] ($3*(a)$) circle (.1);
\fill[red] ($3*(b)$) circle (.1);
\fill[blue] (a) circle (.1);
\fill[blue] ($2*(a)+(b)$) circle (.1);
\fill[blue] ($2*(b)$) circle (.1);
\fill[green] (b) circle (.1);
\fill[green] ($2*(b)+(a)$) circle (.1);
\fill[green] ($2*(a)$) circle (.1);
\fill[brown] ($(b)+(a)$) circle (.1);

\foreach \x in {0,...,2}{
	\foreach \y in {\x,...,2}{
		\arrows{($\x*(a)+3*(b)-\y*(b)$)}{($\x*(a)+2*(b)-\y*(b)$)}
		\arrows{($\x*(b)+2*(a)-\y*(a)$)}{($\x*(b)+3*(a)-\y*(a)$)}
		\arrows{($\x*(b)+3*(a)-\y*(a)$)}{($\x*(b)+(b)+2*(a)-\y*(a)$)}
}}
\arrowsl{($2*(b)$)}{($(a)+2*(b)$)}{$1$}{above}
\arrowsl{($(a)+2*(b)$)}{($(a)+(b)$)}{$2$}{right}
\arrowsl{($(a)+(b)$)}{($(a)$)}{$3$}{right}
\arrowsl{($(a)$)}{($(b)$)}{$4$}{right}
\arrowsl{($(b)$)}{(0,0)}{$5$}{left}
\arrowsl{(0,0)}{($(a)$)}{$6$}{below}
\arrowsl{($(a)$)}{($2*(a)$)}{$7$}{below}
\arrowsl{($2*(a)$)}{($(a)+(b)$)}{$8$}{right}

\draw (0,0) node[left] {$C=(3,0,0)$};
\draw ($3*(a)$) node[right] {$(0,3,0)$};
\draw ($3*(b)$) node[above] {$(0,0,3)$};
\draw[violet] ($2*(b)$) node[left] {$\x=(1,0,2)$};
\draw[violet,ultra thick] ($2*(b)$) circle (.12);

\arrowsl{($3*(b)-2*(a)+(-.5,-2.5)$)}{($3*(b)-(a)+(-.5,-2.5)$)}{$s_1$}{below}
\arrowsl{($3*(b)-2*(a)+(-.5,-2.5)$)}{($4*(b)-3*(a)+(-.5,-2.5)$)}{$s_2$}{left}
\arrowsl{($3*(b)-2*(a)+(-.5,-2.5)$)}{($2*(b)-2*(a)+(-.5,-2.5)$)}{$s_3$}{left}

\draw[purple] (-.5,5) node {$s_1s_3s_3s_2s_3s_1s_1s_2$};
\draw (5.5,5.5) node {\ref{it:b}};
\draw (5.5,3) node {$\D_{3,3}$};
\end{scope}

\begin{scope}[shift={(5,-6)},scale=0.42]
\draw[violet,very thick] (240:\dB) circle (1.05);
\state{180}{\dD}{6}{{2,4,6}}{brown}
\state{180}{\dAA}{6}{{1,2,3}}{red}
\state{180}{\dA}{6}{{3,4,5}}{red}
\state{180}{\dB}{6}{{1,5,6}}{red}
\state{180+60}{\dAA}{6}{{2,3,6}}{blue} 
\state{180+60}{\dA}{6}{{2,4,5}}{blue}
\state{180+60}{\dB}{6}{{1,4,6}}{blue}
\state{180+2*60}{\dAA}{6}{{2,3,5}}{green}
\state{180+2*60}{\dA}{6}{{1,3,6}}{green}
\state{180+2*60}{\dB}{6}{{1,4,5}}{green}
\state{180+3*60}{\dD}{6}{{1,3,5}}{brown}
\state{180+3*60}{\dAA}{6}{{2,3,4}}{red}
\state{180+3*60}{\dA}{6}{{1,2,6}}{red}
\state{180+3*60}{\dB}{6}{{4,5,6}}{red}
\state{180+4*60}{\dAA}{6}{{1,3,4}}{blue}
\state{180+4*60}{\dA}{6}{{1,2,5}}{blue}
\state{180+4*60}{\dB}{6}{{3,5,6}}{blue}
\state{180+5*60}{\dAA}{6}{{1,2,4}}{green}
\state{180+5*60}{\dA}{6}{{3,4,6}}{green}
\state{180+5*60}{\dB}{6}{{2,5,6}}{green}
\foreach \i in {0,2,3,5}{
	\coordinate (aa1) at (180+\i*60:\dAA);
	\coordinate (a1) at (180+\i*60:\dA);
	\coordinate (b1) at (180+\i*60:\dB);
	\coordinate (aa2) at (180+60+\i*60:\dAA);
	\coordinate (a2) at (180+60+\i*60:\dA);
	\coordinate (b2) at (180+60+\i*60:\dB);
	\arrow{(aa1)}{(aa2)}
	\arrow{(a1)}{(a2)}
	\arrow{(b1)}{(b2)}
}
\foreach \i in {1,4}{
	\coordinate (aa1) at (180+\i*60:\dAA);
	\coordinate (a1) at (180+\i*60:\dA);
	\coordinate (b1) at (180+\i*60:\dB);
	\coordinate (aa2) at (180+60+\i*60:\dAA);
	\coordinate (a2) at (180+60+\i*60:\dA);
	\coordinate (b2) at (180+60+\i*60:\dB);
	\arrow{(aa1)}{(aa2)}
	\arrow{(aa1)}{(a2)}
	\arrow{(a1)}{(aa2)}
	\arrow{(a1)}{(b2)}
	\arrow{(b1)}{(a2)}
	\arrow{(b1)}{(b2)}
}
\foreach \i in {2,5}{
	\coordinate (aa1) at (180+\i*60:\dAA);
	\coordinate (a1) at (180+\i*60:\dA);
	\coordinate (b1) at (180+\i*60:\dB);
	\coordinate (d) at (180+60+\i*60-3:\dD+.5);
	\arrow{(aa1)}{(d)}
	\arrow{(a1)}{(d)}
	\arrow{(b1)}{(d)}
}
\foreach \i in {1,4}{
	\coordinate (aa1) at (180+\i*60:\dAA);
	\coordinate (a1) at (180+\i*60:\dA);
	\coordinate (b1) at (180+\i*60:\dB);
	\coordinate (d) at (180-60+\i*60+3:\dD+.5);
	\arrow{(d)}{(aa1)}
	\arrow{(d)}{(a1)}
	\arrow{(d)}{(b1)}
}

\arrowl{(240:\dB)}{(300:\dB)}{$1$}{above}
\arrowl{(300:\dB)}{(0-3:\dD+.5)}{$2$}{above=-1mm}
\arrowl{(0+3:\dD+.5)}{(60:\dA)}{$3$}{above}
\arrowl{(60:\dA)}{(120:\dB)}{$4$}{above}
\arrowl{(120:\dB)}{(180:\dB)}{$5$}{left}
\arrowl{(180:\dB)}{(240:\dB)}{$6$}{left}
\arrowl{(240:\dB)}{(300:\dA)}{$7$}{below}
\arrowl{(300:\dA)}{(0-3:\dD+.5)}{$8$}{below}

\draw[violet] (180+60:\dB)++(1.7,1.65) node {$u=011001$};
\draw[gray] (180+60:\dB)++(180:.5) node {\footnotesize $\mathtt{1}$};
\draw[gray] (180+60:\dB)++(120:.5) node {\footnotesize $\mathtt{2}$};
\draw[gray] (180+60:\dB)++(-60:.5) node {\footnotesize $\mathtt{3}$};
\draw (180+2*60:\dB)++(.9,-.3) node[right] {$101001$};
\draw[gray] (180+2*60:\dB)++(240:.6) node {\footnotesize $\mathtt{1}$};
\draw[gray] (180+2*60:\dB)++(120:.6) node {\footnotesize $\mathtt{2}$};
\draw[gray] (180+2*60:\dB)++(-60:.6) node {\footnotesize $\mathtt{3}$};
\draw (0:\dD)++(1.3,0) node[right] {$101010$};
\draw[gray] (0:\dD)++(240:.7) node {\footnotesize $\mathtt{1}$};
\draw[gray] (0:\dD)++(120:.7) node {\footnotesize $\mathtt{2}$};
\draw[gray] (0:\dD)++(0:.7) node {\footnotesize $\mathtt{3}$};
\draw (60:\dA)++(1.2,0) node[right] {$101100$};
\draw[gray] (60:\dA)++(240:.8) node {\footnotesize $\mathtt{1}$};
\draw[gray] (60:\dA)++(120:.8) node {\footnotesize $\mathtt{2}$};
\draw[gray] (60:\dA)++(60:.8) node {\footnotesize $\mathtt{3}$};
\draw (120:\dB)++(-1.2,0) node[left] {$110100$};
\draw[gray] (120:\dB)++(240:.9) node {\footnotesize $\mathtt{1}$};
\draw[gray] (120:\dB)++(180:.9) node {\footnotesize $\mathtt{2}$};
\draw[gray] (120:\dB)++(60:.9) node {\footnotesize $\mathtt{3}$};
\draw (180:\dB)++(1.2,0) node[right] {$111000$};
\draw[gray] (180:\dB)++(240:1.0) node {\footnotesize $\mathtt{1}$};
\draw[gray] (180:\dB)++(180:1.0) node {\footnotesize $\mathtt{2}$};
\draw[gray] (180:\dB)++(120:1.0) node {\footnotesize $\mathtt{3}$};
\draw[gray] (240:\dB)++(300:1.1) node {\footnotesize $\mathtt{1}$};
\draw[gray] (240:\dB)++(180:1.1) node {\footnotesize $\mathtt{2}$};
\draw[gray] (240:\dB)++(120:1.1) node {\footnotesize $\mathtt{3}$};
\draw (300:\dA)++(1.2,0) node[right] {$011010$};
\draw[gray] (300:\dA)++(0:1.2) node {\footnotesize $\mathtt{1}$};
\draw[gray] (300:\dA)++(180:1.2) node {\footnotesize $\mathtt{2}$};
\draw[gray] (300:\dA)++(120:1.2) node {\footnotesize $\mathtt{3}$};
\draw[gray] (0:\dD)++(0:1.3) node {\footnotesize $\mathtt{1}$};
\draw[gray] (0:\dD)++(240:1.3) node {\footnotesize $\mathtt{2}$};
\draw[gray] (0:\dD)++(120:1.3) node {\footnotesize $\mathtt{3}$};

\draw (-13,7) node {\ref{it:c}};
\draw (11,6) node {$\E_{3,3}$};
\end{scope}
\end{tikzpicture}
\caption{A standard cylindric tableau with inner shape $\alpha=[2,2,0]\in\CS_{3,3}$ (top left), the corresponding walk in $\D_{3,3}$ starting at $\x=(1,0,2)$ (top right), and the corresponding walk in $\E_{3,3}$ starting at $u=011001$ (bottom), using the bijections described in Theorem~\ref{thm:equivalence}.}
\label{fig:33walk}
\end{figure}

\begin{proof}
To describe a bijection between the sets \ref{it:a} and \ref{it:b}, recall the covering map $f$ from $\CSg_{d,L}$ to $\D_{d,L}$ given in Lemma~\ref{lem:PD}, and note that $f(\alpha)=\x$. 
By the map $\ssh$ from Equation~\eqref{eq:ssh}, standard cylindric tableaux of period $(d,L)$ with inner shape $\alpha$ can be interpreted as walks in $\CSg_{d,L}$ starting at $\alpha$. Thus, by Lemma~\ref{lem:covering} in the case of walks with only forward steps, $f$ induces a bijection $\tilde{f}$ from \ref{it:a} to \ref{it:b}, which can be described as follows. 
Given a tableau $T\in\SCT_{d,L}^n(\cdot/\alpha)$, let $i_k\in [d]$ be the row that contains entry $k$, for each $k\in[n]$.
Then $\tilde{f}(T)$ is the 
$n$-step walk  that starts at $\x$ and has steps $s_{i_1}s_{i_2}\dots s_{i_n}$. See Figure~\ref{fig:33walk} for an example of this bijection.
The condition that $T$ has increasing rows and columns guarantees that the walk  stays in the region $\Delta_{d,L}$.

A bijection $\tilde{g}$ between the sets \ref{it:b} and \ref{it:d} arises from the covering map $g$ from $\D_{d,L}$ to $\N_{d,L}$ given in Lemma~\ref{lem:DN}, using again Lemma~\ref{lem:covering} and noting that $g(\x)=\bracket{u}$. 
A bijection $\tilde{q}$ between \ref{it:c} and \ref{it:d} follows from the quotient map $q$ in Example~\ref{ex:EN}, which is a covering map from $\E_{d,L}$ to~$\N_{d,L}$.

It is also possible to give a direct description of the bijection $\tilde{q}^{-1}\circ\tilde{g}$ between the sets \ref{it:b} and \ref{it:c}.
First, index the $d$ particles in state $u$ so that, for each $i\in[d]$, particle $i$ occupies the site corresponding to the $i$th $1$ from the left in the string $0^{x_1}10^{x_2}1\dots 0^{x_d}1$.
Now, given a walk $s_{i_1}s_{i_2}\dots s_{i_n}$ in $\D_{d,L}$ starting at $\x$, we associate to it the walk in $\E_{d,L}$ starting at $u$ that is obtained by successive jumps of the particles $i_1,i_2,\dots,i_n$ to the next site in counterclockwise direction. See Figure~\ref{fig:33walk} for an example.

Similarly, we can directly describe a bijection between the sets \ref{it:a} and \ref{it:c} as follows. Given $T\in\SCT_{d,L}^n(\cdot/\alpha)$, again let $i_k\in [d]$ be the row that contains entry $k$, for each $k\in[n]$. Indexing the particles in state $u$ so that particle $i$ occupies the site corresponding to the $i$th $1$ from the left, the resulting walk in $\E_{d,L}$ starting at $u$ consists of successive jumps of the particles $i_1,i_2,\dots,i_n$ in counterclockwise direction.
In the case $\alpha_d=0$, we have $h(\alpha)=u$, where $h$ is the covering map from Lemma~\ref{lem:PE}, and the walk just described is 
$\tilde{h}(T)$, where $\tilde{h}$ is the bijection given by Lemma~\ref{lem:covering}. For arbitrary $\alpha_d$, the walk $\tilde{h}(T)$ starts at $h(\alpha)=\rho^{\alpha_d}(u)$, so we need to apply the rotation $\rho^{-\alpha_d}$ to each of the vertices of this walk in order to obtain the aforementioned walk starting at $u$.
\end{proof}

A natural involution on cylindric shapes, and by extension on standard cylindric tableaux, is obtained by conjugation. For $\lambda\in\CS_{d,L}$, define its conjugate to be the shape $\lambda'\in\CS_{L,d}$, where $$\lambda'_j=\max\{i:\lambda_i\ge j\}$$ for all $j$. The Young diagrams $Y_\lambda$ and $Y_{\lambda'}$ are reflections of each other along the diagonal $y=x$, that is, $\langle i,j \rangle \in Y_\lambda$ if and only if $\langle j,i \rangle \in Y_{\lambda'}$. The step sequences $\bdry_{\lambda}$ and $\bdry_{\lambda'}$ of their boundary lattice paths are obtained from each other by switching steps $\w$ and $\s$, and reversing.

\begin{example} The transpose of $\mu=[3,1,0]\in\CS_{3,4}$, drawn in  Figure~\ref{fig:SCT}, is $\mu'=[2,1,1,0]\in\CS_{4,3}$. The step sequence of its boundary path is $\bdry_{\mu'}=(\w\s\w\s\s\w\s)^\infty$.
\end{example}

If $T$ is a standard cylindric tableau of shape $\lambda/\mu$, define its conjugate to be the standard cylindric tableau $T'$ of shape $\lambda'/\mu'$ where $T'(\langle j,i \rangle)=T(\langle i,j \rangle)$ for all $\langle j,i \rangle\in Y_{\lambda'/\mu'}$.

For a state $u=u_1u_2\dots u_N$ of the TASEP on $\Z_N$, define its reverse-complement $u^{rc}$ to be the state where $u^{rc}_k=1-u_{N+1-k}$ 
for all $k$, obtained from $u$ by switching occupied sites with empty sites (an operation referred to as {\it particle-hole symmetry}) and reversing the orientation.

\begin{theorem}\label{thm:conjugate}
With the notation from Theorem~\ref{thm:equivalence}, let $\alpha'\in\CS_{L,d}$ be the conjugate of $\alpha$, let $\y=(y_1,y_2,\dots,y_L)\in\Delta_{L,d}$ where $y_j=\alpha'_{j-1}-\alpha'_j$ for $j\in[L]$, and let $u^{rc}$ be the reverse-complement of $u$.
There are explicit bijections between the sets in Theorem~\ref{thm:equivalence} and the following sets:
\begin{enumerate}[label={(\alph*')}]
\item\label{it:a'} The set $\SCT_{L,d}^n(\cdot/\alpha')$  of standard cylindric tableaux of period $(L,d)$ with $n$ cells and inner shape~$\alpha'$.
\item\label{it:b'} The set $\W^n_{L,d}(\y)$ of $n$-step walks in $\D_{L,d}$ starting at vertex $\y$.
\item\label{it:c'} The set of $n$-step walks in $\E_{L,d}$ starting at state $u^{rc}$.
\item\label{it:d'} The set of $n$-step walks in $\N_{L,d}$ starting at state $\bracket{u^{rc}}$.
\end{enumerate}
\end{theorem}

\begin{proof}
Letting $f:\CS_{d,L}\to\Delta_{d,L}$ be the map from Lemma~\ref{lem:PD}, we have $\y=f(\alpha')$.
Letting $g:\Delta_{d,L}\to V(\N_{d,L})$ be the map from Lemma~\ref{lem:DN},  and letting $v=0^{y_1}10^{y_2}1\dots 0^{y_L}1$, we have $g(\y)=\bracket{v}=\bracket{u^{rc}}$.
The reason for the last equality is that $\bracket{v}$ encodes the periodic sequence of steps $\bdry_{\alpha'}$ (with $0$s recording $\w$ steps and $1$s recording $\s$ steps), whereas $\bracket{u}$ encodes the sequence $\bdry_\alpha$. Thus, by Theorem~\ref{thm:equivalence}, where the roles of $\alpha$, $\x$ and $u$ are played by $\alpha'$, $\y$ and $v$, respectively, we obtain explicit bijections between the four sets \ref{it:a'}, \ref{it:b'}, \ref{it:c'} and \ref{it:d'}. 

A trivial bijection between the set \ref{it:a} from Theorem~\ref{thm:equivalence} and the set \ref{it:a'} is given by conjugation of standard cylindric tableaux. It follows that all eight sets \ref{it:a}--\ref{it:d} and \ref{it:a'}--\ref{it:d'} are in bijection with each other.

It is also possible to describe the resulting bijection between the sets \ref{it:c} and \ref{it:c'} (respectively, \ref{it:d} and \ref{it:d'}) directly: it is obtained by applying the reverse-complement map to the vertices of $\E_{d,L}$ (respectively, $\N_{d,L}$).
This map is a graph isomorphism between $\E_{d,L}$ and $\E_{L,d}$ (respectively, $\N_{d,L}$ and $\N_{L,d}$), since changing a cyclic occurrence of $01$ to $10$ in a binary string $w$ is equivalent to changing a cyclic occurrence of $01$ to $10$ in $w^{rc}$.
See Figure~\ref{fig:conjugate} for an example of these bijections.
\end{proof}

\begin{figure}
\centering
\begin{tikzpicture}

\draw (-2.25,1.75) node {\ref{it:a}};
\draw (11.25,1.75) node {\ref{it:a'}};
\draw (-2.25,-2.5) node {\ref{it:b}};
\draw (11.25,-2.5) node {\ref{it:b'}};
\draw (-2.25,-6.75) node {\ref{it:c}};
\draw (11.25,-6.75) node {\ref{it:c'}};

\begin{scope}[scale=.5]
 \fill[yellow!25] (0,-.5)--(0,0) \north\north\east\east\north -- ++(0,.5)-- ++(2,0) -- ++(0,-.5) \west\south\south\west\south\west -- ++(0,-0.5) -- ++(-1,0);
 \fill[yellow] (0,0)\north\north\east\east\north\east\south\south\west\south\west\west;
 \draw[very thin,dotted] (0,0)  grid (4,3);
 \draw[violet, ultra thick] (0,-.5)--(0,0)\north\north\east\east\north  -- ++(0,.5);
 \draw[blue, very thick] (4,3.5)--(4,3)\west\south\south\west\south\west -- ++(0,-.5);
 \draw[thick,->] (0,3.5)--(0,-1.5) node[below] {$x$};
\draw[thick,->] (-.5,3)--(4.5,3) node[right] {$y$};
 \draw[thick,<->,olive] (-.2,0)-- node[left] {$d$} (-.2,3);
 \draw[thick,<->,olive] (0,3.2)-- node[above] {$L$} (2,3.2); 
 \draw[purple] (.5,1.5) node {1}; 
 \draw[purple]  (.5,.5) node {2}; \draw[purple]  (.5+2,.5+3) node {\footnotesize2};
 \draw[purple]  (1.5,1.5) node {3}; 
 \draw[purple]  (1.5,.5) node {4};  \draw[purple]  (1.5+2,.5+3) node {\footnotesize4}; 
 \draw[purple]  (2.5,2.5) node {5};  \draw[purple]  (2.5-2,2.5-3) node {\footnotesize5}; 
 \draw[purple]  (2.5,1.5) node {6}; 
 \draw[violet] (2.29,2.5) node[left] {$\alpha=[2,0,0]$};

\draw[purple] (-3,2.5) node {$s_1$};
\draw[purple] (-3,1.5) node {$s_2$};
\draw[purple] (-3,.5) node {$s_3$};
\draw[<->] (7.6,3)--node[above]{\small conjugate} (9.4,3);
\end{scope}

\begin{scope}[shift={(7,.5)},scale=.5]
 \fill[yellow!25] (-2,-.5)--(-2,0) \east\east\east\north\north\east\east\east -- ++(0,.5)-- ++(2,0) -- ++(0,-.5) \west\south\west\west\south\west -- ++(0,-0.5) -- ++(-4,0);
 \fill[yellow] (0,0) \east\north\north\east\east\east \east\east \west\south\west\west\south\west \west;
 \draw[very thin,dotted] (-2,0)  grid (5,2);
 \draw[violet, ultra thick] (-2,-.5)--(-2,0) \east\east\east\north\north\east\east\east  -- ++(0,.5);
 \draw[blue, very thick] (6,2.5)--(6,2)\west\south\west\west\south\west  -- ++(0,-.5);
 \draw[thick,->] (0,2.5)--(0,-1.5) node[below] {$x$};
\draw[thick,->] (-2,2)--(6.5,2) node[right] {$y$};
 \draw[thick,<->,olive] (-.2,0)-- node[left] {$L$} (-.2,2);
 \draw[thick,<->,olive] (0,2.2)-- node[above] {$d$} (3,2.2); 
 \draw[purple] (1.5,1.5) node {1};  \draw[purple] (1.5-3,1.5-2) node {1}; 
 \draw[purple] (2.5,1.5) node {2};  \draw[purple] (2.5-3,1.5-2) node {2}; 
 \draw[purple] (3.5,1.5) node {5};  \draw[purple] (3.5-3,1.5-2) node {5}; 
 \draw[purple] (4.5,1.5) node {6};  \draw[purple] (4.5-3,1.5-2) node {6}; 
 \draw[purple] (1.5,.5) node {3};  \draw[purple] (1.5+3,.5+2) node {3}; 
 \draw[purple] (2.5,.5) node {4};  \draw[purple] (2.5+3,.5+2) node {4}; 
 \draw[violet] (1.29,1.5) node[left] {$\alpha'=[1,1]$};

\draw[purple] (-3,1.5) node {$s_1$};
\draw[purple] (-3,.5) node {$s_2$};
\end{scope}

\begin{scope}[shift={(-.5,-5.2)},scale=.8]
\coordinate (a) at (2,0);
\coordinate (b) at (1,1.732);

\foreach \x in {0,...,2}{
	\foreach \y in {\x,...,2}{
		\fill ($\x*(a)+2*(b)-\y*(b)$) circle (.1);
}}

\fill[red] (0,0) circle (.1);
\fill[red] ($2*(a)$) circle (.1);
\fill[red] ($2*(b)$) circle (.1);
\fill[blue] (a) circle (.1);
\fill[blue] (b) circle (.1);
\fill[blue] ($(a)+(b)$) circle (.1);

\foreach \x in {0,...,1}{
	\foreach \y in {\x,...,1}{
		\arrows{($\x*(a)+2*(b)-\y*(b)$)}{($\x*(a)+(b)-\y*(b)$)}
		\arrows{($\x*(b)+(a)-\y*(a)$)}{($\x*(b)+2*(a)-\y*(a)$)}
		\arrows{($\x*(b)+2*(a)-\y*(a)$)}{($\x*(b)+(b)+(a)-\y*(a)$)}
}}

\arrowsl{($2*(a)$)}{($(a)+(b)$)}{$1$}{right}
\arrowsl{($(a)+(b)$)}{($(a)$)}{$2$}{right}
\arrowsl{($(a)$)}{($(b)$)}{$3$}{right}
\arrowsl{($(b)$)}{(0,0)}{$4$}{left}
\arrowsl{(0,0)}{($(a)$)}{$5$}{above}
\arrowsl{($(a)+(-.1,0)$)}{($(b)+(-.1,0)$)}{$6$}{left}

\draw (0,0) node[left] {$(2,0,0)$};
\draw[violet] ($2*(a)$) node[right,xshift=1] {$\x=(0,2,0)$};
\draw[violet,ultra thick] ($2*(a)$) circle (.12);
\draw ($2*(b)$) node[above] {$(0,0,2)$};

\arrowsl{($3*(b)-2*(a)+(-.2,-2)$)}{($3*(b)-(a)+(-.2,-2)$)}{$s_1$}{below}
\arrowsl{($3*(b)-2*(a)+(-.2,-2)$)}{($4*(b)-3*(a)+(-.2,-2)$)}{$s_2$}{left}
\arrowsl{($3*(b)-2*(a)+(-.2,-2)$)}{($2*(b)-2*(a)+(-.2,-2)$)}{$s_3$}{left}

\draw[purple] (0.5,4.7) node {$s_2s_3s_2s_3s_1s_2$};
\draw (3.5,2.5) node {$\D_{3,2}$};
\end{scope}

\begin{scope}[shift={(6.5,-4.2)},scale=.8]
\fill[red] (0,0) circle (.1);
\fill[red] (6,0) circle (.1);
\fill[blue] (2,0) circle (.1);
\fill[blue] (4,0) circle (.1);

\foreach \x in {0,...,2}{
	\arrows{(2*\x,.06)}{(2*\x+2,.06)}
	\arrows{(2*\x+2,-.06)}{(2*\x,-.06)}
}

\arrowsl{(0,.06)}{(2,.06)}{$1,5$}{above}
\arrowsl{(2,.06)}{(4,.06)}{$2,6$}{above}
\arrowsl{(2,-.06)}{(0,-.06)}{$4$}{below}
\arrowsl{(4,-.06)}{(2,-.06)}{$3$}{below}

\draw (-.5,-.1) node[below] {$\y=(3,0)$};
\draw[violet,ultra thick] (0,0) circle (.12);
\draw (6,-.1) node[below] {$(0,3)$};

\arrowsl{(3,2)}{(5,2)}{$s_1$}{below}
\arrowsl{(3,2)}{(1,2)}{$s_2$}{below}

\draw[purple] (3,3) node {$s_1s_1s_2s_2s_1s_1$};
\draw (-.5,1) node {$\D_{2,3}$};
\end{scope}

\begin{scope}[shift={(1,-9.5)},scale=.42] 
\draw[violet,very thick] (90+2*360/5:\dA) circle (1.05);
\state{90}{\dA}{5}{{1,2,3}}{red}
\state{90}{\dB}{5}{{2,4,5}}{blue}
\state{90+360/5}{\dA}{5}{{1,4,5}}{red}
\state{90+360/5}{\dB}{5}{{2,3,5}}{blue}
\state{90+2*360/5}{\dA}{5}{{2,3,4}}{red}
\state{90+2*360/5}{\dB}{5}{{1,3,5}}{blue}
\state{90+3*360/5}{\dA}{5}{{1,2,5}}{red}
\state{90+3*360/5}{\dB}{5}{{1,3,4}}{blue}
\state{90+4*360/5}{\dA}{5}{{3,4,5}}{red}
\state{90+4*360/5}{\dB}{5}{{1,2,4}}{blue}
\foreach \i in {0,...,4}{
	\coordinate (a1) at (90+\i*360/5:\dA);  
	\coordinate (b1) at (90+\i*360/5:\dB);
	\coordinate (a2) at (90+360/5+\i*360/5:\dA);
	\coordinate (b2) at (90+360/5+\i*360/5:\dB);
	\arrow{(a1)}{(b2)}
	\arrow{(b1)}{(a2)}
	\arrow{(b1)}{(b2)}
}

\draw[violet] (90+2*360/5:\dA)++(-1.3,0) node[left]{$u=10011$} ;
\draw[gray] (90+2*360/5:\dA)++(90+2*360/5:.6) node {\footnotesize $\mathtt{1}$};
\draw[gray] (90+2*360/5:\dA)++(90-360/5:.6) node {\footnotesize $\mathtt{2}$};
\draw[gray] (90+2*360/5:\dA)++(90-2*360/5:.6) node {\footnotesize $\mathtt{3}$};
\arrowl{(90+2*360/5:\dA)}{(90+3*360/5:\dB)}{$1$}{below}
\arrowl{(90+3*360/5:\dB)}{(90+4*360/5:\dB)}{$2$}{left}
\arrowl{(90+4*360/5:\dB)}{(90:\dB)}{$3$}{below}
\arrowl{(90:\dB)}{(90+360/5:\dA)}{$4$}{above}
\arrowl{(90+360/5:\dA)}{(90+2*360/5:\dB)}{$5$}{left}
\arrowl{(90+2*360/5:\dB)}{(90+3*360/5:\dB)}{$6$}{above}

\draw[<->] (8,-2)--node[above]{\small reverse-complement} (10.2,-2);
\draw (3.5,5.5) node {$\E_{3,2}$};
\end{scope}

\begin{scope}[shift={(9,-9.5)},scale=.42]
\draw[violet,very thick] (90+2*360/5:\dA) circle (1.05);

\state{90}{\dA}{5}{{2,3}}{red}
\state{90}{\dB}{5}{{1,4}}{blue}
\state{90+360/5}{\dA}{5}{{5,3}}{red}
\state{90+360/5}{\dB}{5}{{1,3}}{blue}
\state{90+2*360/5}{\dA}{5}{{1,2}}{red}
\state{90+2*360/5}{\dB}{5}{{3,5}}{blue}
\state{90+3*360/5}{\dA}{5}{{3,4}}{red}
\state{90+3*360/5}{\dB}{5}{{5,2}}{blue}
\state{90+4*360/5}{\dA}{5}{{5,1}}{red}
\state{90+4*360/5}{\dB}{5}{{2,4}}{blue}
\foreach \i in {0,...,4}{
	\coordinate (a1) at (90+\i*360/5:\dA);  
	\coordinate (b1) at (90+\i*360/5:\dB);
	\coordinate (a2) at (90+360/5+\i*360/5:\dA);
	\coordinate (b2) at (90+360/5+\i*360/5:\dB);
	\arrow{(a1)}{(b2)}
	\arrow{(b1)}{(a2)}
	\arrow{(b1)}{(b2)}
}

\draw[violet] (90+2*360/5:\dA)++(-1.3,0) node[left] {$u^{rc}=00110$};
\draw[gray] (90+2*360/5:\dA)++(90:.6) node {\footnotesize $\mathtt{1}$};
\draw[gray] (90+2*360/5:\dA)++(90-360/5:.6) node {\footnotesize $\mathtt{2}$};
\arrowl{(90+2*360/5:\dA)}{(90+3*360/5:\dB)}{$1$}{below}
\arrowl{(90+3*360/5:\dB)}{(90+4*360/5:\dB)}{$2$}{left}
\arrowl{(90+4*360/5:\dB)}{(90:\dB)}{$3$}{below}
\arrowl{(90:\dB)}{(90+360/5:\dA)}{$4$}{above}
\arrowl{(90+360/5:\dA)}{(90+2*360/5:\dB)}{$5$}{left}
\arrowl{(90+2*360/5:\dB)}{(90+3*360/5:\dB)}{$6$}{above}
\draw (-3.5,5.5) node {$\E_{2,3}$};
\end{scope}
\end{tikzpicture}
\caption{A standard cylindric tableau with inner shape $\alpha=[2,0,0]\in\CS_{3,2}$ (top left) and its conjugate, having inner shape $\alpha'=[1,1]\in\CS_{2,3}$ (top right);
their corresponding walks in $\D_{3,2}$ starting at $\x=(0,2,0)$ (middle left), and in $\D_{2,3}$ starting at $\y=(3,0)$ (middle right); and their corresponding walks in $\E_{3,2}$ starting at $u=10011$ (bottom left), and in $\E_{2,3}$ starting at $u^{rc}=00110$ (bottom right), via the bijections from Theorems~\ref{thm:equivalence} and~\ref{thm:conjugate}. }
\label{fig:conjugate}
\end{figure}
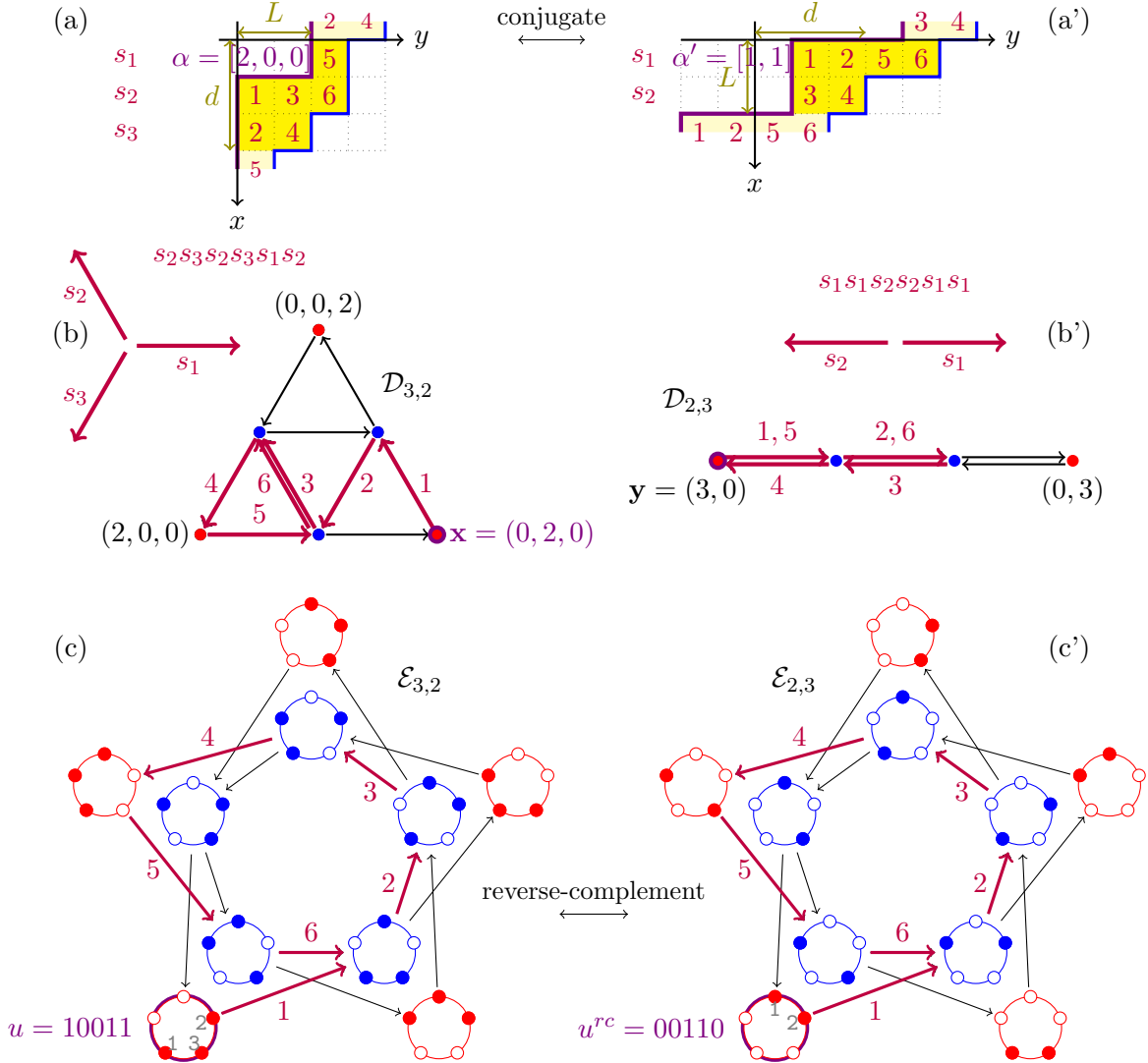 

Unlike the above bijections between \ref{it:a} and \ref{it:a'}, between \ref{it:c} and \ref{it:c'}, and between \ref{it:d} and \ref{it:d'}, the bijection between the sets \ref{it:b} and \ref{it:b'} that results from Theorem~\ref{thm:conjugate} is not as straightforward to describe directly in terms of walks in simplices. This illustrates how the correspondences between the different combinatorial objects can be useful. 

For small values of $d$, the existing enumerative results about walks in simplices (Theorems~\ref{thm:mortimer_number_2015} and~\ref{thm:d4}) translate, via Theorems~\ref{thm:equivalence} and~\ref{thm:conjugate}, into results about standard cylindric tableaux and sequences of states of the TASEP.
In the following corollary, we omit the corresponding sets \ref{it:a'}, \ref{it:c'} and \ref{it:d'} because they are in trivial bijection with \ref{it:a} (via conjugation), \ref{it:c} and \ref{it:d} (via reverse-complement), respectively.

\begin{corollary}
The cardinality of each of the following sets equals $a_{n,L}$, as given by Equation~\eqref{eq:anL}:
\begin{enumerate}[label={(\alph*)}]
\item[\ref{it:a}] The set $\SCT_{3,L}^n(\cdot/[0,0,0])$ of standard cylindric tableaux of period $(3,L)$ with $n$ cells and rectangular inner shape $[0,0,0]$.
\item[\ref{it:b}] The set $\W^n_{3,L}(C)$ of $n$-step walks in $\D_{3,L}$ starting at a given corner of the simplex.
\item[\ref{it:b'}] The set $\W^n_{L,3}(C)$ of $n$-step walks in $\D_{L,3}$ starting at a given corner of the simplex.
\item[\ref{it:c}] The set of $n$-step walks in $\E_{3,L}$ starting at state $1^3 0^{L}$.
\item[\ref{it:d}] The set of $n$-step walks in $\N_{3,L}$ starting at state $\bracket{1^3 0^{L}}$.
\end{enumerate}
\end{corollary}

The analogous result for $d=4$ and $\alpha=[0,0,0,0]$ states that the corresponding sets are enumerated by the coefficient of $t^n$ in the generating function from Theorem~\ref{thm:d4}.

The bijections in this section can easily be extended to the case of oscillating walks, where we allow forward and backward steps.  Theorem~\ref{thm:equivalence} generalizes as follows.

\begin{theorem}\label{thm:equivalence-oscillating} 
Fix $d,L\ge1$ and $n\ge0$. Let $\alpha\in\CS_{d,L}$, let $\x=(x_1,x_2,\dots,x_d)\in\Delta_{d,L}$ where $x_i=\alpha_{i-1}-\alpha_{i}$ for $i\in[d]$, let $u=0^{x_1}10^{x_2}1\dots 0^{x_d}1$, and let $w\in\{+,-\}^n$.
There are explicit bijections between the following sets:
\begin{enumerate}[label={(\Alph*)}]
\item\label{it:A} The set $\OCT_{d,L}^w(\alpha,\cdot)$ of oscillating cylindric tableaux of period $(d,L)$ and type $w$ starting at~$\alpha$.
\item\label{it:B} The set $\OW^w_{d,L}(\x)$ of oscillating walks in $\D_{d,L}$ of type $w$ starting at vertex $\x$.
\item\label{it:C} The set of oscillating walks in $\E_{d,L}$ of type $w$ starting at state $u$.
\item\label{it:D} The set of oscillating walks in $\N_{d,L}$ of type $w$ starting at state $\bracket{u}$.
\end{enumerate}
\end{theorem}

\begin{proof}
The proof is very similar to that of Theorem~\ref{thm:equivalence}.
By Lemma~\ref{lem:covering}, the covering map $f$ from $\CSg_{d,L}$ to $\D_{d,L}$ given in Lemma~\ref{lem:PD} induces a bijection $\tilde{f}$ between the sets \ref{it:A} and~\ref{it:B}, which can be described as follows. 
Given an oscillating tableau $\alpha=\rho^0,\rho^1,\dots,\rho^n$ in $\OCT_{d,L}^w(\alpha,\cdot)$, its image is the oscillating walk in $\D_{d,L}$ starting at $\x$ whose $k$th step is $s_{i_k}$ (respectively, $\bars_{i_k}$) if $\rho^k$ is obtained from $\rho^{k-1}$ by adding (respectively, removing) a cell in row $i_k$. 

A bijection $\tilde{g}$ between the sets \ref{it:B} and \ref{it:D} arises from the covering map $g$ from $\D_{d,L}$ to $\N_{d,L}$ given in Lemma~\ref{lem:DN}, and a straightforward bijection $\tilde{q}$ between \ref{it:C} and \ref{it:D} follows from the quotient map $q$ in Example~\ref{ex:EN}.
\end{proof}

\section{The cylindric RSK correspondence}
\label{sec:CRS}

Having described the basic bijections between standard cylindric tableaux, walks in simplices, and sequences of states in exclusion processes, in this section we exploit them to deduce more significant results. Specifically, we will use a cylindric analogue of the RS correspondence to obtain an unexpected proof of Theorem~\ref{thm:courtiel_bijections_2021}.

Sagan and Stanley introduced analogues of the RSK correspondence for skew tableaux; see \cite[Thm.\ 2.1 \& 5.1]{sagan_robinson-schensted_1990} for the standard case and~\cite[Thm.\ 6.11]{sagan_robinson-schensted_1990} for the semistandard case.
Their description relies on two types of row insertion operations, which were later adapted by Neyman to cylindric tableaux, in order to define a cylindric version of the RSK correspondence \cite[Thm.\ 5.2]{neyman_cylindric_2015}. 
We will describe Neyman's construction in the simpler case of standard cylindric tableaux, which are relevant to us because of their interpretation as walks in simplices.

\subsection{The insertion description}

We say that $i\in[d]$ is an {\em insertion corner} of a cylindric shape $\mu\in\CS_{d,L}$ if $\mu_{i-1}>\mu_i$. 
For example, the shape $\mu=[3,1,0]\in\CS_{3,4}$ from Figure~\ref{fig:SCT} has three insertion corners, namely $1$, $2$ and $3$.
Insertion corners of $\mu$ correspond to occurrences of $\w\s$ in the periodic sequence $\bdry_\mu$ from Equation~\eqref{eq:bdry}, or equivalently, to (cyclic) occurrences of $01$ in the necklace $\bracket{h(\mu)}$, where $h$ is the map from Lemma~\ref{lem:PE}. 
We denote the number of insertion corners of $\mu$ by
\begin{equation}\label{eq:corners}
c(\mu)=|\{i\in[d]:\mu_{i-1}>\mu_i\}|.
\end{equation}
The statistic $c(\mu)$ equals the number of blocks of consecutive $1$s in the necklace $\bracket{h(\mu)}$, or equivalently, the number of blocks of consecutive $0$s.

If $T\in\SCT_{d,L}(\lambda/\mu)$ and $i$ is an insertion corner of $\mu$, the cell $\langle i,\mu_i+1 \rangle$ is called an {\em inner corner} of~$T$. 
For example, the inner corners of the tableau in Figure~\ref{fig:SCT} are $\langle 1,4\rangle$, $\langle 2,2\rangle$, and $\langle 3,1\rangle$.

\begin{definition}[Internal row insertion]\label{def:insertion}
Let $T\in\SCT^{n}_{d,L}(\lambda/\mu)$, and let $\langle i,j\rangle$ be an inner corner of~$T$; that is, $\mu_{i-1}>\mu_i$ and $j=\mu_i+1$.
To perform {\em internal insertion in $T$ at row $i$}, suppose first that cell $\langle i,j \rangle$ contains an element $a$, and remove this cell from $T$ (adding $1$ to $\mu_i$ to change the inner shape). 

Next, insert element $a$ into row $i+1$ as follows. If $a$ is larger than all the entries in this row, place $a$ in a new cell at the end of row $i+1$ (adding $1$ to $\lambda_{i+1}$ to change the outer shape), and stop here. Otherwise, let $a$ replace the smallest entry $a'$ in row $i+1$ such that $a'>a$, and recursively insert element $a'$ into the next row $i+2$.
Denote by $\ins_i(T)$ the standard cylindric tableau obtained at the end of the process.

In the special case that $T$ did not have an entry in cell $\langle i,j \rangle$ (which happens when $\mu_i=\lambda_i$), simply add $1$ to both $\lambda_i$ and $\mu_i$. This changes the inner and outer shape, but the cells of $\ins_i(T)$ are the same as those of $T$.
\end{definition}

An example of internal row insertion is given in Figure~\ref{fig:insertion}. Unlike in the usual RSK correspondence, the insertion process in the cylindric case may visit the same row multiple times, since rows are indexed modulo~$d$.

\begin{figure}[htb]
\centering
\begin{tikzpicture}[scale=0.5]
 \fill[yellow!25] (-1,-.5)--(-1,0) \east \north\east\north\east\east\north\east -- ++(0,.5)-- ++(3,0) -- ++(0,-.5) \west\west\south\south\west\west\south\west\west -- ++(0,-0.5) -- ++(-2,0);
 \fill[yellow] (0,0) \north\east\north\east\east\north\east\east\south\south\west\west\south\west\west\west;
 \hatch{3}{2}
\halfhatch{-1}{-.5}
 \draw[very thin,dotted] (0,0)  grid (5,3);
 \draw[black, very thick] (-1,-.5)--(-1,0) \east \north\east\north\east\east\north\east  -- ++(0,.5);
 \draw[black, very thick] (1,-.5)--(1,0) \east\east \north\east\east\north\north\east\east -- ++(0,.5);
 \draw[thick,<->,olive] (-.2,0)-- node[left] {$d$} (-.2,3);
 \draw[thick,<->,olive] (0,3.2)-- node[above] {$L$} (4,3.2); 
 \draw (1.5,1.5) node {1}; 
 \draw (3.5,2.5) node {2}; \draw (3.5-4,2.5-3) node {\footnotesize2};  
 \draw (.5,.5) node {3}; \draw (.5+4,.5+3) node {\footnotesize3}; 
 \draw (1.5,.5) node {4}; \draw (1.5+4,.5+3) node {\footnotesize4}; 
 \draw (2.5,1.5) node {5}; 
 \draw (2.5,.5) node {6}; \draw (2.5+4,.5+3) node {\footnotesize6};
 \draw (4.5,2.5) node {7}; \draw (4.5-4,2.5-3) node {\footnotesize7}; 
 \draw (3.5,1.5) node {8};
 \draw (4.5,1.5) node {9}; 
 \draw[thick,->] (0,3.5)--(0,-1.5);
\draw[thick,->] (-1,3)--(7.5,3);
\draw[->] (7,1)--(7.5,1);

\begin{scope}[shift={(10 ,0)}]
 \fill[yellow!25] (0,-.5)--(0,0) \north\east\north\east\east\east \north-- ++(0,.5)-- ++(3,0) -- ++(0,-.5) \west\west\south\south\west\west\south\west\west -- ++(0,-0.5) -- ++(-2,0);
 \fill[yellow] (0,0) \north\east\north\east\east\east\north\east\south\south\west\west\south\west\west\west;
 \draw[very thin,dotted] (0,0)  grid (5,3);
 \draw[black, very thick] (0,-.5)--(0,0) \north\east\north\east\east\east\north  -- ++(0,.5);
 \draw[black, very thick] (1,-.5)--(1,0) \east\east \north\east\east\north\north\east\east -- ++(0,.5);
 \draw (1.5,1.5) node {1}; 
 \draw (.5,.5) node {3}; \draw (.5+4,.5+3) node {\footnotesize3}; 
 \draw (1.5,.5) node {4}; \draw (1.5+4,.5+3) node {\footnotesize4}; 
 \draw (2.5,1.5) node {5}; 
 \draw (2.5,.5) node {6}; \draw (2.5+4,.5+3) node {\footnotesize6};
 \draw (4.5,2.5) node {7}; \draw (4.5-4,2.5-3) node {\footnotesize7}; 
 \draw (3.5,1.5) node {8};
 \draw (4.5,1.5) node {9}; 
 \draw[thick,->] (0,3.5)--(0,-1.5);
\draw[thick,->] (-.5,3)--(7.5,3);
\draw[->] (7,1)--(7.5,1);
\draw (-1,1.5) node {$2\rightsquigarrow$};
\end{scope}

\begin{scope}[shift={(20 ,0)}]
 \fill[yellow!25] (0,-.5)--(0,0) \north\east\north\east\east\east \north-- ++(0,.5)-- ++(3,0) -- ++(0,-.5) \west\west\south\south\west\west\south\west\west -- ++(0,-0.5) -- ++(-2,0);
 \fill[yellow] (0,0) \north\east\north\east\east\east\north\east\south\south\west\west\south\west\west\west;
 \draw[very thin,dotted] (0,0)  grid (5,3);
 \draw[black, very thick] (0,-.5)--(0,0) \north\east\north\east\east\east\north  -- ++(0,.5);
 \draw[black, very thick] (1,-.5)--(1,0) \east\east \north\east\east\north\north\east\east -- ++(0,.5);
 \draw (1.5,1.5) node {1}; 
 \draw (.5,.5) node {3}; \draw (.5+4,.5+3) node {\footnotesize3}; 
 \draw (1.5,.5) node {4}; \draw (1.5+4,.5+3) node {\footnotesize4}; 
 \draw (2.5,1.5) node {2}; 
 \draw (2.5,.5) node {6}; \draw (2.5+4,.5+3) node {\footnotesize6};
 \draw (4.5,2.5) node {7}; \draw (4.5-4,2.5-3) node {\footnotesize7}; 
 \draw (3.5,1.5) node {8};
 \draw (4.5,1.5) node {9}; 
 \draw[thick,->] (0,3.5)--(0,-1.5);
\draw[thick,->] (-.5,3)--(7.5,3);
\draw (-1,.5) node {$5\rightsquigarrow$};
\end{scope}

\begin{scope}[shift={(1,-6)}]
\draw[->] (-3,1)--(-2.5,1);
 \fill[yellow!25] (0,-.5)--(0,0) \north\east\north\east\east\east \north-- ++(0,.5)-- ++(3,0) -- ++(0,-.5) \west\west\south\south\west\west\south\west\west -- ++(0,-0.5) -- ++(-2,0);
 \fill[yellow] (0,0) \north\east\north\east\east\east\north\east\south\south\west\west\south\west\west\west;
 \draw[very thin,dotted] (0,0)  grid (5,3);
 \draw[black, very thick] (0,-.5)--(0,0) \north\east\north\east\east\east\north  -- ++(0,.5);
 \draw[black, very thick] (1,-.5)--(1,0) \east\east \north\east\east\north\north\east\east -- ++(0,.5);
 \draw (1.5,1.5) node {1}; 
 \draw (.5,.5) node {3}; \draw (.5+4,.5+3) node {\footnotesize3}; 
 \draw (1.5,.5) node {4}; \draw (1.5+4,.5+3) node {\footnotesize4}; 
 \draw (2.5,1.5) node {2}; 
 \draw (2.5,.5) node {5}; \draw (2.5+4,.5+3) node {\footnotesize5};
 \draw (4.5,2.5) node {7}; \draw (4.5-4,2.5-3) node {\footnotesize7}; 
 \draw (3.5,1.5) node {8};
 \draw (4.5,1.5) node {9}; 
 \draw[thick,->] (0,3.5)--(0,-1.5);
\draw[thick,->] (-.5,3)--(7.5,3);
\draw (-1,2.5) node {$6\rightsquigarrow$};
\end{scope}

\begin{scope}[shift={(11,-6)}]
\draw[->] (-3,1)--(-2.5,1);
 \fill[yellow!25] (0,-.5)--(0,0) \north\east\north\east\east\east \north-- ++(0,.5)-- ++(3,0) -- ++(0,-.5) \west\west\south\south\west\west\south\west\west -- ++(0,-0.5) -- ++(-2,0);
 \fill[yellow] (0,0) \north\east\north\east\east\east\north\east\south\south\west\west\south\west\west\west;
 \draw[very thin,dotted] (0,0)  grid (5,3);
 \draw[black, very thick] (0,-.5)--(0,0) \north\east\north\east\east\east\north  -- ++(0,.5);
 \draw[black, very thick] (1,-.5)--(1,0) \east\east \north\east\east\north\north\east\east -- ++(0,.5);
 \draw (1.5,1.5) node {1}; 
 \draw (.5,.5) node {3}; \draw (.5+4,.5+3) node {\footnotesize3}; 
 \draw (1.5,.5) node {4}; \draw (1.5+4,.5+3) node {\footnotesize4}; 
 \draw (2.5,1.5) node {2}; 
 \draw (2.5,.5) node {5}; \draw (2.5+4,.5+3) node {\footnotesize5};
 \draw (4.5,2.5) node {6}; \draw (4.5-4,2.5-3) node {\footnotesize6}; 
 \draw (3.5,1.5) node {8};
 \draw (4.5,1.5) node {9}; 
 \draw[thick,->] (0,3.5)--(0,-1.5);
\draw[thick,->] (-.5,3)--(7.5,3);
\draw (-1,1.5) node {$7\rightsquigarrow$};
\end{scope}

\begin{scope}[shift={(21,-6)}]
\draw[->] (-3,1)--(-2.5,1);
 \fill[yellow!25] (0,-.5)--(0,0) \north\east\north\east\east\east \north-- ++(0,.5)-- ++(3,0) -- ++(0,-.5) \west\west\south\south\west\west\south\west\west -- ++(0,-0.5) -- ++(-2,0);
 \fill[yellow] (0,0) \north\east\north\east\east\east\north\east\south\south\west\west\south\west\west\west;
 \draw[very thin,dotted] (0,0)  grid (5,3);
 \draw[black, very thick] (0,-.5)--(0,0) \north\east\north\east\east\east\north  -- ++(0,.5);
 \draw[black, very thick] (1,-.5)--(1,0) \east\east \north\east\east\north\north\east\east -- ++(0,.5);
 \draw (1.5,1.5) node {1}; 
 \draw (.5,.5) node {3}; \draw (.5+4,.5+3) node {\footnotesize3}; 
 \draw (1.5,.5) node {4}; \draw (1.5+4,.5+3) node {\footnotesize4}; 
 \draw (2.5,1.5) node {2}; 
 \draw (2.5,.5) node {5}; \draw (2.5+4,.5+3) node {\footnotesize5};
 \draw (4.5,2.5) node {6}; \draw (4.5-4,2.5-3) node {\footnotesize6}; 
 \draw (3.5,1.5) node {7};
 \draw (4.5,1.5) node {9}; 
 \draw[thick,->] (0,3.5)--(0,-1.5);
\draw[thick,->] (-.5,3)--(7.5,3);
\draw (-1,.5) node {$8\rightsquigarrow$};
\end{scope}

\begin{scope}[shift={(1,-12)}]
\draw[->] (-3,1)--(-2.5,1);
 \fill[yellow!25] (0,-.5)--(0,0) \north\east\north\east\east\east \north-- ++(0,.5)-- ++(4,0) -- ++(0,-.5) \west\west\west\south\south\west\south\west\west\west -- ++(0,-0.5) -- ++(-2,0);
 \fill[yellow] (0,0) \north\east\north\east\east\east\north\east\south\south\west\south\west\west\west\west;
 \bright{3}{0}
 \draw[very thin,dotted] (0,0)  grid (5,3);
 \draw[black, very thick] (0,-.5)--(0,0) \north\east\north\east\east\east\north  -- ++(0,.5);
 \draw[black, very thick] (1,-.5)--(1,0) \east\east\east\north\east\north\north\east\east\east -- ++(0,.5);
 \draw (1.5,1.5) node {1}; 
 \draw (.5,.5) node {3}; \draw (.5+4,.5+3) node {\footnotesize3}; 
 \draw (1.5,.5) node {4}; \draw (1.5+4,.5+3) node {\footnotesize4}; 
 \draw (2.5,1.5) node {2}; 
 \draw (2.5,.5) node {5}; \draw (2.5+4,.5+3) node {\footnotesize5};
 \draw (3.5,.5) node {8}; \draw (3.5+4,.5+3) node {\footnotesize8};
 \draw (4.5,2.5) node {6}; \draw (4.5-4,2.5-3) node {\footnotesize6}; 
 \draw (3.5,1.5) node {7};
 \draw (4.5,1.5) node {9}; 
 \draw[thick,->] (0,3.5)--(0,-1.5);
\draw[thick,->] (-.5,3)--(8.5,3);
\end{scope}
\end{tikzpicture}
\caption{Internal insertion at row $1$ in the standard cylindric tableau from Figure~\ref{fig:SCT}.}
\label{fig:insertion}
\end{figure}
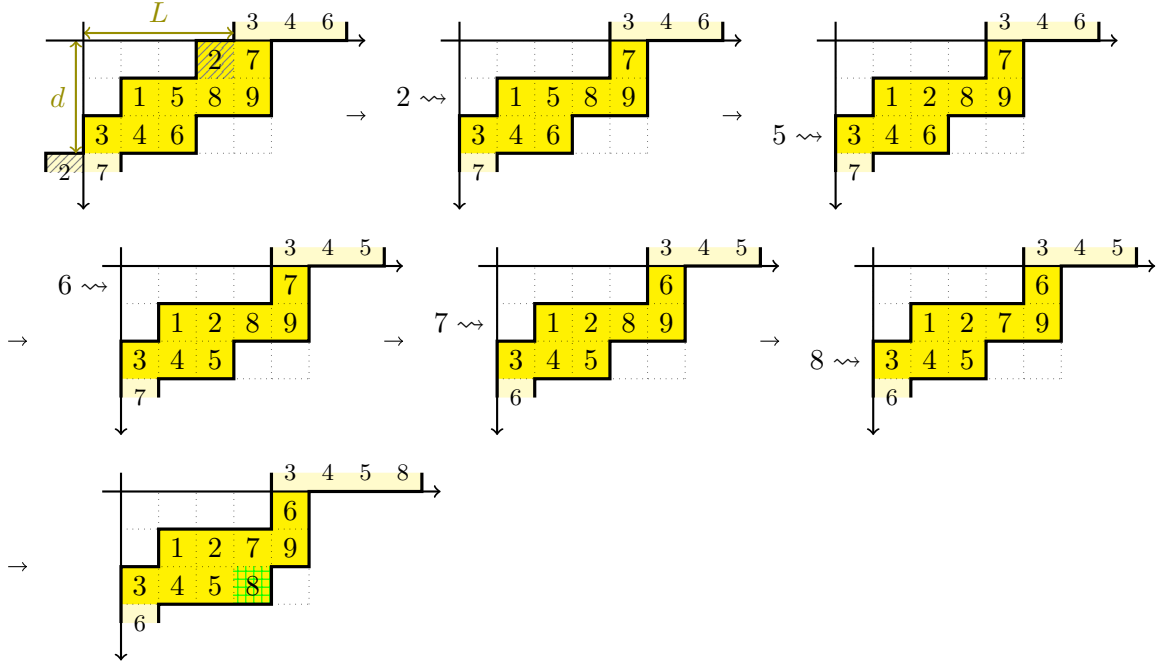 

We can now describe Neyman's bijection in the case of standard cylindric tableaux. We call it the {\em cylindric Robinson--Schensted} algorithm, and denote it by $\CRS$. The symbol $\bigsqcup$ denotes disjoint union.

\begin{theorem}[\cite{neyman_cylindric_2015}]
\label{thm:CRS}
Fix $\alpha,\beta\in\CS_{d,L}$ and $n,m\ge0$. There is a bijection
$$\CRS:\bigsqcup_{\substack{\mu\subseteq\alpha,\beta\\ |\alpha/\mu|=n,|\beta/\mu|=m}} \SCT_{d,L}(\alpha/\mu)\times \SCT_{d,L}(\beta/\mu) \to \bigsqcup_{\substack{\lambda\supseteq\alpha,\beta\\ |\lambda/\beta|=n,|\lambda/\alpha|=m}} \SCT_{d,L}(\lambda/\beta)\times \SCT_{d,L}(\lambda/\alpha).$$
\end{theorem}

\begin{proof}
Suppose that $|\alpha|+m=|\beta|+n$, since otherwise both unions would be empty.
Fix $\mu\subseteq\alpha,\beta$ with $|\alpha/\mu|=n$ and $|\beta/\mu|=m$, and let $(T,U)\in \SCT_{d,L}(\alpha/\mu)\times \SCT_{d,L}(\beta/\mu)$.
We will build a sequence of pairs $(P_k,Q_k)$ for $0\le k\le m$. For $k=0$, let $P_0=T$, and let $Q_0$ be the empty tableau of shape $\alpha/\alpha$. For $k$ from $1$ to $m$, construct $(P_k,Q_k)$ iteratively as follows.
Let $\langle i,j\rangle$ be the cell in $U$ containing $k$, which must be an inner corner of $P_{k-1}$. Let $P_k=\ins_i(P_{k-1})$ (as in Definition~\ref{def:insertion}), and let $Q_k$ be obtained from $Q_{k-1}$ by placing $k$ in the cell where this insertion procedure terminates (that is, the cell that is added to $P_{k-1}$ to obtain $P_k$).
Finally, let $\CRS(T,U)=(P_m,Q_m)$. See Figure~\ref{fig:CRS} for an example.
\end{proof}

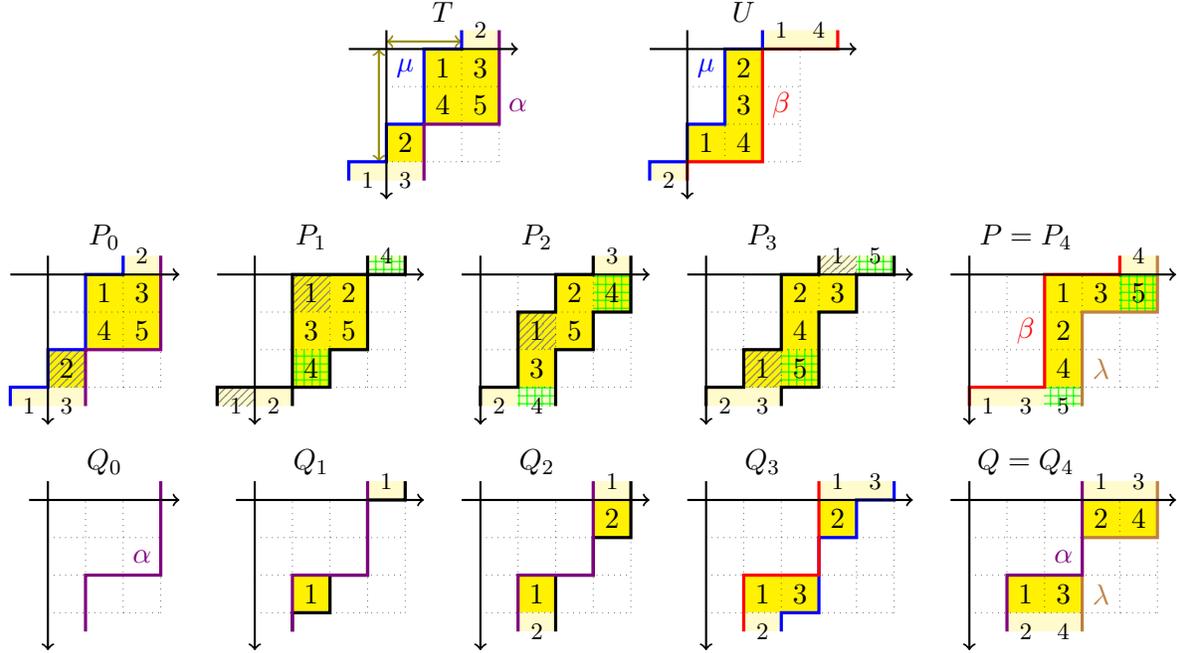
\begin{figure}[htb]
\centering
\begin{tikzpicture}[scale=0.5]
 \fill[yellow!25] (-1,-.5)++(0,.5) \east \north\east\north\north\east -- ++(0,.5)-- ++(1,0) -- ++(0,-.5) \south\south\west\west\south -- ++(0,-0.5) -- ++(-2,0);
 \fill[yellow] (0,0) \north\east\north\north\east \east  \south\south\west\west\south \west;
 \draw[very thin,dotted] (0,0)  grid (3,3);
 \draw[thick,<->,olive] (-.2,0)-- (-.2,3);
 \draw[thick,<->,olive] (0,3.2)-- (2,3.2); 
 \draw[blue, very thick] (-1,-.5)--++(0,.5) \east \north\east\north\north\east -- ++(0,.5);
 \draw[violet, very thick] (3,3.5)-- ++(0,-.5) \south\south\west\west\south -- ++(0,-0.5);
 \draw (1.5,2.5) node {1}; \draw (1.5-2,2.5-3) node {\footnotesize 1}; 
 \draw (.5,.5) node {2}; \draw (.5+2,.5+3) node {\footnotesize 2};  
 \draw (2.5,2.5) node {3}; \draw (2.5-2,2.5-3) node {\footnotesize 3}; 
 \draw (1.5,1.5) node {4}; 
 \draw (2.5,1.5) node {5}; 
 \draw[blue] (.5,2.5) node {$\mu$};
 \draw[violet] (3.5,1.5) node {$\alpha$};
 \draw[thick,->] (0,3.5)--(0,-1);
\draw[thick,->] (-1,3)--(3.5,3);
\draw (1.5,4) node {$T$};

\begin{scope}[shift={(8,0)}]
 \fill[yellow!25] (-1,-.5)++(0,.5) \east \north\east\north\north\east -- ++(0,.5)-- ++(2,0) -- ++(0,-.5) \west\west\south\south\south\west\west -- ++(0,-0.5) -- ++(-1,0);
 \fill[yellow] (0,0) \north\east\north\north\east \east  \west\south\south\south\west\west;
 \draw[very thin,dotted] (0,0)  grid (3,3);
 \draw[blue, very thick] (-1,-.5)--++(0,.5) \east \north\east\north\north\east -- ++(0,.5);
 \draw[red, very thick] (4,3.5)-- ++(0,-.5) \west\west\south\south\south\west\west-- ++(0,-0.5);
 \draw (.5,.5) node {1}; \draw (.5+2,.5+3) node {\footnotesize 1}; 
 \draw (1.5,2.5) node {2}; \draw (1.5-2,2.5-3) node {\footnotesize 2};  
 \draw (1.5,1.5) node {3}; 
 \draw (1.5,.5) node {4}; \draw (1.5+2,.5+3) node {\footnotesize 4}; 
 \draw[blue] (.5,2.5) node {$\mu$};
 \draw[red] (2.5,1.5) node {$\beta$};
 \draw[thick,->] (0,3.5)--(0,-1);
\draw[thick,->] (-1,3)--(4.5,3);
\draw (1.5,4) node {$U$};
\end{scope}

\begin{scope}[shift={(-9,-6)}]
 \fill[yellow!25] (-1,-.5)++(0,.5) \east \north\east\north\north\east -- ++(0,.5)-- ++(1,0) -- ++(0,-.5) \south\south\west\west\south -- ++(0,-0.5) -- ++(-2,0);
 \fill[yellow] (0,0) \north\east\north\north\east \east  \south\south\west\west\south \west;
\hatch{0}{0}
 \draw[very thin,dotted] (0,0)  grid (3,3);
 \draw[blue, very thick] (-1,-.5)--++(0,.5) \east \north\east\north\north\east -- ++(0,.5);
 \draw[violet, very thick] (3,3.5)-- ++(0,-.5) \south\south\west\west\south -- ++(0,-0.5);
 \draw (1.5,2.5) node {1}; \draw (1.5-2,2.5-3) node {\footnotesize 1}; 
 \draw (.5,.5) node {2}; \draw (.5+2,.5+3) node {\footnotesize 2};  
 \draw (2.5,2.5) node {3}; \draw (2.5-2,2.5-3) node {\footnotesize 3}; 
 \draw (1.5,1.5) node {4}; 
 \draw (2.5,1.5) node {5}; 
 \draw[thick,->] (0,3.5)--(0,-1);
\draw[thick,->] (-1,3)--(3.5,3);
\draw (1.5,4) node {$P_0$};

\begin{scope}[shift={(5.5,0)}]
 \fill[yellow!25] (-1,-.5)++(0,.5) \east \east\north\north\north\east\east -- ++(0,.5)-- ++(1,0) -- ++(0,-.5) \west\south\south\west\south\west -- ++(0,-0.5) -- ++(-2,0);
 \fill[yellow] (0,0) \east\north\north\north\east \east  \south\south\west\south\west \west;
\bright{1}{0}\halfbright{1+2}{0+3}
\hatch{1}{2}\halfhatch{-1}{-.5}
 \draw[very thin,dotted] (0,0)  grid (4,3);
 \draw[black,very thick] (-1,-.5)--++(0,.5) \east \east\north\north\north\east\east -- ++(0,.5);
 \draw[black,very thick] (4,3.5)-- ++(0,-.5) \west\south\south\west\south\west -- ++(0,-0.5);
 \draw (1.5,2.5) node {1}; \draw (1.5-2,2.5-3) node {\footnotesize 1}; 
 \draw (1.5,.5) node {4}; \draw (1.5+2,.5+3) node {\footnotesize 4};  
 \draw (2.5,2.5) node {2}; \draw (2.5-2,2.5-3) node {\footnotesize 2}; 
 \draw (1.5,1.5) node {3}; 
 \draw (2.5,1.5) node {5}; 
 \draw[thick,->] (0,3.5)--(0,-1);
\draw[thick,->] (-1,3)--(4.5,3);
\draw (1.5,4) node {$P_1$};
\end{scope}

\begin{scope}[shift={(11.5,0)}]
 \fill[yellow!25] (0,-.5)++(0,.5) \east\north\north\east\north\east -- ++(0,.5)-- ++(1,0) -- ++(0,-.5) \south\west\south\west\south -- ++(0,-0.5) -- ++(-2,0);
 \fill[yellow] (0,0) \east\north\north\east\north\east \east  \south\west\south\west\south \west\west;
\hatch{1}{1}
\bright{3}{2}\halfbright{3-2}{2-2.5}
 \draw[very thin,dotted] (0,0)  grid (4,3);
 \draw[black, very thick] (0,-.5)--++(0,.5) \east\north\north\east\north\east -- ++(0,.5);
 \draw[black, very thick] (4,3.5)-- ++(0,-.5) \south\west\south\west\south -- ++(0,-0.5);
 \draw (3.5,2.5) node {4}; \draw (3.5-2,2.5-3) node {\footnotesize 4}; 
 \draw (1.5,.5) node {3}; \draw (1.5+2,.5+3) node {\footnotesize 3};  
 \draw (2.5,2.5) node {2}; \draw (2.5-2,2.5-3) node {\footnotesize 2}; 
 \draw (1.5,1.5) node {1}; 
 \draw (2.5,1.5) node {5}; 
 \draw[thick,->] (0,3.5)--(0,-1);
\draw[thick,->] (-.5,3)--(4.5,3);
\draw (1.5,4) node {$P_2$};
\end{scope}

\begin{scope}[shift={(17.5,0)}]
 \fill[yellow!25] (0,-.5)++(0,.5) \east\north\east\north\north\east -- ++(0,.5)-- ++(2,0) -- ++(0,-.5) \west\south\west\south\south\west -- ++(0,-0.5) -- ++(-2,0);
 \fill[yellow] (0,0) \east\north\east\north\north\east \east  \south\west\south\south\west \west\west;
 \hatch{1}{0}
\halfhatch{3}{3}
\bright{2}{0}
\halfbright{2+2}{0+3}
 \draw[very thin,dotted] (0,0)  grid (5,3);
 \draw[black, very thick] (0,-.5)--++(0,.5) \east\north\east\north\north\east -- ++(0,.5);
 \draw[black, very thick] (5,3.5)-- ++(0,-.5) \west\south\west\south\south\west -- ++(0,-0.5);
 \draw (3.5,2.5) node {3}; \draw (3.5-2,2.5-3) node {\footnotesize 3}; 
 \draw (1.5,.5) node {1}; \draw (1.5+2,.5+3) node {\footnotesize 1};  
 \draw (2.5,2.5) node {2}; \draw (2.5-2,2.5-3) node {\footnotesize 2}; 
 \draw (2.5,.5) node {5}; \draw (2.5+2,.5+3) node {\footnotesize 5};  
 \draw (2.5,1.5) node {4}; 
 \draw[thick,->] (0,3.5)--(0,-1);
\draw[thick,->] (-.5,3)--(5.5,3);
\draw (1.5,4) node {$P_3$};
\end{scope}

\begin{scope}[shift={(24.5,0)}]
 \fill[yellow!25] (0,-.5)++(0,.5) \east\east\north\north\north\east\east -- ++(0,.5)-- ++(1,0) -- ++(0,-.5) \south\west\west\south\south -- ++(0,-0.5) -- ++(-3,0);
 \fill[yellow] (0,0) \east\east\north\north\north\east\east \east  \south\west\west\south\south\west \west\west;
\bright{4}{2}
\halfbright{4-2}{2-2.5}
 \draw[very thin,dotted] (0,0)  grid (5,3);
 \draw[red, very thick] (0,-.5)--++(0,.5) \east\east\north\north\north\east\east -- ++(0,.5);
 \draw[brown, very thick] (5,3.5)-- ++(0,-.5) \south\west\west\south\south -- ++(0,-0.5);
 \draw (3.5,2.5) node {3}; \draw (3.5-2,2.5-3) node {\footnotesize 3}; 
 \draw (4.5,2.5) node {5}; \draw (4.5-2,2.5-3) node {\footnotesize 5};  
 \draw (2.5,2.5) node {1}; \draw (2.5-2,2.5-3) node {\footnotesize 1}; 
 \draw (2.5,.5) node {4}; \draw (2.5+2,.5+3) node {\footnotesize 4};  
 \draw (2.5,1.5) node {2}; 
 \draw[thick,->] (0,3.5)--(0,-1);
\draw[thick,->] (-.5,3)--(5.5,3);
\draw (1.5,4) node {$P=P_4$};
\draw[red] (1.5,1.5) node {$\beta$};
\draw[brown] (3.5,.5) node {$\lambda$};
\end{scope}
\end{scope}

\begin{scope}[shift={(-9,-12)}]
 \draw[very thin,dotted] (0,0)  grid (3,3);
 \draw[violet, very thick] (1,-.5)--++(0,.5)  \north\east\east\north\north -- ++(0,.5);
 \draw[thick,->] (0,3.5)--(0,-1);
\draw[thick,->] (-.5,3)--(3.5,3);
\draw (1.5,4) node {$Q_0$};
\draw[violet] (2.5,1.5) node {$\alpha$};

\begin{scope}[shift={(5.5,0)}]
 \fill[yellow!25] (1+2,0+3) rectangle (2+2,3.5);
 \fill[yellow] (1,0) rectangle (2,1);
 \draw[very thin,dotted] (0,0)  grid (4,3);
 \draw[black, very thick] (4,3.5)-- ++(0,-.5)\west\south\south\west\south\west-- ++(0,-0.5);
 \draw[violet, very thick] (1,-.5)--++(0,.5)  \north\east\east\north\north -- ++(0,.5);
 \draw (1.5,.5) node {1}; \draw (1.5+2,.5+3) node {\footnotesize 1}; 
 \draw[thick,->] (0,3.5)--(0,-1);
\draw[thick,->] (-.5,3)--(4.5,3);
\draw (1.5,4) node {$Q_1$};
\end{scope}

\begin{scope}[shift={(11.5,0)}]
 \fill[yellow!25] (1+2,0+3) rectangle (2+2,3.5);  \fill[yellow!25] (1,-.5) rectangle (2,0);
 \fill[yellow] (1,0) rectangle (2,1);  \fill[yellow] (3,2) rectangle (4,3);
 \draw[very thin,dotted] (0,0)  grid (4,3);
 \draw[black, very thick] (4,3.5)-- ++(0,-.5)\south\west\south\west\south-- ++(0,-0.5);
 \draw[violet, very thick] (1,-.5)--++(0,.5)  \north\east\east\north\north -- ++(0,.5);
 \draw (1.5,.5) node {1}; \draw (1.5+2,.5+3) node {\footnotesize 1}; 
 \draw (3.5,2.5) node {2}; \draw (3.5-2,2.5-3) node {\footnotesize 2}; 
 \draw[thick,->] (0,3.5)--(0,-1);
\draw[thick,->] (-.5,3)--(4.5,3);
\draw (1.5,4) node {$Q_2$};
\end{scope}

\begin{scope}[shift={(17.5,0)}]
 \fill[yellow!25] (3,3) rectangle (5,3.5);  \fill[yellow!25] (1,-.5) rectangle (2,0);
 \fill[yellow] (1,0) rectangle (3,1);  \fill[yellow] (3,2) rectangle (4,3);
 \draw[very thin,dotted] (0,0)  grid (5,3);
 \draw[blue, very thick] (5,3.5)-- ++(0,-.5)\west\south\west\south\south\west-- ++(0,-0.5);
 \draw[red, very thick] (1,-.5)--++(0,.5)  \north\east\east\north\north -- ++(0,.5);
 \draw (1.5,.5) node {1}; \draw (1.5+2,.5+3) node {\footnotesize 1}; 
 \draw (3.5,2.5) node {2}; \draw (3.5-2,2.5-3) node {\footnotesize 2}; 
 \draw (2.5,.5) node {3}; \draw (2.5+2,.5+3) node {\footnotesize 3};  
 \draw[thick,->] (0,3.5)--(0,-1);
\draw[thick,->] (-.5,3)--(5.5,3);
\draw (1.5,4) node {$Q_3$};
\end{scope}

\begin{scope}[shift={(24.5,0)}]
 \fill[yellow!25] (1,-.5)++(0,.5) \north\east\east\north\north -- ++(0,.5)-- ++(2,0) -- ++(0,-.5) \south\west\west\south\south -- ++(0,-0.5) -- ++(-2,0);
 \fill[yellow] (1,0)\north\east\east\north\north \east\east \south\west\west\south\south  \west\west;
 \draw[very thin,dotted] (0,0)  grid (5,3);
 \draw[brown, very thick] (5,3.5)-- ++(0,-.5)\south\west\west\south\south-- ++(0,-0.5);
 \draw[violet, very thick] (1,-.5)--++(0,.5)  \north\east\east\north\north -- ++(0,.5);
 \draw (1.5,.5) node {1}; \draw (1.5+2,.5+3) node {\footnotesize 1}; 
 \draw (4.5,2.5) node {4}; \draw (4.5-2,2.5-3) node {\footnotesize 4};  
 \draw (3.5,2.5) node {2}; \draw (3.5-2,2.5-3) node {\footnotesize 2}; 
 \draw (2.5,.5) node {3}; \draw (2.5+2,.5+3) node {\footnotesize 3};  
 \draw[thick,->] (0,3.5)--(0,-1);
\draw[thick,->] (-.5,3)--(5.5,3);
\draw (1.5,4) node {$Q=Q_4$};
\draw[violet] (2.5,1.5) node {$\alpha$};
\draw[brown] (3.5,.5) node {$\lambda$};
\end{scope}
\end{scope} 
\end{tikzpicture}
\caption{An example of the bijection $\CRS$ from Theorem~\ref{thm:CRS} with $\alpha=[3,3,1],\beta=[2,2,2]\in\CS_{3,2}$, and $n=5$, $m=4$. In this example, $\mu=[1,1,0]$ and $\lambda=[5,3,3]$. In the computation of $\CRS(T,U)=(P,Q)$, the inner corner of each $P_k$ where internal insertion is about to occur is shaded with \textcolor{gray}{gray} diagonal lines, and the newly added cell where the previous insertion terminated is highlighted with a \textcolor{green}{green} grid pattern. }
\label{fig:CRS}
\end{figure} 

Neyman \cite{neyman_cylindric_2015} also introduces a more complicated multi-insertion operation to deal with repeated entries in semistandard cylindric tableaux. He then uses this operation to describe the {\em cylindric Robinson--Schensted--Knuth} algorithm, denoted by $\CRSK$. 

\begin{theorem}[\cite{neyman_cylindric_2015}]
\label{thm:CRSK}
Fix $\alpha,\beta\in\CS_{d,L}$. There is a bijection
$$\CRSK:\bigsqcup_{\mu\subseteq\alpha,\beta} \SSCT_{d,L}(\alpha/\mu)\times \SSCT_{d,L}(\beta/\mu) \to \bigsqcup_{\lambda\supseteq\alpha,\beta} \SSCT_{d,L}(\lambda/\beta)\times \SSCT_{d,L}(\lambda/\alpha)$$
such that, if $\CRSK(T,U)=(P,Q)$, then $\wt(T)=\wt(P)$ and $\wt(U)=\wt(Q)$.
\end{theorem}

\subsection{The symmetric case}

Neyman shows in \cite[Thm.~5.18]{neyman_cylindric_2015}, adapting a similar result for skew tableaux \cite[Thm.~3.3]{sagan_robinson-schensted_1990}, that if $\CRSK(T,U)=(P,Q)$, then $\CRSK(U,T)=(Q,P)$. We will see an alternative proof of this symmetry in Section~\ref{sec:symmetry}. In the special case of standard cylindric tableaux with $\alpha=\beta$, it follows that, if we let $T\in \SCT_{d,L}(\alpha/\mu)$ with $|\alpha/\mu|=n$, then $\CRS(T,T)=(P,P)$ for some $P\in\SCT_{d,L}(\lambda/\alpha)$ with $|\lambda/\alpha|=n$. 
Denoting by $\bij$ the map such that $\bij(T)=P$, Theorem~\ref{thm:CRS} implies the following result. See Figure~\ref{fig:bij} for an example.

\begin{corollary}\label{cor:bij}
For any fixed $\alpha\in\CS_{d,L}$, there is a bijection 
$$\bij:\SCT^n_{d,L}(\alpha/\cdot)\to \SCT^n_{d,L}(\cdot/\alpha).$$
\end{corollary}

\begin{figure}
\centering
\begin{tikzpicture}[scale=0.5]

\fill[yellow!25] (-4,-.5) rectangle (-1,0);
\fill[yellow!25] (1,3) rectangle (3,3.5);
\fill[yellow] (-2,0) \north\east\north\north\east\east \east\east  \west\south\south\west\west\south \west\west;
\hatch{-1}{2}
\halfhatch{-4}{-.5}
 \draw[very thin,dotted] (-2,0)  grid (2,3);
 \draw[thick,<->,olive] (-.2,0)-- (-.2,3);
 \draw[thick,<->,olive] (0,3.2)-- (3,3.2); 
 \draw[black, very thick] (-4,-.5)--++(0,.5) \east\east \north\east\north\north\east\east -- ++(0,.5);
 \draw[violet, ultra thick] (3,3.5)-- ++(0,-.5)\west\south\south\west\west\south \west -- ++(0,-0.5);
 \draw (-.5,2.5) node {1}; \draw (-.5-3,2.5-3) node {\footnotesize 1}; 
 \draw (-1.5,.5) node {2}; \draw (-1.5+3,.5+3) node {\footnotesize 2};  
 \draw (.5,2.5) node {3}; \draw (.5-3,2.5-3) node {\footnotesize 3}; 
 \draw (-.5,1.5) node {4}; 
 \draw (.5,1.5) node {5}; 
 \draw (1.5,2.5) node {6}; \draw (1.5-3,2.5-3) node {\footnotesize 6}; 
 \draw (1.5,1.5) node {7}; 
 \draw (-.5,.5) node {8}; \draw (-.5+3,.5+3) node {\footnotesize 8};  
 \draw[violet] (2.5,1.5) node {$\alpha$};
 \draw[->] (0,3.5)--(0,-1);
\draw[->] (-4,3)--(3.5,3);
\draw (-.5,4) node {$T$};
\draw[->] (3.5,1.5)--(4,1.5);

\begin{scope}[shift={(7.5,0)}]
\fill[yellow!25] (-3,-.5) rectangle (0,0);
\fill[yellow!25] (1,3) rectangle (3,3.5);
\fill[yellow] (-2,0) \north\east\north\east\north\east \east\east \south \west\south\west\west\south \west\west;
\hatch{-2}{0}\halfhatch{1}{3}
\newcell{2}{2}
\fill[pink!50](-1,-.5) rectangle (0,0);
\bright{2}{2}\halfbright{2-3}{2-2.5}
 \draw[very thin,dotted] (-2,0)  grid (3,3);
 \draw[black, very thick] (-3,-.5)--++(0,.5) \east \north\east\north\east\north\east -- ++(0,.5);
 \draw[black, very thick] (3,3.5)-- ++(0,-.5)\south\west\south\west\west\south  -- ++(0,-0.5);
  \draw[violet, ultra thick,dotted] (3,3.5)-- ++(0,-.5)\west\south\south\west\west\south \west -- ++(0,-0.5);
 \draw (2.5,2.5) node {8}; \draw (2.5-3,2.5-3) node {\footnotesize 8}; 
 \draw (-1.5,.5) node {2}; \draw (-1.5+3,.5+3) node {\footnotesize 2};  
 \draw (.5,2.5) node {3}; \draw (.5-3,2.5-3) node {\footnotesize 3}; 
 \draw (-.5,1.5) node {1}; 
 \draw (.5,1.5) node {5}; 
 \draw (1.5,2.5) node {6}; \draw (1.5-3,2.5-3) node {\footnotesize 6}; 
 \draw (1.5,1.5) node {7}; 
 \draw (-.5,.5) node {4}; \draw (-.5+3,.5+3) node {\footnotesize 4};  
 \draw[->] (0,3.5)--(0,-1);
\draw[->] (-3,3)--(3.5,3);
\draw[->] (3.5,1.5)--(4,1.5);
\end{scope}

\begin{scope}[shift={(15,0)}]
\fill[yellow!25] (-3,-.5) rectangle (0,0);
\fill[yellow!25] (2,3) rectangle (4,3.5);
\fill[yellow] (-1,0) \north\north\east\north\east\east \east \south \west\south\west\south \west\west;
\hatch{0}{2}\halfhatch{-3}{-.5}
\newcell{2}{2}\newcell{0}{0}
\fill[pink!50](-1,-.5) rectangle (0,0);
\fill[pink!50](3,3) rectangle (4,3.5);
\bright{0}{0}\halfbright{3}{3}
 \draw[very thin,dotted] (-1,0)  grid (3,3);
 \draw[black, very thick] (-3,-.5)--++(0,.5) \east\east\north\north\east\north\east\east -- ++(0,.5);
 \draw[black, very thick] (4,3.5)-- ++(0,-.5)\west\south\west\south\west\south\west  -- ++(0,-0.5);
  \draw[violet, ultra thick,dotted] (3,3.5)-- ++(0,-.5)\west\south\south\west\west\south \west -- ++(0,-0.5);
 \draw (2.5,2.5) node {8}; \draw (2.5-3,2.5-3) node {\footnotesize 8}; 
 \draw (.5,.5) node {5}; \draw (.5+3,.5+3) node {\footnotesize 5};  
 \draw (.5,2.5) node {2}; \draw (.5-3,2.5-3) node {\footnotesize 2}; 
 \draw (-.5,1.5) node {1}; 
 \draw (.5,1.5) node {3}; 
 \draw (1.5,2.5) node {6}; \draw (1.5-3,2.5-3) node {\footnotesize 6}; 
 \draw (1.5,1.5) node {7}; 
 \draw (-.5,.5) node {4}; \draw (-.5+3,.5+3) node {\footnotesize 4};  
 \draw[->] (0,3.5)--(0,-1);
\draw[->] (-3,3)--(4.5,3);
\end{scope}

\begin{scope}[shift={(1,-5.5)}]
\fill[yellow!25] (-2,-.5) rectangle (0,0);
\fill[yellow!25] (2,3) rectangle (5,3.5);
\fill[yellow] (-1,0) \north\north\east\east\north\east \east \south\west\south\south \west\west\west;
\hatch{-1}{1}
\newcell{2}{2}\newcell{0}{0}\newcell{1}{0}
\fill[pink!50](-1,-.5) rectangle (0,0);
\fill[pink!50](3,3) rectangle (5,3.5);
\bright{1}{0}\halfbright{1+3}{3}
 \draw[very thin,dotted] (-1,0)  grid (3,3);
 \draw[black, very thick] (-2,-.5)--++(0,.5) \east\north\north\east\east\north\east -- ++(0,.5);
 \draw[black, very thick] (5,3.5)-- ++(0,-.5)\west\west\south\west\south\south\west\west  -- ++(0,-0.5);
  \draw[violet, ultra thick,dotted] (3,3.5)-- ++(0,-.5)\west\south\south\west\west\south \west -- ++(0,-0.5);
 \draw (2.5,2.5) node {8}; \draw (2.5-3,2.5-3) node {\footnotesize 8}; 
 \draw (.5,.5) node {5}; \draw (.5+3,.5+3) node {\footnotesize 5};  
 \draw (1.5,.5) node {7}; \draw (1.5+3,.5+3) node {\footnotesize 7}; 
 \draw (-.5,1.5) node {1}; 
 \draw (.5,1.5) node {2}; 
 \draw (1.5,2.5) node {4}; \draw (1.5-3,2.5-3) node {\footnotesize 4}; 
 \draw (1.5,1.5) node {6}; 
 \draw (-.5,.5) node {3}; \draw (-.5+3,.5+3) node {\footnotesize 3};  
 \draw[->] (0,3.5)--(0,-1);
\draw[->] (-2,3)--(5.5,3);
\draw[->] (-3,1.5)--(-2.5,1.5);

\begin{scope}[shift={(8.5,0)}]
\fill[yellow!25] (-2,-.5) rectangle (0,0);
\fill[yellow!25] (2,3) rectangle (5,3.5);
\fill[yellow] (-1,0) \north\east\north\east\north\east \east \south\south\west\south \west\west\west;
\hatch{0}{1}
\newcell{2}{2}\newcell{0}{0}\newcell{1}{0}\newcell{2}{1}
\fill[pink!50](-1,-.5) rectangle (0,0);
\fill[pink!50](3,3) rectangle (5,3.5);
\bright{2}{1}
 \draw[very thin,dotted] (-1,0)  grid (3,3);
 \draw[black, very thick] (-2,-.5)--++(0,.5) \east\north\east\north\east\north\east -- ++(0,.5);
 \draw[black, very thick] (5,3.5)-- ++(0,-.5)\west\west\south\south\west\south\west\west  -- ++(0,-0.5);
  \draw[violet, ultra thick,dotted] (3,3.5)-- ++(0,-.5)\west\south\south\west\west\south \west -- ++(0,-0.5);
 \draw (2.5,2.5) node {7}; \draw (2.5-3,2.5-3) node {\footnotesize 7}; 
 \draw (.5,.5) node {5}; \draw (.5+3,.5+3) node {\footnotesize 5};  
 \draw (1.5,.5) node {6}; \draw (1.5+3,.5+3) node {\footnotesize 6}; 
 \draw (2.5,1.5) node {8}; 
 \draw (.5,1.5) node {2}; 
 \draw (1.5,2.5) node {3}; \draw (1.5-3,2.5-3) node {\footnotesize 3}; 
 \draw (1.5,1.5) node {4}; 
 \draw (-.5,.5) node {1}; \draw (-.5+3,.5+3) node {\footnotesize 1};  
 \draw[->] (0,3.5)--(0,-1);
\draw[->] (-2,3)--(5.5,3);
\draw[->] (-3,1.5)--(-2.5,1.5);
\end{scope}

\begin{scope}[shift={(17,0)}]
\fill[yellow!25] (-2,-.5) rectangle (0,0);
\fill[yellow!25] (2,3) rectangle (6,3.5);
\fill[yellow] (-1,0) \north\east\east\north\north\east \east \south\south\south \west\west\west\west;
\hatch{1}{2}\halfhatch{-2}{-.5}
\newcell{2}{2}\newcell{0}{0}\newcell{1}{0}\newcell{2}{1}\newcell{2}{0}
\fill[pink!50](-1,-.5) rectangle (0,0);
\fill[pink!50](3,3) rectangle (6,3.5);
\bright{2}{0}\halfbright{2+3}{3}
 \draw[very thin,dotted] (-1,0)  grid (3,3);
 \draw[black, very thick] (-2,-.5)--++(0,.5) \east\north\east\east\north\north\east -- ++(0,.5);
 \draw[black, very thick] (6,3.5)-- ++(0,-.5)\west\west\west\south\south\south\west\west\west  -- ++(0,-0.5);
  \draw[violet, ultra thick,dotted] (3,3.5)-- ++(0,-.5)\west\south\south\west\west\south \west -- ++(0,-0.5);
 \draw (2.5,2.5) node {5}; \draw (2.5-3,2.5-3) node {\footnotesize 5}; 
 \draw (.5,.5) node {2}; \draw (.5+3,.5+3) node {\footnotesize 2};  
 \draw (1.5,.5) node {6}; \draw (1.5+3,.5+3) node {\footnotesize 6}; 
 \draw (2.5,1.5) node {7}; 
 \draw (2.5,.5) node {8}; \draw (2.5+3,.5+3) node {\footnotesize 8}; 
 \draw (1.5,2.5) node {3}; \draw (1.5-3,2.5-3) node {\footnotesize 3}; 
 \draw (1.5,1.5) node {4}; 
 \draw (-.5,.5) node {1}; \draw (-.5+3,.5+3) node {\footnotesize 1};  
 \draw[->] (0,3.5)--(0,-1);
\draw[->] (-2,3)--(6.5,3);
\draw[->] (-3,1.5)--(-2.5,1.5);
\end{scope}
\end{scope} 

\begin{scope}[shift={(0,-11)}]
\fill[yellow!25] (-1,-.5) rectangle (1,0);
\fill[yellow!25] (2,3) rectangle (6,3.5);
\fill[yellow] (-1,0) \north\east\east\north\east\north \east\east \south \west\south\south\west\west\west\west;
\hatch{1}{1}
\newcell{2}{2}\newcell{0}{0}\newcell{1}{0}\newcell{2}{1}\newcell{2}{0}\newcell{3}{2}
\fill[pink!50](-1,-.5) rectangle (1,0);
\fill[pink!50](3,3) rectangle (6,3.5);
\bright{3}{2}\halfbright{3-3}{2-2.5}
 \draw[very thin,dotted] (-1,0)  grid (4,3);
 \draw[black, very thick] (-1,-.5)--++(0,.5) \north\east\east\north\east\north -- ++(0,.5);
 \draw[black, very thick] (6,3.5)-- ++(0,-.5)\west\west\south\west\south\south\west\west  -- ++(0,-0.5);
  \draw[violet, ultra thick,dotted] (3,3.5)-- ++(0,-.5)\west\south\south\west\west\south \west -- ++(0,-0.5);
 \draw (2.5,2.5) node {5}; \draw (2.5-3,2.5-3) node {\footnotesize 5}; 
 \draw (.5,.5) node {2}; \draw (.5+3,.5+3) node {\footnotesize 2};  
 \draw (1.5,.5) node {4}; \draw (1.5+3,.5+3) node {\footnotesize 4}; 
 \draw (2.5,1.5) node {7}; 
 \draw (2.5,.5) node {8}; \draw (2.5+3,.5+3) node {\footnotesize 8}; 
 \draw (3.5,2.5) node {6}; \draw (3.5-3,2.5-3) node {\footnotesize 6}; 
 \draw (1.5,1.5) node {3}; 
 \draw (-.5,.5) node {1}; \draw (-.5+3,.5+3) node {\footnotesize 1};  
 \draw[->] (0,3.5)--(0,-1);
\draw[->] (-1,3)--(6.5,3);
\draw[->] (-3,1.5)--(-2.5,1.5);

\begin{scope}[shift={(9.5,0)}]
\fill[yellow!25] (-1,-.5) rectangle (2,0);
\fill[yellow!25] (2,3) rectangle (6,3.5);
\fill[yellow] (-1,0) \north\east\east\east\north\north \east\east\east \south \west\west\south\south\west\west\west\west;
\hatch{-1}{0}\halfhatch{2}{3}
\newcell{2}{2}\newcell{0}{0}\newcell{1}{0}\newcell{2}{1}\newcell{2}{0}\newcell{3}{2}\newcell{4}{2}
\fill[pink!50](-1,-.5) rectangle (2,0);
\fill[pink!50](3,3) rectangle (6,3.5);
\bright{4}{2}\halfbright{4-3}{2-2.5}
\draw[very thin,dotted] (-1,0)  grid (5,3);
 \draw[black, very thick] (-1,-.5)--++(0,.5) \north\east\east\east\north\north -- ++(0,.5);
 \draw[black, very thick] (6,3.5)-- ++(0,-.5)\west\south\west\west\south\south\west  -- ++(0,-0.5);
  \draw[violet, ultra thick,dotted] (3,3.5)-- ++(0,-.5)\west\south\south\west\west\south \west -- ++(0,-0.5);
 \draw (2.5,2.5) node {4}; \draw (2.5-3,2.5-3) node {\footnotesize 4}; 
 \draw (.5,.5) node {2}; \draw (.5+3,.5+3) node {\footnotesize 2};  
 \draw (1.5,.5) node {3}; \draw (1.5+3,.5+3) node {\footnotesize 3}; 
 \draw (2.5,1.5) node {5}; 
 \draw (2.5,.5) node {7}; \draw (2.5+3,.5+3) node {\footnotesize 7}; 
 \draw (3.5,2.5) node {6}; \draw (3.5-3,2.5-3) node {\footnotesize 6}; 
 \draw (4.5,2.5) node {8}; \draw (4.5-3,2.5-3) node {\footnotesize 8}; 
 \draw (-.5,.5) node {1}; \draw (-.5+3,.5+3) node {\footnotesize 1};  
 \draw[->] (0,3.5)--(0,-1);
\draw[->] (-1,3)--(6.5,3);
\draw[->] (-3,1.5)--(-2.5,1.5);
\end{scope}

\begin{scope}[shift={(19,0)}]
\fill[pink!50] (-1,-.5) rectangle (2,0);
\fill[pink!50] (3,3) rectangle (6,3.5);
\fill[pink] (0,0) \north\east\east\north\north \east\east\east \south \west\south\west\south\west\west\west;
\bright{3}{1}
 \draw[very thin,dotted] (0,0)  grid (5,3);
 \draw[black, very thick] (6,3.5)-- ++(0,-.5)\west\south\west\south\west\south\west  -- ++(0,-0.5);
  \draw[violet, ultra thick] (3,3.5)-- ++(0,-.5)\west\south\south\west\west\south \west -- ++(0,-0.5);
  \draw[violet] (1.5,1.5) node {$\alpha$};
 \draw (2.5,2.5) node {1}; \draw (2.5-3,2.5-3) node {\footnotesize 1}; 
 \draw (.5,.5) node {2}; \draw (.5+3,.5+3) node {\footnotesize 2};  
 \draw (1.5,.5) node {3}; \draw (1.5+3,.5+3) node {\footnotesize 3}; 
 \draw (2.5,1.5) node {4}; 
 \draw (2.5,.5) node {5}; \draw (2.5+3,.5+3) node {\footnotesize 5}; 
 \draw (3.5,2.5) node {6}; \draw (3.5-3,2.5-3) node {\footnotesize 6}; 
 \draw (4.5,2.5) node {7}; \draw (4.5-3,2.5-3) node {\footnotesize 7}; 
 \draw (3.5,1.5) node {8}; 
 \draw[->] (0,3.5)--(0,-1);
\draw[->] (-1,3)--(7.5,3);
\draw[->] (-3,1.5)--(-2.5,1.5);
\draw (1,4) node {$\bij(T)=P$};
\end{scope}
\end{scope}

\end{tikzpicture}
\caption{An example of the bijection $\bij$ from Corollary~\ref{cor:bij} with $\alpha=[2,2,0]\in\CS_{3,3}$ and $n=8$. In the computation of $\bij(T)=P$, the inner corners where internal insertion is about to occur are shaded with \textcolor{gray}{gray} diagonal lines, and the newly added cells where the previous insertion terminated are highlighted with a \textcolor{green}{green} grid pattern. }
\label{fig:bij}
\end{figure}
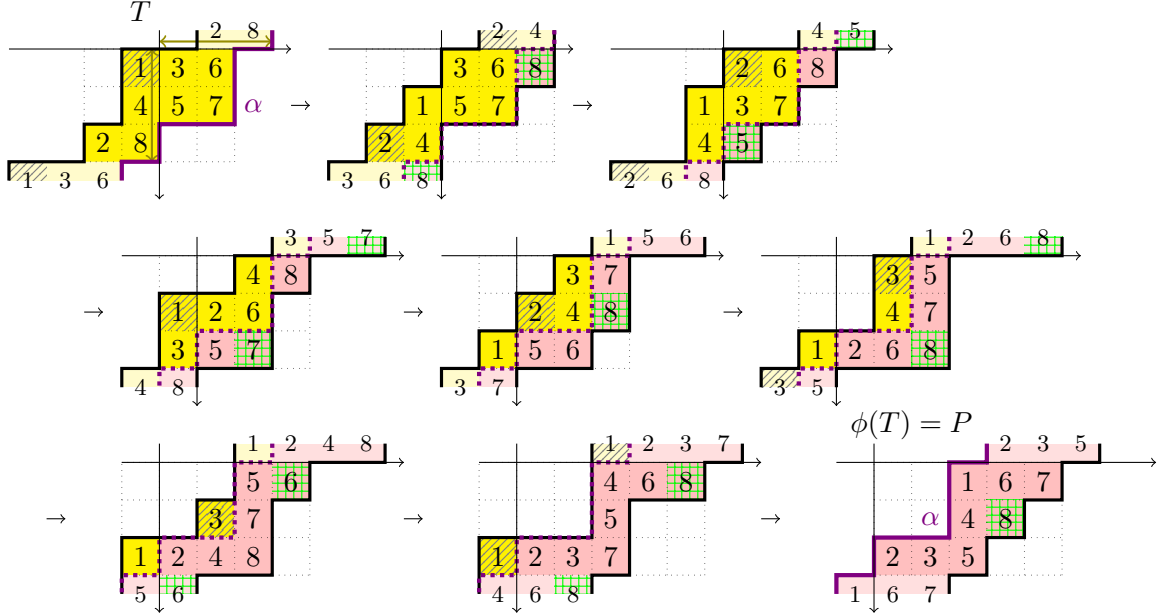 

 In Section~\ref{sec:complementation} we will show that the map $\bij$ is essentially an involution.
We are now ready to use our correspondences from Section~\ref{sec:connections} to translate the bijection $\bij$ into the other settings, namely, walks in simplicies and sequences of states in the TASEP.

\begin{theorem}\label{thm:backward}
Fix $d,L\ge1$. Let $\alpha\in\CS_{d,L}$, let $\x=(x_1,x_2,\dots,x_d)\in\Delta_{d,L}$ where $x_i=\alpha_{i-1}-\alpha_{i}$ for $i\in[d]$, and let $u=0^{x_1}10^{x_2}1\dots 0^{x_d}1$.
There are bijections between the sets in Theorem~\ref{thm:equivalence} and the following sets:
\begin{enumerate}[label={(\^{\alph*})}]
\item\label{it:a''} The set $\SCT_{d,L}^n(\alpha/\cdot)$ of standard cylindric tableaux of period $(d,L)$ with $n$ cells and outer shape~$\alpha$.
\item\label{it:b''} The set of $n$-step walks in $\D_{d,L}$ ending at vertex $\x$.
\item\label{it:c''} The set of $n$-step walks in $\E_{d,L}$ ending at state $u$.
\item\label{it:d''} The set of $n$-step walks in $\N_{d,L}$ ending at state $\bracket{u}$.
\end{enumerate}
\end{theorem}

\begin{proof}
The map $\bij$ from Corollary~\ref{cor:bij} provides a bijection between the set \ref{it:a''} and the set \ref{it:a} from Theorem~\ref{thm:equivalence}. 

The bijections between the sets \ref{it:a''}--\ref{it:d''} are obtained by applying Theorem~\ref{thm:equivalence-oscillating} for $w=-^n$. 
Indeed, elements of the set~\ref{it:A} are oscillating tableaux of type $-^n$ starting at $\alpha$. By reversing the order of the shapes, these are equivalent to oscillating tableaux of type $+^n$ ending at $\alpha$, which correspond to elements of $\SCT_{d,L}^n(\alpha/\cdot)$. Similarly, elements of the sets \ref{it:B}, \ref{it:C} and~\ref{it:D} are oscillating walks in $\D_{d,L}$, $\E_{d,L}$ and $\N_{d,L}$ of type $-^n$ starting at $\x$, $u$ and $\bracket{u}$, respectively. After reversal, these become oscillating walks of type $+^n$ (i.e., $n$-step forward walks) ending at~$\x$, $u$ and $\bracket{u}$, which correspond to the sets \ref{it:b''}, \ref{it:c''} and~\ref{it:d''}, respectively.

It is also not hard to describe the bijections directly. To go from \ref{it:a''} to \ref{it:b''}, let $T\in\SCT_{d,L}^n(\alpha/\cdot)$, and let $i_k\in[d]$ be the row that contains entry $k$, for each $k\in[n]$. Then the image of $T$ is the $n$-step walk in $\D_{d,L}$ with steps $s_{i_1}s_{i_2}\dots s_{i_n}$ ending at~$\x$. 

To obtain the corresponding element in \ref{it:c''}, we first index the $d$ particles in state $u$ so that, for $i\in[d]$, particle $i$ occupies the site corresponding to the $i$th $1$ from the left. Then the vertex preceding $u$ in the walk is the state obtained from $u$ by moving particle $i_n$ to the next site in clockwise direction. The vertex before that is obtained now by moving particle $i_{n-1}$ to the next site in clockwise direction, and so on, until particle $i_{1}$ is moved, at which point the resulting state is the initial vertex of the walk.
\end{proof}

Figure~\ref{fig:backward} gives an example of these bijections when applied to the standard cylindric tableau and the walks in Figure~\ref{fig:33walk}.

\begin{figure}
\centering
\begin{tikzpicture}
\begin{scope}[scale=0.5]
\fill[yellow!25] (-4,-.5) rectangle (-1,0);
\fill[yellow!25] (1,3) rectangle (3,3.5);
\fill[yellow] (-2,0) \north\east\north\north\east\east \east\east  \west\south\south\west\west\south \west\west;
 \draw[very thin,dotted] (-2,0)  grid (2,3);
 \draw[green, very thick] (-4,-.5)--++(0,.5) \east\east \north\east\north\north\east\east -- ++(0,.5);
 \draw[violet, ultra thick] (3,3.5)-- ++(0,-.5)\west\south\south\west\west\south \west -- ++(0,-0.5);
  \draw[thick,<->,olive] (-.2,0)--(-.2,3); \draw[olive] (.3,.6) node {$d$};
 \draw[thick,<->,olive] (0,3.2)-- (3,3.2); \draw[olive] (.5,3.6) node {$L$};
 \draw[purple] (-.5,2.5) node {1}; \draw[purple] (-.5-3,2.5-3) node {\footnotesize 1}; 
 \draw[purple] (-1.5,.5) node {2}; \draw[purple] (-1.5+3,.5+3) node {\footnotesize 2};  
 \draw[purple] (.5,2.5) node {3}; \draw[purple] (.5-3,2.5-3) node {\footnotesize 3}; 
 \draw[purple] (-.5,1.5) node {4}; 
 \draw[purple] (.5,1.5) node {5}; 
 \draw[purple] (1.5,2.5) node {6}; \draw[purple] (1.5-3,2.5-3) node {\footnotesize 6}; 
 \draw[purple] (1.5,1.5) node {7}; 
 \draw[purple] (-.5,.5) node {8}; \draw[purple] (-.5+3,.5+3) node {\footnotesize 8};  
 \draw[violet] (2.5,1.5) node {$\alpha$};
 \draw[->] (0,3.5)--(0,-1);
 \draw[->] (-4,3)--(3.5,3);
\draw (-.5,4) node {$T$};

\draw[purple] (-4.3,2.5) node {$s_1$};
\draw[purple] (-4.3,1.5) node {$s_2$};
\draw[purple] (-4.3,.5) node {$s_3$};

\draw (1,5.5) node {\ref{it:a''}};
\end{scope}

\begin{scope}[shift={(6,-1.2)},scale=.8]
\coordinate (a) at (2,0);
\coordinate (b) at (1,1.732);

\foreach \x in {0,...,3}{
	\foreach \y in {\x,...,3}{
		\fill ($\x*(a)+3*(b)-\y*(b)$) circle (.1);
}}

\fill[red] (0,0) circle (.1);
\fill[red] ($3*(a)$) circle (.1);
\fill[red] ($3*(b)$) circle (.1);
\fill[blue] (a) circle (.1);
\fill[blue] ($2*(a)+(b)$) circle (.1);
\fill[blue] ($2*(b)$) circle (.1);
\fill[green] (b) circle (.1);
\fill[green] ($2*(b)+(a)$) circle (.1);
\fill[green] ($2*(a)$) circle (.1);
\fill[brown] ($(b)+(a)$) circle (.1);

\foreach \x in {0,...,2}{
	\foreach \y in {\x,...,2}{
		\arrows{($\x*(a)+3*(b)-\y*(b)$)}{($\x*(a)+2*(b)-\y*(b)$)}
		\arrows{($\x*(b)+2*(a)-\y*(a)$)}{($\x*(b)+3*(a)-\y*(a)$)}
		\arrows{($\x*(b)+3*(a)-\y*(a)$)}{($\x*(b)+(b)+2*(a)-\y*(a)$)}
}}
\arrowsl{($(b)$)}{($(a)+(b)$)}{$1$}{below}
\arrowsl{($(a)+(b)$)}{($(a)$)}{$2$}{left}
\arrowsl{($(a)$)}{($2*(a)$)}{$3$}{below}
\arrowsl{($2*(a)$)}{($(a)+(b)$)}{$4$}{right}
\arrowsl{($(a)+(b)$)}{($2*(b)$)}{$5$}{right}
\arrowsl{($2*(b)$)}{($2*(b)+(a)$)}{$6$}{below}
\arrowsl{($2*(b)+(a)$)}{($3*(b)$)}{$7$}{right}
\arrowsl{($3*(b)$)}{($2*(b)$)}{$8$}{left}

\draw (0,0) node[left] {$C=(3,0,0)$};
\draw ($3*(a)$) node[right] {$(0,3,0)$};
\draw ($3*(b)$) node[above] {$(0,0,3)$};
\draw[violet] ($2*(b)$) node[left] {$\x=(1,0,2)$};
\draw[violet,ultra thick] ($2*(b)$) circle (.12);

\arrowsl{($3*(b)-2*(a)+(-.5,-2.5)$)}{($3*(b)-(a)+(-.5,-2.5)$)}{$s_1$}{below}
\arrowsl{($3*(b)-2*(a)+(-.5,-2.5)$)}{($4*(b)-3*(a)+(-.5,-2.5)$)}{$s_2$}{left}
\arrowsl{($3*(b)-2*(a)+(-.5,-2.5)$)}{($2*(b)-2*(a)+(-.5,-2.5)$)}{$s_3$}{left}

\draw[purple] (-.5,5) node {$s_1s_3s_1s_2s_2s_1s_2s_3$};
\draw (5.5,5.5) node {\ref{it:b''}};
\draw (5.5,3) node {$\D_{3,3}$};
\end{scope}

\begin{scope}[shift={(4.5,-6)},scale=0.42]
\draw[violet,very thick] (240:\dB) circle (1.05);
\state{180}{\dD}{6}{{2,4,6}}{brown}
\state{180}{\dAA}{6}{{1,2,3}}{red}
\state{180}{\dA}{6}{{3,4,5}}{red}
\state{180}{\dB}{6}{{1,5,6}}{red}
\state{180+60}{\dAA}{6}{{2,3,6}}{blue} 
\state{180+60}{\dA}{6}{{2,4,5}}{blue}
\state{180+60}{\dB}{6}{{1,4,6}}{blue}
\state{180+2*60}{\dAA}{6}{{2,3,5}}{green}
\state{180+2*60}{\dA}{6}{{1,3,6}}{green}
\state{180+2*60}{\dB}{6}{{1,4,5}}{green}
\state{180+3*60}{\dD}{6}{{1,3,5}}{brown}
\state{180+3*60}{\dAA}{6}{{2,3,4}}{red}
\state{180+3*60}{\dA}{6}{{1,2,6}}{red}
\state{180+3*60}{\dB}{6}{{4,5,6}}{red}
\state{180+4*60}{\dAA}{6}{{1,3,4}}{blue}
\state{180+4*60}{\dA}{6}{{1,2,5}}{blue}
\state{180+4*60}{\dB}{6}{{3,5,6}}{blue}
\state{180+5*60}{\dAA}{6}{{1,2,4}}{green}
\state{180+5*60}{\dA}{6}{{3,4,6}}{green}
\state{180+5*60}{\dB}{6}{{2,5,6}}{green}
\foreach \i in {0,2,3,5}{
	\coordinate (aa1) at (180+\i*60:\dAA);
	\coordinate (a1) at (180+\i*60:\dA);
	\coordinate (b1) at (180+\i*60:\dB);
	\coordinate (aa2) at (180+60+\i*60:\dAA);
	\coordinate (a2) at (180+60+\i*60:\dA);
	\coordinate (b2) at (180+60+\i*60:\dB);
	\arrow{(aa1)}{(aa2)}
	\arrow{(a1)}{(a2)}
	\arrow{(b1)}{(b2)}
}
\foreach \i in {1,4}{
	\coordinate (aa1) at (180+\i*60:\dAA);
	\coordinate (a1) at (180+\i*60:\dA);
	\coordinate (b1) at (180+\i*60:\dB);
	\coordinate (aa2) at (180+60+\i*60:\dAA);
	\coordinate (a2) at (180+60+\i*60:\dA);
	\coordinate (b2) at (180+60+\i*60:\dB);
	\arrow{(aa1)}{(aa2)}
	\arrow{(aa1)}{(a2)}
	\arrow{(a1)}{(aa2)}
	\arrow{(a1)}{(b2)}
	\arrow{(b1)}{(a2)}
	\arrow{(b1)}{(b2)}
}
\foreach \i in {2,5}{
	\coordinate (aa1) at (180+\i*60:\dAA);
	\coordinate (a1) at (180+\i*60:\dA);
	\coordinate (b1) at (180+\i*60:\dB);
	\coordinate (d) at (180+60+\i*60-3:\dD+.5);
	\arrow{(aa1)}{(d)}
	\arrow{(a1)}{(d)}
	\arrow{(b1)}{(d)}
}
\foreach \i in {1,4}{
	\coordinate (aa1) at (180+\i*60:\dAA);
	\coordinate (a1) at (180+\i*60:\dA);
	\coordinate (b1) at (180+\i*60:\dB);
	\coordinate (d) at (180-60+\i*60+3:\dD+.5);
	\arrow{(d)}{(aa1)}
	\arrow{(d)}{(a1)}
	\arrow{(d)}{(b1)}
}

\draw[violet] (180+60:\dB)++(1.7,1.65) node {$u=011001$};
\draw[gray] (180+60:\dB)++(180:.6) node {\footnotesize $\mathtt{1}$};
\draw[gray] (180+60:\dB)++(120:.6) node {\footnotesize $\mathtt{2}$};
\draw[gray] (180+60:\dB)++(-60:.6) node {\footnotesize $\mathtt{3}$};
\draw (180:\dB)++(1.1,0) node[right] {$111000$};
\draw[gray] (180:\dB)++(180:.6) node {\footnotesize $\mathtt{1}$};
\draw[gray] (180:\dB)++(120:.6) node {\footnotesize $\mathtt{2}$};
\draw[gray] (180:\dB)++(-120:.6) node {\footnotesize $\mathtt{3}$};
\draw (120:\dB)++(.6,-1.6) node {$110100$};
\draw[gray] (120:\dB)++(180:.6) node {\footnotesize $\mathtt{1}$};
\draw[gray] (120:\dB)++(60:.6) node {\footnotesize $\mathtt{2}$};
\draw[gray] (120:\dB)++(-120:.6) node {\footnotesize $\mathtt{3}$};
\draw (60:\dA)++(1.1,0) node[right] {$101100$};
\draw[gray] (60:\dA)++(120:.6) node {\footnotesize $\mathtt{1}$};
\draw[gray] (60:\dA)++(60:.6) node {\footnotesize $\mathtt{2}$};
\draw[gray] (60:\dA)++(-120:.6) node {\footnotesize $\mathtt{3}$};
\draw (0:\dD)++(1.3,0) node[right] {$101010$};
\draw[gray] (0:\dD)++(120:.6) node {\footnotesize $\mathtt{1}$};
\draw[gray] (0:\dD)++(0:.6) node {\footnotesize $\mathtt{2}$};
\draw[gray] (0:\dD)++(-120:.6) node {\footnotesize $\mathtt{3}$};
\draw (300:\dB)++(.9,-.3) node[right] {$101001$};
\draw[gray] (300:\dB)++(120:.6) node {\footnotesize $\mathtt{1}$};
\draw[gray] (300:\dB)++(-60:.6) node {\footnotesize $\mathtt{2}$};
\draw[gray] (300:\dB)++(-120:.6) node {\footnotesize $\mathtt{3}$};
\draw (240:\dA)++(1.1,0) node[right] {$100101$};
\draw[gray] (240:\dA)++(60:.6) node {\footnotesize $\mathtt{1}$};
\draw[gray] (240:\dA)++(-60:.6) node {\footnotesize $\mathtt{2}$};
\draw[gray] (240:\dA)++(-120:.6) node {\footnotesize $\mathtt{3}$};
\draw (180:\dD)++(-1.3,0) node[left] {$010101$};
\draw[gray] (180:\dD)++(60:.6) node {\footnotesize $\mathtt{1}$};
\draw[gray] (180:\dD)++(-60:.6) node {\footnotesize $\mathtt{2}$};
\draw[gray] (180:\dD)++(180:.6) node {\footnotesize $\mathtt{3}$};
\draw (120:\dA)++(1.1,0) node[right] {$010011$};
\draw[gray] (120:\dA)++(0:.6) node {\footnotesize $\mathtt{1}$};
\draw[gray] (120:\dA)++(-60:.6) node {\footnotesize $\mathtt{2}$};
\draw[gray] (120:\dA)++(180:.6) node {\footnotesize $\mathtt{3}$};

\arrowl{(180-60:\dA)}{(180-3:\dD+.5)}{$1$}{above}
\arrowl{(180+3:\dD+.5)}{(180+60:\dA)}{$2$}{below}
\arrowl{(180+60:\dA)}{(180+2*60:\dB)}{$3$}{below  right=-1mm}
\arrowl{(180+2*60:\dB)}{(180+3*60-3:\dD+.5)}{$4$}{below}
\arrowl{(180+3*60+3:\dD+.5)}{(180+4*60:\dA)}{$5$}{above}
\arrowl{(180+4*60:\dA)}{(180+5*60:\dB)}{$6$}{above left=-1mm}
\arrowl{(180+5*60:\dB)}{(180:\dB)}{$7$}{left}
\arrowl{(180:\dB)}{(180+60:\dB)}{$8$}{left}

\draw (-13,7) node {\ref{it:c''}};
\draw (11,6) node {$\E_{3,3}$};
\end{scope}
\end{tikzpicture}
\caption{An example of the bijections from Theorem~\ref{thm:backward} when applied to the objects in Figure~\ref{fig:33walk}: 
the standard cylindric tableau $T=\bij^{-1}(P)$ with outer shape $\alpha=[2,2,0]\in\CS_{3,3}$ (top left), the corresponding walk in $\D_{3,3}$ ending at $\x=(1,0,2)$ (top right), and the corresponding walk in $\E_{3,3}$ ending at $u=011001$ (bottom).}
\label{fig:backward}
\end{figure}
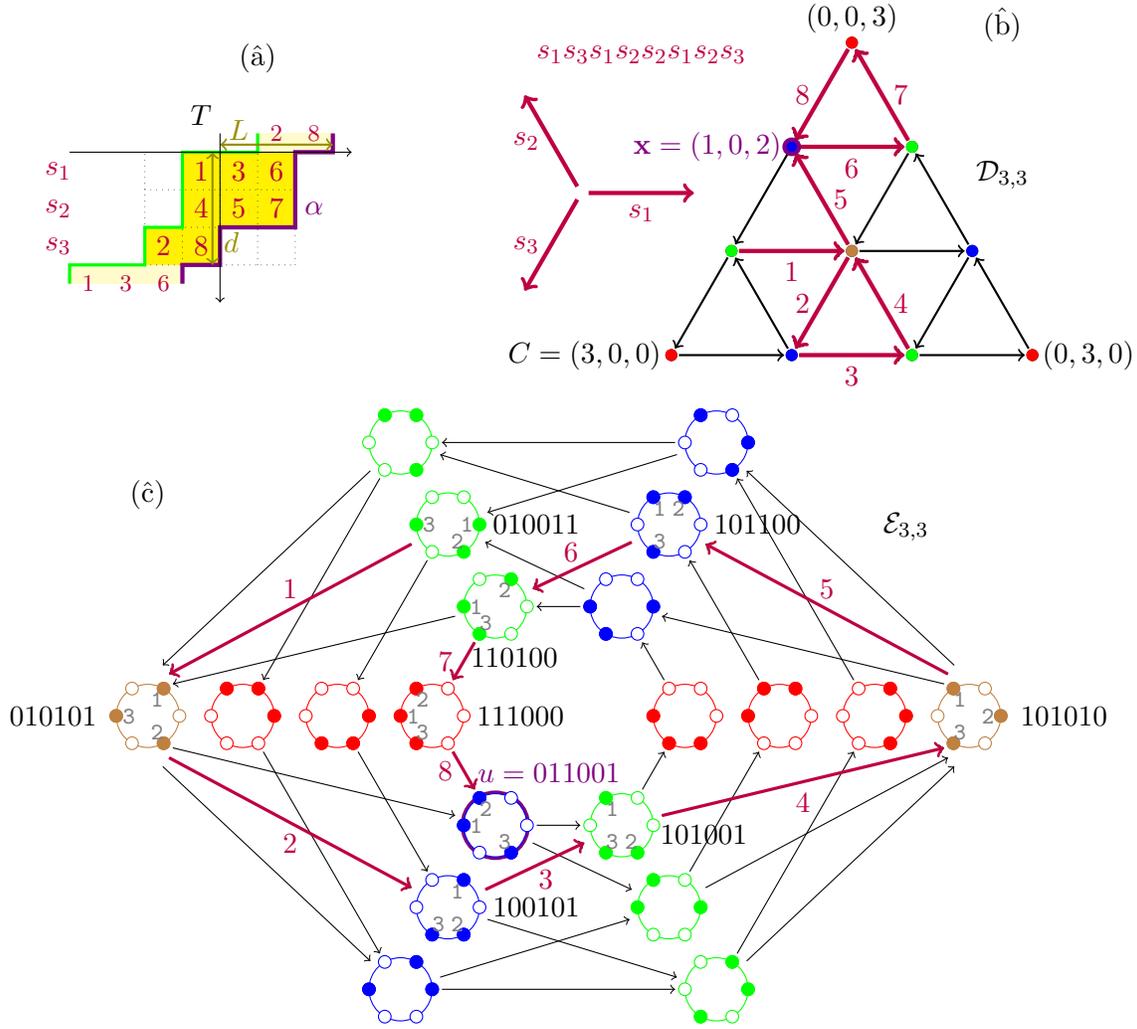 

The resulting bijection between the sets~\ref{it:b} (from Theorem~\ref{thm:equivalence}) and~\ref{it:b''} (from Theorem~\ref{thm:backward}) gives an alternative proof Theorem~\ref{thm:courtiel_bijections_2021}.
Our bijection is obtained by translating the map $\bij$ from Corollary~\ref{cor:bij} (which is a special case of cylindric RS when both tableaux are equal) in terms of simplicial walks.

By comparison, Courtiel, Elvey Price and Marcovici's construction repeatedly applies certain flips to adjacent steps and to the last step of the walk in $\D_{d,L}$ starting at $\x$, until all the forward steps have been switched into backward steps. The authors also give an alternative description of this bijection in \cite[Sec.~2.3]{courtiel_bijections_2021} in terms of tilings of a tilted square with labeled tiles that must follow certain rules that emulate the allowed flips in the walks. 

\section{Cylindric growth diagrams}\label{sec:growth} 

Growth diagrams were introduced by Fomin~\cite{fomin_generalized_1986,fomin_duality_1994,fomin_schensted_1995} as an alternative description of the RS correspondence, in order to generalize it to differential posets~\cite{stanley_differential_1988}. We refer the reader to \cite[Sec.~7.13]{stanley_enumerative_1999} and \cite[Sec.~5.2]{sagan_symmetric_2001} for expositions on growth diagrams. Roby showed in \cite[Ch.~3]{roby_applications_1991} that Sagan and Stanley's analogue of the RS algorithm for skew tableaux~\cite{sagan_robinson-schensted_1990} also has a natural description in terms of growth diagrams.

The goal of this section is to extend the growth diagram approach to the cylindric case. We will define cylindric growth diagrams and use them to provide a new description of the bijections $\CRS$ and $\CRSK$ from Theorems~\ref{thm:CRS} and~\ref{thm:CRSK}, i.e., Neyman's cylindric version of the RSK algorithm. As in the case of skew tableaux, this description makes the construction more natural, and it elucidates the symmetry obtained when switching the two tableaux in the pair.

\subsection{Growth diagrams for SCT}\label{sec:growthSCT} 

We start with the case of standard cylindric tableaux, which is simpler and more similar to the classical case. In Subsection~\ref{sec:growthSSCT} we will explain how to deal with repeated entries. 

Fix $\alpha,\beta\in\CS_{d,L}$ and $n,m\ge0$.
Given $T\in\SCT_{d,L}(\alpha/\mu)$ and $U\in\SCT_{d,L}(\beta/\mu)$, where $\mu\subseteq\alpha,\beta$ satisfies $|\alpha/\mu|=n$ and $|\beta/\mu|=m$, we will describe how to compute a pair $(\tP,\tQ)$ using cylindric growth diagrams, and then show that $(\tP,\tQ)=\CRS(T,U)$, where $\CRS$ is the bijection defined in Theorem~\ref{thm:CRS}.

We draw growth diagrams as $m\times n$ rectangular grids whose vertices have integer coordinates $(x,y)$ for $0\le x\le m$ and $0\le y\le n$. 
We use the usual orientation of Cartesian coordinates for the vertices of growth diagrams, rather than the rotated convention that we used for cells of the cylinder $\C_{d,L}$ in Section~\ref{sec:SCT}. Each vertex of the grid will be labeled with an element of $\CS_{d,L}$.

Before we proceed with the description of the labels, let us highlight some differences between our approach and the usual setup for growth diagrams \cite{roby_applications_1991,sagan_symmetric_2001,stanley_enumerative_1999}.
Unlike Young's lattice, the lattice $(\CS_{d,L},\subseteq)$ of cylindric shapes ordered by containment is not an {\em $r$-differential poset} as defined in \cite[Def.~1.1]{stanley_differential_1988}, since it does not have a minimal element $\hat{0}$. 
However, if we disregard this condition and relax the usual requirement that $r$ be a positive integer, then $(\CS_{d,L},\subseteq)$ satisfies the other parts of this definition with $r=0$. Specifically,
\begin{itemize}
\item[(D1)] it is locally finite and it has a rank function (letting the rank of $\rho$ be $|\rho|$, as defined in Section~\ref{sec:SCT});
\item[(D2)] for any $\rho,\nu\in\CS_{d,L}$ with $\rho\neq\nu$, the number of elements covered by both $\rho$ and $\nu$ equals the number of elements that cover both $\rho$ and $\nu$ (this number is always $0$ or $1$); and
\item[(D3)] every $\rho\in\CS_{d,L}$ covers the same number of elements that it is covered by (this number equals the number of insertion corners $c(\rho)$, see Equation~\eqref{eq:corners}). 
\end{itemize}

For comparison, condition (D3) for an $r$-differential poset states that if an element covers exactly $k$ elements (for some $k$), then it must be covered by exactly $k+r$ elements. One consequence of setting $r=0$ in this rule is that, unlike usual growth diagrams, cylindric  growth diagrams will never have crosses in the squares of the $m\times n$ grid. Another difference is that, for any edge in a cylindric growth diagram (in the case of standard tableaux) with endpoints labeled $\nu$ and $\rho$, with $\rho$ above or to the right of $\nu$, the shape $\rho$ is obtained from $\nu$ by adding exactly one cell; in particular, $\nu\neq\rho$.

Using Fomin's terminology~\cite{fomin_duality_1994,fomin_schensted_1995}, the corresponding {\em graded graph} is $\CSg_{d,L}$. The case $r=0$ for graded graphs was in fact considered in \cite[Eq.~(2.2.9)]{fomin_duality_1994}.

Now we are ready to describe the labels of the vertices of the $m\times n$ grid. We start by labeling the vertices in the left boundary  (i.e., those with coordinates of the form $(0,y)$), from bottom to top, by the sequence $\ssh(T)$, as defined in Equation~\eqref{eq:ssh}. 
Similarly, we label the vertices in the bottom boundary (i.e., those of the form $(x,0)$), from left to right, by the sequence $\ssh(U)$.
In particular, vertex $(0,0)$ has label $\mu$, vertex $(0,n)$ has label $\alpha$, and vertex $(m,0)$ has label $\beta$. See Figure~\ref{fig:growth} for an example.

\begin{figure}[htb]
\centering
\begin{tikzpicture}[scale=1]
\draw[dotted] (0,0) grid (4,5);
\draw[blue] (0,0) node {$[1 1 0]$}; \draw[blue] (-.3,0) node[left] {$\mu=$}; 
\draw (1,0) node {$[1 1 1]$};
\draw (2,0) node {$[2 1 1]$}; 
\draw (3,0) node {$[2 2 1]$};
\draw[red] (4,0) node {$[2 2 2]$}; \draw[red] (4.3,0) node[right] {$=\beta$}; 
\draw (0,1) node {$[2 1 0]$};
\draw (1,1) node {$[2 1 1]$};
\draw (2,1) node {$[2 2 1]$}; 
\draw (3,1) node {$[2 2 2]$};
\draw (4,1) node[purple] {$[3 2 2]$}; 
\draw (0,2) node {$[2 1 1]$};
\draw (1,2) node {$[3 1 1]$};
\draw (2,2) node {$[3 2 1]$}; 
\draw (3,2) node[purple] {$[3 2 2]$};
\draw (4,2) node[purple] {$[3 3 2]$}; 
\draw (0,3) node {$[3 1 1]$};
\draw (1,3) node {$[3 2 1]$};
\draw (2,3) node {$[3 2 2]$}; 
\draw (3,3) node[purple] {$[4 2 2]$};
\draw (4,3) node {$[4 3 2]$};
\draw (0,4) node {$[3 2 1]$};
\draw (1,4) node[purple] {$[3 2 2]$};
\draw (2,4) node[purple] {$[4 2 2]$}; 
\draw (3,4) node[purple] {$[4 3 2]$};
\draw (4,4) node {$[4 3 3]$};
\draw[violet] (0,5) node {$[3 3 1]$}; \draw[violet] (-.3,5) node[left] {$\alpha=$}; 
\draw (1,5) node[purple] {$[3 3 2]$};
\draw (2,5) node {$[4 3 2]$}; 
\draw (3,5) node {$[4 3 3]$};
\draw[brown] (4,5) node {$[5 3 3]$}; \draw[brown] (4.3,5) node[right] {$=\lambda$}; 

\draw[->] (1.5,-.4)--node[below]{$U$} (2.5,-.4);
\draw[->] (-.7,2)--node[left]{$T$} (-.7,3);
\draw[->] (1.5,5.4)--node[above]{$Q$} (2.5,5.4);
\draw[->] (4.7,2)--node[right]{$P$} (4.7,3);

\begin{scope}[scale=.5,shift={(-8,6)}]
 \fill[yellow!25] (-1,-.5)++(0,.5) \east \north\east\north\north\east -- ++(0,.5)-- ++(1,0) -- ++(0,-.5) \south\south\west\west\south -- ++(0,-0.5) -- ++(-2,0);
 \fill[yellow] (0,0) \north\east\north\north\east \east  \south\south\west\west\south \west;
 \draw[very thin,dotted] (0,0)  grid (3,3);
 \draw[thick,<->,olive] (-.2,0)-- (-.2,3);
 \draw[thick,<->,olive] (0,3.2)-- (2,3.2); 
 \draw[blue, very thick] (-1,-.5)--++(0,.5) \east \north\east\north\north\east -- ++(0,.5);
 \draw[violet, very thick] (3,3.5)-- ++(0,-.5) \south\south\west\west\south -- ++(0,-0.5);
 \draw (1.5,2.5) node {1}; \draw (1.5-2,2.5-3) node {\footnotesize 1}; 
 \draw (.5,.5) node {2}; \draw (.5+2,.5+3) node {\footnotesize 2};  
 \draw (2.5,2.5) node {3}; \draw (2.5-2,2.5-3) node {\footnotesize 3}; 
 \draw (1.5,1.5) node {4}; 
 \draw (2.5,1.5) node {5}; 
 \draw[blue] (.5,2.5) node {$\mu$};
 \draw[violet] (3.5,1.5) node {$\alpha$};
 \draw[thick,->] (0,3.5)--(0,-1);
\draw[thick,->] (-1,3)--(3.5,3);
\draw (1.5,4) node {$T$};
\end{scope}

\begin{scope}[scale=.5,shift={(-8,0)}]
 \fill[yellow!25] (-1,-.5)++(0,.5) \east \north\east\north\north\east -- ++(0,.5)-- ++(2,0) -- ++(0,-.5) \west\west\south\south\south\west\west -- ++(0,-0.5) -- ++(-1,0);
 \fill[yellow] (0,0) \north\east\north\north\east \east  \west\south\south\south\west\west;
 \draw[very thin,dotted] (0,0)  grid (3,3);
 \draw[blue, very thick] (-1,-.5)--++(0,.5) \east \north\east\north\north\east -- ++(0,.5);
 \draw[red, very thick] (4,3.5)-- ++(0,-.5) \west\west\south\south\south\west\west-- ++(0,-0.5);
 \draw (.5,.5) node {1}; \draw (.5+2,.5+3) node {\footnotesize 1}; 
 \draw (1.5,2.5) node {2}; \draw (1.5-2,2.5-3) node {\footnotesize 2};  
 \draw (1.5,1.5) node {3}; 
 \draw (1.5,.5) node {4}; \draw (1.5+2,.5+3) node {\footnotesize 4}; 
 \draw[blue] (.5,2.5) node {$\mu$};
 \draw[red] (2.5,1.5) node {$\beta$};
 \draw[thick,->] (0,3.5)--(0,-1);
\draw[thick,->] (-1,3)--(4.5,3);
\draw (1.5,4) node {$U$};
\end{scope}

\begin{scope}[scale=.5,shift={(13,6)}]
 \fill[yellow!25] (0,-.5)++(0,.5) \east\east\north\north\north\east\east -- ++(0,.5)-- ++(1,0) -- ++(0,-.5) \south\west\west\south\south -- ++(0,-0.5) -- ++(-3,0);
 \fill[yellow] (0,0) \east\east\north\north\north\east\east \east  \south\west\west\south\south\west \west\west;
 \draw[very thin,dotted] (0,0)  grid (5,3);
 \draw[red, very thick] (0,-.5)--++(0,.5) \east\east\north\north\north\east\east -- ++(0,.5);
 \draw[brown, very thick] (5,3.5)-- ++(0,-.5) \south\west\west\south\south -- ++(0,-0.5);
 \draw (3.5,2.5) node {3}; \draw (3.5-2,2.5-3) node {\footnotesize 3}; 
 \draw (4.5,2.5) node {5}; \draw (4.5-2,2.5-3) node {\footnotesize 5};  
 \draw (2.5,2.5) node {1}; \draw (2.5-2,2.5-3) node {\footnotesize 1}; 
 \draw (2.5,.5) node {4}; \draw (2.5+2,.5+3) node {\footnotesize 4};  
 \draw (2.5,1.5) node {2}; 
 \draw[thick,->] (0,3.5)--(0,-1);
\draw[thick,->] (-.5,3)--(5.5,3);
\draw (1.5,4) node {$P$};
\draw[red] (1.5,1.5) node {$\beta$};
\draw[brown] (3.5,.5) node {$\lambda$};
\end{scope}

\begin{scope}[scale=.5,shift={(13,0)}]
 \fill[yellow!25] (1,-.5)++(0,.5) \north\east\east\north\north -- ++(0,.5)-- ++(2,0) -- ++(0,-.5) \south\west\west\south\south -- ++(0,-0.5) -- ++(-2,0);
 \fill[yellow] (1,0)\north\east\east\north\north \east\east \south\west\west\south\south  \west\west;
 \draw[very thin,dotted] (0,0)  grid (5,3);
 \draw[brown, very thick] (5,3.5)-- ++(0,-.5)\south\west\west\south\south-- ++(0,-0.5);
 \draw[violet, very thick] (1,-.5)--++(0,.5)  \north\east\east\north\north -- ++(0,.5);
 \draw (1.5,.5) node {1}; \draw (1.5+2,.5+3) node {\footnotesize 1}; 
 \draw (4.5,2.5) node {4}; \draw (4.5-2,2.5-3) node {\footnotesize 4};  
 \draw (3.5,2.5) node {2}; \draw (3.5-2,2.5-3) node {\footnotesize 2}; 
 \draw (2.5,.5) node {3}; \draw (2.5+2,.5+3) node {\footnotesize 3};  
 \draw[thick,->] (0,3.5)--(0,-1);
\draw[thick,->] (-.5,3)--(5.5,3);
\draw (1.5,4) node {$Q$};
\draw[violet] (2.5,1.5) node {$\alpha$};
\draw[brown] (3.5,.5) node {$\lambda$};
\end{scope}
\end{tikzpicture}
\caption{The computation of $\CRS(T,U)=(P,Q)$ using growth diagrams, for the same example as in Figure~\ref{fig:CRS}. 
The shapes colored in \textcolor{purple}{purple} represent the oscillating tableau from Example~\ref{ex:OCT}.
}
\label{fig:growth}
\end{figure} 

To compute the labels of the remaining vertices, we use the following {\em forward local rule}, which describes, for each unit square of the grid, the label of the upper-right vertex in terms of the labels of the other three vertices. 
Let the diagram below denote the labels of the vertices of a unit square, and suppose that $\rholl$, $\rhoul$ and $\rholr$ have already been computed.
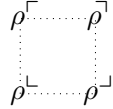
\begin{equation}\label{eq:square}
\begin{tikzpicture}[scale=1]
\draw[dotted] (0,0) rectangle (1,1);
\draw (0.07,0) node{$\rholl$};
\draw (0.07,1) node{$\rhoul$};
\draw (1.04,0) node{$\rholr$};
\draw (1.04,1) node{$\rhour$};
\end{tikzpicture}
\end{equation}
Then $\rhour$ is computed as follows:
\begin{itemize}
\item[(F1)] If $\rhoul\neq \rholr$, let $\rhour=\rhoul\cup \rholr$, i.e., $\rhour_i=\max\{\rhoul_i,\rholr_i\}$ for all $i$. 
\item[(F2)] If $\rhoul=\rholr$, then this shape is obtained from $\rholl$ by adding a cell to some row~$i$.
Let $\rhour$ be the shape obtained from $\rhoul=\rholr$ by adding 
a cell to row~$i+1$ (with indices modulo $d$).
\end{itemize}

Once all the labels of the $m\times n$ grid have been computed, let $\lambda$ be the label of vertex $(m,n)$, and note that  $\lambda\supseteq\alpha,\beta$, $|\lambda/\beta|=n$, and $|\lambda/\alpha|=m$.
Let $\tP\in\SCT_{d,L}(\lambda/\beta)$ be such that $\ssh(\tP)$ is the sequence of the labels of the right boundary of the grid (i.e., vertices of the form $(m,y)$), from bottom to top. Similarly, let $\tQ\in\SCT_{d,L}(\lambda/\alpha)$ be such that $\ssh(\tQ)$ is the sequence of labels of the top boundary (i.e., the vertices of the form $(x,n)$), from left to right. 
See Figure~\ref{fig:growth} for an example, where the pair $(\tP,\tQ)$ is denoted by $(P,Q)$ since we will soon show that it coincides with the image of $(T,U)$ under $\CRS$.

The forward local rule can be inverted in order to determine the label of the lower-left vertex of a unit square of the grid in terms of the labels of the other three vertices. Assuming that $\rhoul$, $\rholr$ and $\rhour$ have already been computed, the following  {\em backward local rule} determines $\rholl$:
\begin{itemize}
\item[(B1)] If $\rhoul\neq \rholr$, let $\rholl=\rhoul\cap \rholr$, i.e., $\rholl_i=\min\{\rhoul_i,\rholr_i\}$ for all $i$. 
\item[(B2)] If $\rhoul=\rholr$, then this shape is obtained from $\rhour$ by 
removing a cell from some row~$i+1$.
Let $\rholl$ be the shape obtained from $\rhoul=\rholr$ by 
removing a cell from row~$i$ (with indices modulo~$d$).
\end{itemize}
This allows us to compute the inverse map that recovers the pair $(T,U)$ from the pair $(\tP,\tQ)\in\SCT_{d,L}(\lambda/\beta)\times\SCT_{d,L}(\lambda/\alpha)$.

We will show that the map $(T,U)\mapsto(\tP,\tQ)$ described above in terms of growth diagrams is equivalent to the map $\CRS$ described in the proof of Theorem~\ref{thm:CRS} in terms of iterated insertion operations.
For $0\le x\le m$ and $0\le y\le n$, let $\rho^{(x,y)}\in\CS_{d,L}$ denote the label of vertex $(x,y)$ in the above growth diagram starting from the pair $(T,U)$.
For $0\le k\le m$, let $(P_k,Q_k)$ be the sequence of pairs of tableaux built in the computation of $\CRS(T,U)=(P,Q)$ in the proof of Theorem~\ref{thm:CRS}. The next lemma relates the two constructions. Recall the map $\ssh$ defined in Equation~\eqref{eq:ssh}.

\begin{lemma} \label{lem:Pk}
For each $0\le k\le m$, we have $\ssh(P_k)=(\rho^{(k,0)},\rho^{(k,1)},\dots,\rho^{(k,n)})$.
 Additionally, $\ssh(Q)=(\rho^{(0,n)},\rho^{(1,n)},\dots,\rho^{(m,n)})$.
\end{lemma}

\begin{proof}
We will prove the first statement by induction on $k$. It trivially holds for $k=0$, since, by construction, the left boundary of the growth diagram is labeled by $\ssh(T)=\ssh(P_0)$. 
Now let $k\in[m]$, and suppose that $\ssh(P_{k-1})=(\rho^{(k-1,0)},\rho^{(k-1,1)},\dots,\rho^{(k-1,n)})$. Recall from the proof of Theorem~\ref{thm:CRS} that $P_k=\ins_i(P_{k-1})$ (as in Definition~\ref{def:insertion}), where $i\in[d]$ is the row of $U$ containing entry $k$. Let $\tP_k$ be such that $\ssh(\tP_k)=(\rho^{(k,1)},\rho^{(k,2)},\dots,\rho^{(k,n)})$. We want to show that $P_k=\tP_k$.

Since the bottom boundary of the growth diagram is labeled by $\ssh(U)$, the shape $\rho^{(k,0)}$ (i.e., the inner shape of $\tilde{P}_k$) is obtained from $\rho^{(k-1,0)}$ (i.e., the inner shape of $P_{k-1}$) by adding a cell to row $i$. We denote this by 
$\rho^{(k-1,0)}\add{i}\rho^{(k,0)}$.
For each $\ell\in[n]$, let $r_\ell\in[d]$ the row of $P_{k-1}$ containing $\ell$, so that 
$\rho^{(k-1,0)}\add{r_1}\rho^{(k-1,1)}\add{r_2}\cdots\add{r_n}\rho^{(k-1,n)}$. 
If $r_1\neq i$, then $\rho^{(k-1,1)}\neq\rho^{(k,0)}$, so case (F1) of the forward local rule is applied when computing $\rho^{(k,1)}$, and we have $\rho^{(k,0)}\add{r_1}\rho^{(k,1)}$ and $\rho^{(k-1,1)}\add{i}\rho^{(k,1)}$.
More generally, letting $a$ be the smallest (equivalently, leftmost) entry in row $i$ of $P_{k-1}$,
case (F1) is used when computing $\rho^{(k,\ell)}$ for $\ell<a$, since $r_\ell\neq i$, and we get $\rho^{(k,0)}\add{r_1}\rho^{(k,1)}\add{r_2}\cdots\add{r_{a-1}}\rho^{(k,a-1)}$. This means that entries $1,2,\dots,a-1$ are in the same rows in $\tP_k$ as in $P_{k-1}$. Additionally, for $\ell<a$, we have $\rho^{(k-1,\ell)}\add{i}\rho^{(k,\ell)}$. See the diagram in Figure~\ref{fig:Pkgrowth}.

\begin{figure}[htb]
\centering
\begin{tikzpicture}
[scale=1,decoration={
    markings,
    mark=at position .5 with {\arrow{>}}}
    ] 
\foreach \y in {0,1,2.5,3.5,4.5} {
	\draw[fill] (0,\y) circle (.05); \draw[fill] (1,\y) circle (.05); 
}
\draw[postaction={decorate},shorten <=1mm,shorten >=1mm,purple] (0,0) node[left,scale=.8,black] {$(k-1,0)$} -- node[left,scale=.9]{$r_1$} (0,1) node[left,scale=.8,black] {$(k-1,1)$};
\draw[postaction={decorate},shorten <=1mm,shorten >=1mm,purple] (1,0) node[right,scale=.8,black] {$(k,0)$} -- node[right,scale=.9]{$r_1$} (1,1) node[right,scale=.8,black] {$(k,1)$};
\draw[postaction={decorate},shorten <=1mm,shorten >=1mm,purple] (0,0) -- node[above,scale=.9]{$i$} (1,0);
\draw[postaction={decorate},shorten <=1mm,shorten >=1mm,purple] (0,1) -- node[above,scale=.9]{$i$} (1,1);
\draw[postaction={decorate},shorten <=1mm,shorten >=1mm,purple] (0,1) -- node[left,scale=.9]{$r_2$} (0,1.6);
\draw[postaction={decorate},shorten <=1mm,shorten >=1mm,purple] (1,1) -- node[right,scale=.9]{$r_2$} (1,1.6);
\draw (.5,1.8) node {$\vdots$};
\draw[postaction={decorate},shorten <=1mm,shorten >=1mm,purple] (0,1.9) -- node[left,scale=.9]{$r_{a-1}$} (0,2.5);
\draw[postaction={decorate},shorten <=1mm,shorten >=1mm,purple] (1,1.9) -- node[right,scale=.9]{$r_{a-1}$} (1,2.5);
\draw[postaction={decorate},shorten <=1mm,shorten >=1mm,purple] (0,2.5) node[left,scale=.8,black] {$(k-1,a-1)$} -- node[left,scale=.9]{$r_a=i$} (0,3.5) node[left,scale=.8,black] {$(k-1,a)$};
\draw[postaction={decorate},shorten <=1mm,shorten >=1mm,purple] (1,2.5) node[right,scale=.8,black] {$(k,a-1)$} -- node[right,scale=.9]{$i+1$} (1,3.5) node[right,scale=.8,black] {$(k,a)$};
\draw[postaction={decorate},shorten <=1mm,shorten >=1mm,purple] (0,2.5) -- node[above,scale=.9]{$i$} (1,2.5);
\draw[postaction={decorate},shorten <=1mm,shorten >=1mm,purple] (0,3.5) -- node[above,scale=.9]{$i+1$} (1,3.5);
\draw[postaction={decorate},shorten <=1mm,shorten >=1mm,purple] (0,3.5) -- node[left,scale=.9]{$r_{a+1}$} (0,4.5) node[left,scale=.8,black] {$(k-1,a+1)$};
\draw[postaction={decorate},shorten <=1mm,shorten >=1mm,purple] (1,3.5) -- node[right,scale=.9]{$r_{a+1}$} (1,4.5) node[right,scale=.8,black] {$(k,a+1)$};
\draw[postaction={decorate},shorten <=1mm,shorten >=1mm,purple] (0,4.5) -- node[above,scale=.9]{$i+1$} (1,4.5);
\draw[postaction={decorate},shorten <=1mm,shorten >=1mm,purple] (0,4.5) -- (0,5.2);
\draw[postaction={decorate},shorten <=1mm,shorten >=1mm,purple] (1,4.5) -- (1,5.2);
\draw (.5,5.2) node {$\vdots$};
\draw (.5,.6) node[scale=.7,brown]{(F1)};
\draw (.5,2.2) node[scale=.7,brown]{(F1)};
\draw (.5,3.1) node[scale=.7,brown]{(F2)};
\draw (.5,4.1) node[scale=.7,brown]{(F1)};
\end{tikzpicture}
\caption{The computation of the $k$th column of the growth diagram, illustrating the proof of Lemma~\ref{lem:Pk}. We write the coordinates $(x,y)$ of each vertex instead of the label $\rho^{(x,y)}$.}
\label{fig:Pkgrowth}
\end{figure} 

Since $r_a=i$, both $\rho^{(k-1,a)}$ and $\rho^{(k,a-1)}$ are obtained from $\rho^{(k-1,a-1)}$ by adding a cell to row $i$, so $\rho^{(k-1,a)}=\rho^{(k,a-1)}$. Thus, applying case (F2) of the forward local rule, $\rho^{(k,a)}$ is obtained from $\rho^{(k-1,a)}=\rho^{(k,a-1)}$ by adding a cell to row $i+1$ (with row indices modulo $d$).
Consequently, entry $a$ is placed in row $i+1$ of $\tP_k$.
If there are entries larger than $a$ in row $i+1$ of $P_{k-1}$, let $a'$ be the smallest such entry. 
Since $r_\ell\neq i+1$ for $a<\ell<a'$, case (F1) is used to compute $\rho^{(k,a)}\add{r_{a+1}}\rho^{(k,a+1)}\add{r_{a+2}}\cdots\add{r_{a'-1}}\rho^{(k,a'-1)}$, which causes entries $a+1,a+2,\dots,a'-1$ to be in the same rows in $\tP_k$ as in $P_{k-1}$. Additionally, for $a<\ell<a'$, we have $\rho^{(k-1,\ell)}\add{i+1}\rho^{(k,\ell)}$.

Since $r_{a'}=i+1$, case (F2) is used when computing $\rho^{(k,a')}$, which is obtained from $\rho^{(k-1,a')}=\rho^{(k,a'-1)}$ by adding a cell to row $i+2$. Thus, entry $a'$ is placed in row $i+2$ of $\tP_k$. If there are entries larger than $a'$ in row $i+2$ of $P_{k-1}$, we bump be the smallest such entry and repeat the process. Eventually, the bumped entry (call it $a''$) is placed in a row (call it $t$) where $P_{k-1}$ contains no larger entries. Then case (F1) is used to compute $\rho^{(k,a'')}\add{r_{a''+1}}\rho^{(k,a''+1)}\add{r_{a''+2}}\cdots\add{r_{n}}\rho^{(k,n)}$, causing entries $a''+1,a''+2,\dots,n$ to be in the same rows in $\tP_k$ as in $P_{k-1}$. 
Additionally, for $a''<\ell\le n$, we have $\rho^{(k-1,\ell)}\add{t}\rho^{(k,\ell)}$.

By comparing the above description of $\tP_k$ with the process of internal row insertion at row $i$ of $P_{k-1}$, described in Definition~\ref{def:insertion}, we see that each $\ell\in[n]$ is in the same row in both $\tP_k$ and $\ins_i(P_{k-1})$. Since these tableaux have the same inner shape, it follows that $\tP_k=\ins_i(P_{k-1})=P_k$.

Finally, we have $\rho^{(k-1,n)}\add{t}\rho^{(k,n)}$, where $t$ is the row where the above insertion process terminates. 
In the proof of Theorem~\ref{thm:CRS}, this is also the row where a cell is added to $Q_{k-1}$ in order to obtain $Q_k$. Since this holds for all $k\in[m]$, we deduce that $\ssh(Q)=(\rho^{(0,n)},\rho^{(1,n)},\dots,\rho^{(m,n)})$.
\end{proof}

\begin{theorem}\label{thm:CRS=growth},
Let $(T,U)\in\SCT_{d,L}(\alpha/\mu)\times \SCT_{d,L}(\beta/\mu)$. Suppose that $(T,U)\mapsto (\tP,\tQ)$ via the above growth diagram construction, and that $\CRS(T,U)=(P,Q)$, where $\CRS$ is the bijection from Theorem~\ref{thm:CRS}.
Then 
$(\tP,\tQ)=(P,Q)$.
\end{theorem}

\begin{proof}
Consider the cylindric growth diagram with whose left and bottom boundaries are $\ssh(T)$ and $\ssh(U)$, respectively, and the remaining labels are computed using the forward local rule. By Lemma~\ref{lem:Pk} with $k=m$, the right boundary of the diagram is $\ssh(P_m)=\ssh(P)$, and the top boundary is $\ssh(Q)$. We deduce that $(\tP,\tQ)=(P,Q)$.
\end{proof}

\subsection{Growth diagrams for SSCT} \label{sec:growthSSCT}
Next we generalize the above construction to deal with semistandard cylindric tableaux. In this case, for any edge of the cylindric growth diagram with endpoints labeled $\nu$ and $\rho$, with $\rho$ above or to the right of $\nu$, the shape $\rho/\nu$ is a horizontal strip. 

Again, fix $\alpha,\beta\in\CS_{d,L}$, and now let $T\in\SSCT_{d,L}(\alpha/\mu)$ and $U\in\SSCT_{d,L}(\beta/\mu)$, where $\mu\subseteq\alpha,\beta$. Let $N$ and $M$ be upper bounds on the entries of $T$ and $U$, respectively. 
We will build a cylindric growth diagram on a $M\times N$ rectangular grid, where again each vertex is labeled with an element of $\CS_{d,L}$, and use it to compute a pair $(\tP,\tQ)$ such that $(\tP,\tQ)=\CRSK(T,U)$.

We start by labeling the vertices in the left boundary of the grid, from bottom to top, by the sequence $\ssh(T)$, and the vertices in the bottom boundary, from left to right, by the sequence $\ssh(U)$. See Figure~\ref{fig:growthSSCT} for an example.

\begin{figure}[htb]
\centering
\begin{tikzpicture}[scale=1]
\draw[dotted] (0,0) grid (3,4);
\draw[blue] (0,0) node {$[2 1 0]$}; \draw[blue] (-.3,0) node[left] {$\mu=$}; 
\draw (1,0) node {$[3 1 1]$};
\draw (2,0) node {$[5 2 1]$}; 
\draw[red] (3,0) node {$[5 3 2]$}; \draw[red] (3.3,0) node[right] {$=\beta$}; 
\draw (0,1) node {$[4 2 1]$};
\draw (1,1) node {$[5 3 1]$};
\draw (2,1) node {$[5 5 2]$}; 
\draw (3,1) node {$[6 5 3]$};
\draw (0,2) node {$[4 4 1]$};
\draw (1,2) node {$[5 4 2]$};
\draw (2,2) node {$[6 5 3]$}; 
\draw (3,2) node {$[7 6 3]$};
\draw (0,3) node {$[5 4 1]$};
\draw (1,3) node {$[5 5 2]$};
\draw (2,3) node {$[6 5 4]$}; 
\draw (3,3) node {$[7 6 4]$};
\draw[violet] (0,4) node {$[5 5 3]$}; \draw[violet] (-.3,4) node[left] {$\alpha=$}; 
\draw (1,4) node {$[6 5 4]$};
\draw (2,4) node {$[8 6 4]$}; 
\draw[brown] (3,4) node {$[8 7 5]$}; \draw[brown] (3.3,4) node[right] {$=\lambda$}; 

\draw[->] (1,-.4)--node[below]{$U$} (2,-.4);
\draw[->] (-.7,1.5)--node[left]{$T$} (-.7,2.5);
\draw[->] (1,4.4)--node[above]{$Q$} (2,4.4);
\draw[->] (3.7,1.5)--node[right]{$P$} (3.7,2.5);

\begin{scope}[scale=.5,shift={(-11,6)}]
 \fill[yellow!25] (-2,-.5)++(0,.5) \east\east \north\east\north\east\north\east\east -- ++(0,.5)-- ++(3,0) -- ++(0,-.5)
 \west\west \south\south\west\west\south\west\west -- ++(0,-0.5) -- ++(-3,0);
 \fill[yellow] (0,0)\north\east\north\east\north\east\east  \east \south\south\west\west\south\west\west \west;
 \draw[very thin,dotted] (0,0)  grid (5,3);
 \draw[thick,<->,olive] (-.2,0)-- node[left] {$d$} (-.2,3);
 \draw[thick,<->,olive] (0,3.2)-- node[above] {$L$} (4,3.2); 
 \draw[blue, very thick] (-2,-.5)--++(0,.5)  \east\east \north\east\north\east\north\east\east  -- ++(0,.5);
 \draw[violet, very thick] (7,3.5)-- ++(0,-.5) \west\west \south\south\west\west\south\west\west  -- ++(0,-0.5);
 \draw (2.5,2.5) node {1}; \draw (2.5-4,2.5-3) node {\footnotesize 1}; 
 \draw (3.5,2.5) node {1}; \draw (3.5-4,2.5-3) node {\footnotesize 1}; 
 \draw (1.5,1.5) node {1}; 
 \draw (.5,.5) node {1}; \draw (.5+4,.5+3) node {\footnotesize 1}; 
 \draw (2.5,1.5) node {2};  \draw (3.5,1.5) node {2}; 
 \draw (4.5,2.5) node {3}; \draw (4.5-4,2.5-3) node {\footnotesize 3}; 
  \draw (4.5,1.5) node {4}; 
 \draw (1.5,.5) node {4}; \draw (1.5+4,.5+3) node {\footnotesize 4}; 
  \draw (2.5,.5) node {4}; \draw (2.5+4,.5+3) node {\footnotesize 4}; 
 \draw[blue] (1.5,2.5) node {$\mu$};
 \draw[violet] (5.5,1.5) node {$\alpha$};
 \draw[thick,->] (0,3.5)--(0,-1);
\draw[thick,->] (-1,3)--(7.5,3);
\draw (3.5,4) node {$T$};
\end{scope}

\begin{scope}[scale=.5,shift={(-11,0)}]
 \fill[yellow!25] (-2,-.5)++(0,.5) \east\east \north\east\north\east\north\east\east -- ++(0,.5)-- ++(2,0) -- ++(0,-.5)
 \west \south\west\west\south\west\south\west -- ++(0,-0.5) -- ++(-3,0);
 \fill[yellow] (0,0)\north\east\north\east\north\east\east  \east\south\west\west\south\west\south\west  \west\west;
 \draw[very thin,dotted] (0,0)  grid (5,3);
 \draw[thick,<->,olive] (-.2,0)-- node[left] {$d$} (-.2,3);
 \draw[thick,<->,olive] (0,3.2)-- node[above] {$L$} (4,3.2); 
 \draw[blue, very thick] (-2,-.5)--++(0,.5)  \east\east \north\east\north\east\north\east\east  -- ++(0,.5);
 \draw[red, very thick] (6,3.5)-- ++(0,-.5) \west \south\west\west\south\west\south\west  -- ++(0,-0.5);
 \draw (2.5,2.5) node {1}; \draw (2.5-4,2.5-3) node {\footnotesize 1}; 
 \draw (3.5,2.5) node {2}; \draw (3.5-4,2.5-3) node {\footnotesize 2}; 
 \draw (1.5,1.5) node {2}; 
 \draw (.5,.5) node {1}; \draw (.5+4,.5+3) node {\footnotesize 1}; 
 \draw (2.5,1.5) node {3}; 
 \draw (4.5,2.5) node {2}; \draw (4.5-4,2.5-3) node {\footnotesize 2}; 
 \draw (1.5,.5) node {3}; \draw (1.5+4,.5+3) node {\footnotesize 3}; 
 \draw[blue] (1.5,2.5) node {$\mu$};
 \draw[red] (3.5,1.5) node {$\beta$};
 \draw[thick,->] (0,3.5)--(0,-1);
\draw[thick,->] (-1,3)--(6.5,3);
\draw (3.5,4) node {$U$};
\end{scope}

\begin{scope}[scale=.5,shift={(10.4,6)}]
 \fill[yellow!25] (1,-.5)++(0,.5) \east \north\east\north\east\east\north\east -- ++(0,.5)-- ++(3,0) -- ++(0,-.5) \west \south\west\south\west\west\south\west -- ++(0,-0.5) -- ++(-3,0);
 \fill[yellow] (2,0)\north\east\north\east\east\north\east \east\east \south\west\south\west\west\south\west \west\west;
 \draw[very thin,dotted] (0,0)  grid (8,3);
 \draw[red, very thick] (1,-.5)--++(0,.5)  \east \north\east\north\east\east\north\east-- ++(0,.5);
 \draw[brown, very thick] (9,3.5)-- ++(0,-.5)\west \south\west\south\west\west\south\west -- ++(0,-0.5);
 \draw (5.5,2.5) node {1}; \draw (5.5-4,2.5-3) node {\footnotesize 1}; 
 \draw (6.5,2.5) node {2}; \draw (6.5-4,2.5-3) node {\footnotesize 2};
 \draw (7.5,2.5) node {4}; \draw (7.5-4,2.5-3) node {\footnotesize 4}; 
 \draw (3.5,1.5) node {1};
 \draw (4.5,1.5) node {1}; 
 \draw (5.5,1.5) node {2}; 
 \draw (6.5,1.5) node {4}; 
 \draw (2.5,.5) node {1}; \draw (2.5+4,.5+3) node {\footnotesize 1};
 \draw (3.5,.5) node {3}; \draw (3.5+4,.5+3) node {\footnotesize 3};  
 \draw (4.5,.5) node {4}; \draw (4.5+4,.5+3) node {\footnotesize 4};  
 \draw[thick,->] (0,3.5)--(0,-1);
\draw[thick,->] (-.5,3)--(9.5,3);
\draw (4.5,4) node {$P$};
\draw[red] (2.5,1.5) node {$\beta$};
\draw[brown] (5.5,.5) node {$\lambda$};
\end{scope}

\begin{scope}[scale=.5,shift={(10.4,0)}]
 \fill[yellow!25] (1,-.5)++(0,.5) \east\east \north\east\east\north\north\east\east -- ++(0,.5)-- ++(2,0) -- ++(0,-.5) \west \south\west\south\west\west\south\west -- ++(0,-0.5) -- ++(-3,0);
 \fill[yellow] (3,0)\north\east\east\north\north\east\east  \east \south\west\south\west\west\south\west \west;
 \draw[very thin,dotted] (0,0)  grid (8,3);
 \draw[violet, very thick] (1,-.5)--++(0,.5)  \east\east \north\east\east\north\north\east\east -- ++(0,.5);
 \draw[brown, very thick] (9,3.5)-- ++(0,-.5)\west \south\west\south\west\west\south\west -- ++(0,-0.5);
 \draw (5.5,2.5) node {1}; \draw (5.5-4,2.5-3) node {\footnotesize 1}; 
 \draw (6.5,2.5) node {2}; \draw (6.5-4,2.5-3) node {\footnotesize 2};
 \draw (7.5,2.5) node {2}; \draw (7.5-4,2.5-3) node {\footnotesize 2}; 
 \draw (5.5,1.5) node {2}; 
 \draw (6.5,1.5) node {3}; 
 \draw (3.5,.5) node {1}; \draw (3.5+4,.5+3) node {\footnotesize 1};
 \draw (4.5,.5) node {3}; \draw (4.5+4,.5+3) node {\footnotesize 3};  
 \draw[thick,->] (0,3.5)--(0,-1);
\draw[thick,->] (-.5,3)--(9.5,3);
\draw (4.5,4) node {$Q$};
\draw[violet] (4.5,1.5) node {$\alpha$};
\draw[brown] (5.5,.5) node {$\lambda$};
\end{scope}
\end{tikzpicture}
\caption{The computation of $\CRSK(T,U)=(P,Q)$ using growth diagrams.
}
\label{fig:growthSSCT}
\end{figure} 

The forward local rule that computes $\rhour$ given the labels $\rholl$, $\rhoul$ and $\rholr$ of the vertices of a unit square, as in the diagram~\eqref{eq:square}, is now generalized by letting
\begin{equation}\label{eq:forward}
\rhour_i=\max\{\rhoul_i,\rholr_i\}+\min\{\rhoul_{i-1},\rholr_{i-1}\}-\rholl_{i-1}
\end{equation}
for all $i$. It is not hard to see that this rule restricts to (F1) and (F2) when both $\rhoul$ and $\rholr$ are obtained from $\rholl$ by adding one cell.

As an example of equation~\eqref{eq:forward}, let $\rholl=[2,1,0]$, $\rhoul=[4,2,1]$ and $\rholr=[3,1,1]$, as in the bottom-left square in Figure~\ref{fig:growthSSCT}. Then 
$\rhour_1=\max\{4,3\}+\min\{5,5\}-4=5$, $\rhour_2=\max\{2,1\}+\min\{4,3\}-2=3$ and $\rhour_3=\max\{1,1\}+\min\{2,1\}-1=1$, so
$\rhour=[5,3,1]$.

The next lemma shows that the forward local rule~\eqref{eq:forward} has the desired properties when applied to horizontal strips.

\begin{lemma}\label{lem:forward}
Suppose that $\rholl,\rhoul,\rholl,\rhour\in\CS_{d,L}$ satisfy Equation~\eqref{eq:forward}.
Then $\rhoul/\rholl$ and $\rholr/\rholl$ are horizontal strips if and only if so are $\rhour/\rholr$ and $\rhour/\rhoul$. In addition, 
 $|\rhour/\rholr|=|\rhoul/\rholl|$ and $|\rhour/\rhoul|=|\rholr/\rholl|$.
\end{lemma}

\begin{proof}
The fact that $\rhoul/\rholl$ and $\rholr/\rholl$ are horizontal strips is equivalent to the string of inequalities 
$$\rholl_i\le \rhoul_i,\rholr_i\le \rholl_{i-1}$$ for all $i$, which in turn is equivalent to
\begin{equation}\label{eq:rho-max-min} \max\{\rhoul_i,\rholr_i\}\le \rholl_{i-1} \le \min\{\rhoul_{i-1},\rholr_{i-1}\}\end{equation} for all $i$.
By Equation~\eqref{eq:forward}, $\rhour_i$ is obtained by reflecting $\rholl_{i-1}$ with respect to the midpoint of the interval
$[\max\{\rhoul_i,\rholr_i\},\min\{\rhoul_{i-1},\rholr_{i-1}\}]$. Thus, $\rholl_{i-1}$ belongs to this interval if and only if so does $\rhour_i$, making Equation~\eqref{eq:rho-max-min} equivalent to
$$\max\{\rhoul_i,\rholr_i\}\le \rhour_{i} \le \min\{\rhoul_{i-1},\rholr_{i-1}\}$$ for all~$i$. But these inequalities simply state the fact that $\rhour/\rholr$ and $\rhour/\rhoul$ are horizontal strips.

To prove the statement about the number of cells, subtract $\rholl_i$ from both sides of Equation~\eqref{eq:forward} to get
$$\rhour_i-\rholl_i=\max\{\rhoul_i-\rholl_i,\rholr_i-\rholl_i\}+\min\{\rhoul_{i-1}-\rholl_{i-1},\rholr_{i-1}-\rholl_{i-1}\}.$$
Summing over $1\le i\le d$ and considering the periodicity of cylindric shapes, we get
$$|\rhour/\rholl|=|\rhoul/\rholl|+|\rholr/\rholl|.$$
Writing the left-hand side as $|\rhour/\rhoul|+|\rhoul/\rholl|$ and canceling the term $|\rhoul/\rholl|$, we deduce that $|\rhour/\rhoul|=|\rholr/\rholl|$. Similarly, writing the left-hand side as $|\rhour/\rholr|+|\rholr/\rholl|$, we deduce that $|\rhour/\rholr|=|\rhoul/\rholl|$.
\end{proof}

Once all the labels of the grid have been computed, let $\lambda$ be the label of vertex $(M,N)$. Let $\tP\in\SSCT_{d,L}(\lambda/\beta)$ be determined by the sequence of the labels of the right boundary of the grid, and let $\tQ\in\SSCT_{d,L}(\lambda/\alpha)$ be determined by the sequence of labels of the top boundary. Lemma~\ref{lem:forward} guarantees that $\tP$ and $\tQ$ are semistandard cylindric tableaux, and that $\wt(\tP)=\wt(T)$ and $\wt(\tQ)=\wt(U)$.

In these growth diagrams for semistandard cylindric tableaux, the backward rule to compute $\rholl$ given $\rhoul$, $\rholr$ and $\rhour$ is obtained by solving Equation~\eqref{eq:forward} for $\rholl_{i-1}$ and shifting the indices:
\begin{equation}\label{eq:backward}
\rholl_i=\max\{\rhoul_{i+1},\rholr_{i+1}\}+\min\{\rhoul_{i},\rholr_{i}\}-\rhour_{i+1}
\end{equation}
for all~$i$. Iterating this rule yields the inverse map that takes $(\tP,\tQ)\in\SSCT_{d,L}(\lambda/\beta)\times\SSCT_{d,L}(\lambda/\alpha)$ to the pair $(T,U)$.

By comparing the multi-insertion construction from~\cite{neyman_cylindric_2015} with our growth diagram construction, the next theorem can be proved similarly to how we proved Theorem~\ref{thm:CRS=growth}.

\begin{theorem}
Let $(T,U)\in\SSCT_{d,L}(\alpha/\mu)\times \SSCT_{d,L}(\beta/\mu)$. Suppose that $(T,U)\mapsto (\tP,\tQ)$ via the above growth diagram construction, and that $\CRSK(T,U)=(P,Q)$, where $\CRSK$ is the bijection from Theorem~\ref{thm:CRSK}.
Then  $(\tP,\tQ)=(P,Q)$.
\end{theorem}

\subsection{Edge local rules}

An alternative way to construct cylindric growth diagrams is by labeling the edges instead of the vertices, in analogy with Viennot's edge local rules \cite{viennot_growth_2018} for growth diagrams for standard Young tableaux (see also \cite{pfannerer_promotion_2020}).

The label of each edge is the difference of the labels of its endpoints.
Specifically, for a square with vertex labels as in Equation~\eqref{eq:square}, the edge labels are given by the following diagram:
$$
\begin{tikzpicture}[scale=1,decoration={
    markings,
    mark=at position .5 with {\arrow{>}}}
    ] 
\draw[dotted] (0,0) rectangle (1,1);
\draw[postaction={decorate},purple] (0,0) -- node[below,scale=.9]{$\delta^B=\rholr-\rholl$} ++(1,0);
\draw[postaction={decorate},purple] (0,1) -- node[above,scale=.9]{$\delta^T=\rhour-\rhoul$} ++(1,0);
\draw[postaction={decorate},purple] (0,0) -- node[left,scale=.9]{$\delta^L=\rhoul-\rholl$} ++(0,1);
\draw[postaction={decorate},purple] (1,0) -- node[right,scale=.9]{$\delta^R=\rhour-\rholr$} ++(0,1);
\end{tikzpicture}
$$
Note that all the edge labels are periodic sequences with period $d$, so a label $\delta$ is determined by $[\delta_1,\delta_2,\dots,\delta_d]$. 

The local rules (F1) and (F2), and more generally Equation~\eqref{eq:forward}, can be translated into local rules for edge labels, which describe how to compute $\delta^R$ and $\delta^T$ in terms of $\delta^L$ and $\delta^B$.

In the standard case (Subsection~\ref{sec:growthSCT}), $\rhoul$ is obtained from $\rholl$ by adding a cell to some row $i$, where $1\le i\le d$, so the sequence $[\delta^L_1,\delta^L_2,\dots,\delta^L_d]$ has a one in position $i$ and zeros elsewhere. We will denote this by $\delta^L=(i)$. Similarly, we have $\delta^B=(j)$ for some $1\le j\le d$.
The rule (F1),
when translated to edge labels, states that if $i\neq j$, then $\delta^R=(i)$ and $\delta^T=(j)$. The rule (F2) states that if $i=j$, then $\delta^R=\delta^T=(i+1)$, with indices modulo $d$. We can represent these edge local rules as
\begin{equation}\label{eq:edge_rules}
\begin{tikzpicture}[scale=1,decoration={
    markings,
    mark=at position 0.5 with {\arrow{>}}}
    ] 
\draw[dotted] (0,0) rectangle (1,1);
\draw[postaction={decorate},purple] (0,0) -- node[below,scale=.9]{$(j)$} ++(1,0);
\draw[postaction={decorate},purple] (0,1) -- node[above,scale=.9]{$(j)$} ++(1,0);
\draw[postaction={decorate},purple] (0,0) -- node[left,scale=.9]{$(i)$} ++(0,1);
\draw[postaction={decorate},purple] (1,0) -- node[right,scale=.9]{$(i)$} ++(0,1);
\draw (1.8,.5) node[right]{for all $i\neq j$,};
\begin{scope}[shift={(6,0)}]
\draw[dotted] (0,0) rectangle (1,1);
\draw[postaction={decorate},purple] (0,0) -- node[below,scale=.9]{$(i)$} ++(1,0);
\draw[postaction={decorate},purple] (0,1) -- node[above,scale=.9]{$(i+1)$} ++(1,0);
\draw[postaction={decorate},purple] (0,0) -- node[left,scale=.9]{$(i)$} ++(0,1);
\draw[postaction={decorate},purple] (1,0) -- node[right,scale=.9]{$(i+1)$} ++(0,1);
\draw (2.5,.5) node[right]{for all $i$.};
\end{scope}
\end{tikzpicture}
\end{equation}
See the left of Figure~\ref{fig:growthSSCT-edges} for an example of a growth diagram for standard cylindric tableaux with labeled edges.

\begin{figure}[htb]
\centering
\begin{tikzpicture}[scale=1]
\draw[dotted] (0,0) grid (4,5);
\draw[purple] (.5,0) node[scale=.9]{$(3)$};
\draw[purple] (1.5,0) node[scale=.9]{$(1)$};
\draw[purple] (2.5,0) node[scale=.9]{$(2)$};
\draw[purple] (3.5,0) node[scale=.9]{$(3)$};
\draw[purple] (0,.5) node[scale=.9]{$(1)$};
\draw[purple] (0,1.5) node[scale=.9]{$(3)$};
\draw[purple] (0,2.5) node[scale=.9]{$(1)$};
\draw[purple] (0,3.5) node[scale=.9]{$(2)$};
\draw[purple] (0,4.5) node[scale=.9]{$(2)$};

\draw[purple] (1,.5) node[scale=.9]{$(1)$};
\draw[purple] (.5,1) node[scale=.9]{$(3)$};
\draw[purple] (2,.5) node[scale=.9]{$(2)$};
\draw[purple] (1.5,1) node[scale=.9]{$(2)$};
\draw[purple] (3,.5) node[scale=.9]{$(3)$};
\draw[purple] (2.5,1) node[scale=.9]{$(3)$};
\draw[purple] (4,.5) node[scale=.9]{$(1)$};
\draw[purple] (3.5,1) node[scale=.9]{$(1)$};

\draw[purple] (1,1.5) node[scale=.9]{$(1)$};
\draw[purple] (.5,2) node[scale=.9]{$(1)$};
\draw[purple] (2,1.5) node[scale=.9]{$(1)$};
\draw[purple] (1.5,2) node[scale=.9]{$(2)$};
\draw[purple] (3,1.5) node[scale=.9]{$(1)$};
\draw[purple] (2.5,2) node[scale=.9]{$(3)$};
\draw[purple] (4,1.5) node[scale=.9]{$(2)$};
\draw[purple] (3.5,2) node[scale=.9]{$(2)$};

\draw[purple] (1,2.5) node[scale=.9]{$(2)$};
\draw[purple] (.5,3) node[scale=.9]{$(2)$};
\draw[purple] (2,2.5) node[scale=.9]{$(3)$};
\draw[purple] (1.5,3) node[scale=.9]{$(3)$};
\draw[purple] (3,2.5) node[scale=.9]{$(1)$};
\draw[purple] (2.5,3) node[scale=.9]{$(1)$};
\draw[purple] (4,2.5) node[scale=.9]{$(1)$};
\draw[purple] (3.5,3) node[scale=.9]{$(2)$};

\draw[purple] (1,3.5) node[scale=.9]{$(3)$};
\draw[purple] (.5,4) node[scale=.9]{$(3)$};
\draw[purple] (2,3.5) node[scale=.9]{$(1)$};
\draw[purple] (1.5,4) node[scale=.9]{$(1)$};
\draw[purple] (3,3.5) node[scale=.9]{$(2)$};
\draw[purple] (2.5,4) node[scale=.9]{$(2)$};
\draw[purple] (4,3.5) node[scale=.9]{$(3)$};
\draw[purple] (3.5,4) node[scale=.9]{$(3)$};

\draw[purple] (1,4.5) node[scale=.9]{$(2)$};
\draw[purple] (.5,5) node[scale=.9]{$(3)$};
\draw[purple] (2,4.5) node[scale=.9]{$(2)$};
\draw[purple] (1.5,5) node[scale=.9]{$(1)$};
\draw[purple] (3,4.5) node[scale=.9]{$(3)$};
\draw[purple] (2.5,5) node[scale=.9]{$(3)$};
\draw[purple] (4,4.5) node[scale=.9]{$(1)$};
\draw[purple] (3.5,5) node[scale=.9]{$(1)$};

\draw[->] (1.5,-.4)--node[below]{$U$} (2.5,-.4);
\draw[->] (-.7,2)--node[left]{$T$} (-.7,3);
\draw[->] (1.5,5.4)--node[above]{$Q$} (2.5,5.4);
\draw[->] (4.7,2)--node[right]{$P$} (4.7,3);

\begin{scope}[shift={(8,-.5)},scale=1.6,decoration={
    markings,
    mark=at position 0.5 with {\arrow{>}}}
    ] 
\draw[dotted] (0,0) grid (3,4);
\draw[purple] (.5,0) node[scale=.9]{$[101]$};
\draw[purple] (1.5,0) node[scale=.9]{$[210]$};
\draw[purple] (2.5,0) node[scale=.9]{$[011]$};
\draw[purple] (0,.5) node[scale=.9]{$[211]$};
\draw[purple] (0,1.5) node[scale=.9]{$[020]$};
\draw[purple] (0,2.5) node[scale=.9]{$[100]$};
\draw[purple] (0,3.5) node[scale=.9]{$[012]$};
\draw[olive] (.53,.37) node[scale=.8]{$[101]$};
\draw[teal] (.47,.63) node[scale=.8]{$[01\mn{1}]$};
\draw[purple] (.5,1) node[scale=.9]{$[110]$};
\draw[purple] (1,.5) node[scale=.9]{$[220]$};
\begin{scope}[shift={(1,0)}]
\draw[olive] (.53,.37) node[scale=.8]{$[210]$};
\draw[teal] (.47,.63) node[scale=.8]{$[\mn{2}11]$};
\draw[purple] (.5,1) node[scale=.9]{$[021]$};
\draw[purple] (1,.5) node[scale=.9]{$[031]$};
\end{scope}
\begin{scope}[shift={(2,0)}]
\draw[olive] (.53,.37) node[scale=.8]{$[011]$};
\draw[teal] (.47,.63) node[scale=.8]{$[1\mn{1}0]$};
\draw[purple] (.5,1) node[scale=.9]{$[101]$};
\draw[purple] (1,.5) node[scale=.9]{$[121]$};
\end{scope}
\begin{scope}[shift={(0,1)}]
\draw[olive] (.53,.37) node[scale=.8]{$[010]$};
\draw[teal] (.47,.63) node[scale=.8]{$[0\mn{1}1]$};
\draw[purple] (.5,1) node[scale=.9]{$[101]$};
\draw[purple] (1,.5) node[scale=.9]{$[011]$};
\end{scope}
\begin{scope}[shift={(1,1)}]
\draw[olive] (.53,.37) node[scale=.8]{$[011]$};
\draw[teal] (.47,.63) node[scale=.8]{$[1\mn{1}0]$};
\draw[purple] (.5,1) node[scale=.9]{$[111]$};
\draw[purple] (1,.5) node[scale=.9]{$[101]$};
\end{scope}
\begin{scope}[shift={(2,1)}]
\draw[olive] (.53,.37) node[scale=.8]{$[101]$};
\draw[teal] (.47,.63) node[scale=.8]{$[01\mn{1}]$};
\draw[purple] (.5,1) node[scale=.9]{$[110]$};
\draw[purple] (1,.5) node[scale=.9]{$[110]$};
\end{scope}
\begin{scope}[shift={(0,2)}]
\draw[olive] (.53,.37) node[scale=.8]{$[100]$};
\draw[teal] (.47,.63) node[scale=.8]{$[\mn{1}10]$};
\draw[purple] (.5,1) node[scale=.9]{$[011]$};
\draw[purple] (1,.5) node[scale=.9]{$[010]$};
\end{scope}
\begin{scope}[shift={(1,2)}]
\draw[olive] (.53,.37) node[scale=.8]{$[010]$};
\draw[teal] (.47,.63) node[scale=.8]{$[0\mn{1}1]$};
\draw[purple] (.5,1) node[scale=.9]{$[102]$};
\draw[purple] (1,.5) node[scale=.9]{$[001]$};
\end{scope}
\begin{scope}[shift={(2,2)}]
\draw[olive] (.53,.37) node[scale=.8]{$[000]$};
\draw[teal] (.47,.63) node[scale=.8]{$[000]$};
\draw[purple] (.5,1) node[scale=.9]{$[110]$};
\draw[purple] (1,.5) node[scale=.9]{$[001]$};
\end{scope}
\begin{scope}[shift={(0,3)}]
\draw[olive] (.53,.37) node[scale=.8]{$[011]$};
\draw[teal] (.47,.63) node[scale=.8]{$[1\mn{1}0]$};
\draw[purple] (.5,1) node[scale=.9]{$[101]$};
\draw[purple] (1,.5) node[scale=.9]{$[102]$};
\end{scope}
\begin{scope}[shift={(1,3)}]
\draw[olive] (.53,.37) node[scale=.8]{$[102]$};
\draw[teal] (.47,.63) node[scale=.8]{$[11\mn{2}]$};
\draw[purple] (.5,1) node[scale=.9]{$[210]$};
\draw[purple] (1,.5) node[scale=.9]{$[210]$};
\end{scope}
\begin{scope}[shift={(2,3)}]
\draw[olive] (.53,.37) node[scale=.8]{$[110]$};
\draw[teal] (.47,.63) node[scale=.8]{$[\mn{1}01]$};
\draw[purple] (.5,1) node[scale=.9]{$[011]$};
\draw[purple] (1,.5) node[scale=.9]{$[111]$};
\end{scope}
\draw[->] (1,-.3)--node[below]{$U$} (2,-.3);
\draw[->] (-.5,1.5)--node[left]{$T$} (-.5,2.5);
\draw[->] (1,4.3)--node[above]{$Q$} (2,4.3);
\draw[->] (3.5,1.5)--node[right]{$P$} (3.5,2.5);
\end{scope}
\end{tikzpicture}
\caption{The computation of the growth diagrams from Figures~\ref{fig:growth} (left) and~\ref{fig:growthSSCT} (right) via edge local rules. 
Inside each square on the right diagram, we have written  $\textcolor{olive}{m}$, with $\textcolor{teal}{\Delta m}$ above it.
}
\label{fig:growthSSCT-edges}
\end{figure} 

In the semistandard case (Subsection~\ref{sec:growthSSCT}), the edge labels can have arbitrary nonnegative entries. To translate the rule in Equation~\eqref{eq:forward} into a rule for the edge labels, we first rewrite this equation as
$$\rhour_i=\rhoul_i+\rholr_i-\min\{\rhoul_i,\rholr_i\}+\min\{\rhoul_{i-1},\rholr_{i-1}\}-\rholl_{i-1},$$
and add $-\rhoul_i-\rholr_i+\rholl_i$ to both sides to get
$$\left.\begin{array}{c}(\rhour_i-\rholr_i)-(\rhoul_i-\rholl_i)\\
(\rhour_i-\rhoul_i)-(\rholr_i-\rholl_i)\end{array}\right\}=\min\{\rhoul_{i-1}-\rholl_{i-1},\rholr_{i-1}-\rholl_{i-1}\}-\min\{\rhoul_i-\rholl_i,\rholr_i-\rholl_i\},$$
or equivalently,
$$\left.\begin{array}{c}\delta^R_i-\delta^L_i\\
\delta^T_i-\delta^B_i\end{array}\right\}=\min\{\delta^L_{i-1},\delta^B_{i-1}\}-\min\{\delta^L_i,\delta^B_i\}.$$

Thus, if we let $m_i= \min\{\delta^L_i,\delta^B_i\}$ and $\Delta m_i=m_{i-1}-m_i$ for all $i$, we can compute $\delta^R$ and $\delta^T$ from $\delta^L$ and $\delta^B$ by using the formulas
\begin{align*} 
 \delta^R&=\delta^L+\Delta m,\\
\delta^T&=\delta^B+\Delta m.
 \end{align*}

For example, if $\delta^L=[2,1,1]$ and $\delta^B=[1,0,1]$, then $m=[1,0,1]$ and $\Delta m=[0,1,-1]$, so 
$\delta^T=[1,0,1]+[0,1,-1]=[1,1,0]$ and $\delta^R=[2,1,1]+[0,1,-1]=[2,2,0]$.
A full example of the computation of the edge labels of a growth diagram for semistandard cylindric tableaux is given on 
the right of Figure~\ref{fig:growthSSCT-edges}.

\section{Properties of $\CRSK$}\label{sec:properties}

\subsection{Symmetry}\label{sec:symmetry}

One advantage of the growth diagram description of $\CRSK$ is that it explains the symmetry of this correspondence when $T$ and $U$ are swapped. 
The following result was proved by Neyman~\cite{neyman_cylindric_2015} using the insertion description. We can now provide a more transparent proof.

\begin{corollary}[{\cite[Thm.\ 5.18]{neyman_cylindric_2015}}]
If $\CRSK(T,U)=(P,Q)$, then $\CRSK(U,T)=(Q,P)$.
\end{corollary}

\begin{proof}
The local rule in equation~\eqref{eq:forward} is symmetric with respect to switching the axes, since it does not distinguish between $\rhoul$ and $\rholr$. Thus, the growth diagram for the pair $(U,T)$ is the reflection of the growth diagram for the pair $(T,U)$. This reflection swaps the upper boundary with the right boundary, and hence the roles of $P$ and~$Q$.
\end{proof}

For standard cylindric tableaux, in the special case $T=U$ (which requires $\alpha=\beta$ and $m=n$), the resulting growth diagram is symmetric, and so $P=Q$. The resulting map $T\mapsto P$ is the bijection 
$\bij:\SCT^n_{d,L}(\alpha/\cdot)\to\SCT^n_{d,L}(\cdot/\alpha)$ from Corollary~\ref{cor:bij}. Figure~\ref{fig:growth-bij} shows an example of the computation of $\bij$ using growth diagrams.

\begin{figure}[htb]
\centering
\begin{tikzpicture}[scale=1]
\draw[dotted] (0,0) grid (8,8);
\draw (0,0) node {$[\mn1\mn1\mn2]$};
\draw (1,0) node {$[0 \mn1 \mn2]$}; \draw (0,1) node {$[0 \mn1 \mn2]$};
\draw (2,0) node {$[0 \mn1 \mn1]$}; \draw (0,2) node {$[0 \mn1 \mn1]$};
\draw (3,0) node {$[1 \mn1 \mn1]$}; \draw (0,3) node {$[1 \mn1 \mn1]$};
\draw (4,0) node {$[1 0 \mn1]$}; \draw (0,4) node {$[1 0 \mn1]$};
\draw (5,0) node {$[1 1 \mn1]$}; \draw (0,5) node {$[1 1 \mn1]$};
\draw (6,0) node {$[2 1 \mn1]$}; \draw (0,6) node {$[2 1 \mn1]$};
\draw (7,0) node {$[2 2 \mn1]$}; \draw (0,7) node {$[2 2 \mn1]$};
\draw (8,0) node[violet] {$[2 2 0]$}; \draw (8.3,0) node[right,violet] {$=\alpha$}; 
\draw (0,8) node[violet] {$[2 2 0]$}; \draw (-.3,8) node[left,violet] {$\alpha=$}; 
\draw (1,1) node {$[0 0 \mn2]$}; 
\draw (2,1) node {$[0 0 \mn1]$}; \draw (1,2) node {$[0 0 \mn1]$};
\draw (3,1) node {$[1 0 \mn1]$}; \draw (1,3) node {$[1 0 \mn1]$};
\draw (4,1) node {$[1 0 0]$}; \draw (1,4) node {$[1 0 0]$};
\draw (5,1) node {$[1 1 0]$}; \draw (1,5) node {$[1 1 0]$};
\draw (6,1) node {$[2 1 0]$}; \draw (1,6) node {$[2 1 0]$};
\draw (7,1) node {$[2 2 0]$}; \draw (1,7) node {$[2 2 0]$};
\draw (8,1) node {$[3 2 0]$}; \draw (1,8) node {$[3 2 0]$};
\draw (2,2) node {$[1 0 \mn1]$};
\draw (3,2) node {$[1 1 \mn1]$}; \draw (2,3) node {$[1 1 \mn1]$};
\draw (4,2) node {$[1 1 0]$}; \draw (2,4) node {$[1 1 0]$};
\draw (5,2) node {$[1 1 1]$}; \draw (2,5) node {$[1 1 1]$};
\draw (6,2) node {$[2 1 1]$}; \draw (2,6) node {$[2 1 1]$};
\draw (7,2) node {$[2 2 1]$}; \draw (2,7) node {$[2 2 1]$};
\draw (8,2) node {$[3 2 1]$}; \draw (2,8) node {$[3 2 1]$};
\draw (3,3) node {$[1 1 0]$};
\draw (4,3) node {$[2 1 0]$}; \draw (3,4) node {$[2 1 0]$};
\draw (5,3) node {$[2 1 1]$}; \draw (3,5) node {$[2 1 1]$};
\draw (6,3) node {$[2 2 1]$}; \draw (3,6) node {$[2 2 1]$};
\draw (7,3) node {$[2 2 2]$}; \draw (3,7) node {$[2 2 2]$};
\draw (8,3) node {$[3 2 2]$}; \draw (3,8) node {$[3 2 2]$};
\draw (4,4) node {$[2 2 0]$};
\draw (5,4) node {$[2 2 1]$}; \draw (4,5) node {$[2 2 1]$};
\draw (6,4) node {$[2 2 2]$}; \draw (4,6) node {$[2 2 2]$};
\draw (7,4) node {$[3 2 2]$}; \draw (4,7) node {$[3 2 2]$};
\draw (8,4) node {$[3 3 2]$}; \draw (4,8) node {$[3 3 2]$};
\draw (5,5) node {$[3 2 1]$};
\draw (6,5) node {$[3 2 2]$}; \draw (5,6) node {$[3 2 2]$};
\draw (7,5) node {$[3 3 2]$}; \draw (5,7) node {$[3 3 2]$};
\draw (8,5) node {$[3 3 3]$}; \draw (5,8) node {$[3 3 3]$};
 \draw (6,6) node {$[4 2 2]$};
\draw (7,6) node {$[4 3 2]$}; \draw (6,7) node {$[4 3 2]$};
\draw (8,6) node {$[4 3 3]$}; \draw (6,8) node {$[4 3 3]$};
\draw (7,7) node {$[4 3 3]$};
\draw (8,7) node {$[5 3 3]$}; \draw (7,8) node {$[5 3 3]$};
\draw (8,8) node {$[5 4 3]$};
\draw[->] (3.5,-.4)--node[below]{$T$} (4.5,-.4);
\draw[->] (-.7,3.5)--node[left]{$T$} (-.7,4.5);
\draw[->] (3.5,8.4)--node[above]{$P$} (4.5,8.4);
\draw[->] (8.7,3.5)--node[right]{$P$} (8.7,4.5);

\begin{scope}[scale=.5,shift={(-7,11)}]
\fill[yellow!25] (-4,-.5) rectangle (-1,0);
\fill[yellow!25] (1,3) rectangle (3,3.5);
\fill[yellow] (-2,0) \north\east\north\north\east\east \east\east  \west\south\south\west\west\south \west\west;
 \draw[very thin,dotted] (-2,0)  grid (2,3);
  \draw[thick,<->,olive] (-.2,0)-- (-.2,3);
 \draw[thick,<->,olive] (0,3.2)-- (3,3.2); 
 \draw[black, very thick] (-4,-.5)--++(0,.5) \east\east \north\east\north\north\east\east -- ++(0,.5);
 \draw[violet, ultra thick] (3,3.5)-- ++(0,-.5)\west\south\south\west\west\south \west -- ++(0,-0.5);
 \draw (-.5,2.5) node {1}; \draw (-.5-3,2.5-3) node {\footnotesize 1}; 
 \draw (-1.5,.5) node {2}; \draw (-1.5+3,.5+3) node {\footnotesize 2};  
 \draw (.5,2.5) node {3}; \draw (.5-3,2.5-3) node {\footnotesize 3}; 
 \draw (-.5,1.5) node {4}; 
 \draw (.5,1.5) node {5}; 
 \draw (1.5,2.5) node {6}; \draw (1.5-3,2.5-3) node {\footnotesize 6}; 
 \draw (1.5,1.5) node {7}; 
 \draw (-.5,.5) node {8}; \draw (-.5+3,.5+3) node {\footnotesize 8};  
 \draw[violet] (2.5,1.5) node {$\alpha$};
 \draw[->] (0,3.5)--(0,-1);
\draw[->] (-4,3)--(3.5,3);
\draw (-.5,4) node {$T$};
\end{scope}

\begin{scope}[scale=.5,shift={(-10,2)}]
\fill[pink!50] (-1,-.5) rectangle (2,0);
\fill[pink!50] (3,3) rectangle (6,3.5);
\fill[pink] (0,0) \north\east\east\north\north \east\east\east \south \west\south\west\south\west\west\west;
 \draw[very thin,dotted] (0,0)  grid (5,3);
 \draw[black, very thick] (6,3.5)-- ++(0,-.5)\west\south\west\south\west\south\west  -- ++(0,-0.5);
  \draw[violet, ultra thick] (3,3.5)-- ++(0,-.5)\west\south\south\west\west\south \west -- ++(0,-0.5);
  \draw[violet] (1.5,1.5) node {$\alpha$};
 \draw (2.5,2.5) node {1}; \draw (2.5-3,2.5-3) node {\footnotesize 1}; 
 \draw (.5,.5) node {2}; \draw (.5+3,.5+3) node {\footnotesize 2};  
 \draw (1.5,.5) node {3}; \draw (1.5+3,.5+3) node {\footnotesize 3}; 
 \draw (2.5,1.5) node {4}; 
 \draw (2.5,.5) node {5}; \draw (2.5+3,.5+3) node {\footnotesize 5}; 
 \draw (3.5,2.5) node {6}; \draw (3.5-3,2.5-3) node {\footnotesize 6}; 
 \draw (4.5,2.5) node {7}; \draw (4.5-3,2.5-3) node {\footnotesize 7}; 
 \draw (3.5,1.5) node {8}; 
 \draw[->] (0,3.5)--(0,-1);
\draw[->] (-1,3)--(6.5,3);
\draw (2,4) node {$P$};
\end{scope}

\end{tikzpicture}
\caption{The computation of $\bij$ using growth diagrams, for the same example as in Figure~\ref{fig:bij}. 
}
\label{fig:growth-bij}
\end{figure} 

\subsection{Complementation}\label{sec:complementation}

In Section~\ref{sec:connections} we studied conjugation, which is the bijection between $\CS_{d,L}$ and $\CS_{L,d}$ obtained by reflection along the diagonal $y=x$. Next we consider another natural symmetry on cylindric shapes that results from  $180^\circ$ rotation. 
For $\lambda=[\lambda_1,\lambda_2,\dots,\lambda_d]\in\CS_{d,L}$, let 
$$\blam=[L-\lambda_{d},L-\lambda_{d-1},\dots,L-\lambda_{1}]\in\CS_{d,L}.
$$
We call $\blam$ the {\em complement}\footnote{The operation $\lambda\mapsto\blam$ is similar to the Flip function from \cite[Def.~4.12]{neyman_cylindric_2015}, but we use a different indexing of the rows and columns of the flipped diagram.} of $\lambda$,
and we note that the map $\lam\mapsto\blam$ is an involution on $\CS_{d,L}$.
For example, the complement of $\lambda=[2,2,0]\in\CS_{3,3}$ is $\blam=[3,1,1]$.
The boundary sequences $\bdry_{\lambda}$ and $\bdry_{\blam}$ are reversals of each other, since
\begin{align*}\bdry_{\blam}&=(\w^{\blam_0-\blam_1}\s\w^{\blam_1-\blam_2}\s\dots \w^{\blam_{d-1}-\blam_d}\s)^\infty
=(\w^{\lam_d-\lam_{d+1}}\s\w^{\lam_{d-1}-\lam_d}\s\dots \w^{\lam_{1}-\lam_2}\s)^\infty\\
&=(\s\w^{\lam_{d-1}-\lam_d}\s\dots \w^{\lam_{1}-\lam_2}\s\w^{\lam_0-\lam_{1}})^\infty.
\end{align*}

For $\langle i,j\rangle\in\C_{d,L}$, define $\barij=\langle d+1-i,L+1-j \rangle$.
It follows from the definitions that $\langle i,j \rangle\in Y_\lambda$ if and only if $\barij\notin Y_{\blam}$. Thus, for $\lambda,\mu\in\CS_{d,L}$, we have $\mu\subseteq\lambda$ if and only if $\blam\subseteq\bar\mu$, and 
$\langle i,j \rangle\in Y_{\lambda/\mu}$ if and only if $\barij\in Y_{\bar\mu/\blam}$.
For $T\in\SCT_{d,L}(\lambda/\mu)$ with $|\lambda/\mu|=n$, we define its {\em complement} tableau $\bT\in\SCT_{d,L}(\bar\mu/\blam)$ to be the one with entries
$\bT(\barij)=n+1-T(\ij)$ for all $\ij\in Y_{\lambda/\mu}$. Viewed as fillings of cylindric Young diagrams, $\bT$ is obtained from $T$ by performing a $180^\circ$ rotation and replacing each entry $k$ with $n+1-k$. See the examples in Figure~\ref{fig:complement}.
For $T\in\SSCT_{d,L}(\lambda/\mu)$, its {\em complement} $\bT\in\SSCT_{d,L}(\bar\mu/\blam)$ is defined similarly, but with the role of $n+1$ played by the sum of the largest and the smallest entries of~$T$.

\begin{figure}[htb]
\centering
\begin{tikzpicture}[scale=0.5]
\fill[yellow!25] (-4,-.5) rectangle (-1,0);
\fill[yellow!25] (1,3) rectangle (3,3.5);
\fill[yellow] (-2,0) \north\east\north\north\east\east \east\east  \west\south\south\west\west\south \west\west;
 \draw[very thin,dotted] (-2,0)  grid (2,3);
 \draw[thick,<->,olive] (-.2,0)-- (-.2,3);
 \draw[thick,<->,olive] (0,3.2)-- (3,3.2); 
 \draw[green, very thick] (-4,-.5)--++(0,.5) \east\east \north\east\north\north\east\east -- ++(0,.5);
 \draw[violet, ultra thick] (3,3.5)-- ++(0,-.5)\west\south\south\west\west\south \west -- ++(0,-0.5);
 \draw (-.5,2.5) node {1}; \draw (-.5-3,2.5-3) node {\footnotesize 1}; 
 \draw (-1.5,.5) node {2}; \draw (-1.5+3,.5+3) node {\footnotesize 2};  
 \draw (.5,2.5) node {3}; \draw (.5-3,2.5-3) node {\footnotesize 3}; 
 \draw (-.5,1.5) node {4}; 
 \draw (.5,1.5) node {5}; 
 \draw (1.5,2.5) node {6}; \draw (1.5-3,2.5-3) node {\footnotesize 6}; 
 \draw (1.5,1.5) node {7}; 
 \draw (-.5,.5) node {8}; \draw (-.5+3,.5+3) node {\footnotesize 8};  
 \draw[violet] (2.5,1.5) node {$\alpha$};
 \draw[->] (0,3.5)--(0,-1);
\draw[->] (-4,3)--(3.5,3);
\draw (-.5,4) node {$T$};
\draw[->] (4,1.5)-- node[above]{$\bij$}(5,1.5);
\draw[<->] (-1,-1.5)--node[left]{complement}(-1,-2.75);

\begin{scope}[shift={(7,0)}]
\fill[pink!50] (-1,-.5) rectangle (2,0);
\fill[pink!50] (3,3) rectangle (6,3.5);
\fill[pink] (0,0) \north\east\east\north\north \east\east\east \south \west\south\west\south\west\west\west;
 \draw[very thin,dotted] (0,0)  grid (5,3);
 \draw[brown, very thick] (6,3.5)-- ++(0,-.5)\west\south\west\south\west\south\west  -- ++(0,-0.5);
  \draw[violet, ultra thick] (3,3.5)-- ++(0,-.5)\west\south\south\west\west\south \west -- ++(0,-0.5);
  \draw[violet] (1.5,1.5) node {$\alpha$};
 \draw (2.5,2.5) node {1}; \draw (2.5-3,2.5-3) node {\footnotesize 1}; 
 \draw (.5,.5) node {2}; \draw (.5+3,.5+3) node {\footnotesize 2};  
 \draw (1.5,.5) node {3}; \draw (1.5+3,.5+3) node {\footnotesize 3}; 
 \draw (2.5,1.5) node {4}; 
 \draw (2.5,.5) node {5}; \draw (2.5+3,.5+3) node {\footnotesize 5}; 
 \draw (3.5,2.5) node {6}; \draw (3.5-3,2.5-3) node {\footnotesize 6}; 
 \draw (4.5,2.5) node {7}; \draw (4.5-3,2.5-3) node {\footnotesize 7}; 
 \draw (3.5,1.5) node {8}; 
 \draw[->] (0,3.5)--(0,-1);
\draw[->] (-1,3)--(6.5,3);
\draw (1.5,4) node {$P$};
\draw[<->] (1,-1.5)--node[right]{complement}(1,-2.75);
\end{scope}

\begin{scope}[shift={(-4,-7.5)}]\fill[yellow!25] (0,-.5) rectangle (2,0);
\fill[yellow!25] (4,3) rectangle (7,3.5);
\fill[yellow] (1,0) \north\north\east\east\north\east \east \south\west\south\south\west\west\west;
 \draw[very thin,dotted] (0,0)  grid (5,3);
 \draw[violet, ultra thick] (4,3.5)-- ++(0,-.5)\west\south\west\west\south\south\west -- ++(0,-0.5);
 \draw[green, very thick] (7,3.5)-- ++(0,-.5)\west\west\south\west\south\south\west\west  -- ++(0,-0.5);
  \draw[violet] (.5,1.5) node {$\balpha$};
 \draw (3.5,2.5) node {1}; \draw (3.5-3,2.5-3) node {\footnotesize 1}; 
 \draw (1.5,1.5) node {2}; 
 \draw (1.5,.5) node {3}; \draw (1.5+3,.5+3) node {\footnotesize 3};  
 \draw (2.5,1.5) node {4}; 
 \draw (3.5,1.5) node {5}; 
 \draw (2.5,.5) node {6}; \draw (2.5+3,.5+3) node {\footnotesize 6};  
 \draw (4.5,2.5) node {7}; \draw (4.5-3,2.5-3) node {\footnotesize 7}; 
 \draw (3.5,.5) node {8}; \draw (3.5+3,.5+3) node {\footnotesize 8};  
 \draw[->] (0,3.5)--(0,-1);
\draw[->] (-.5,3)--(7.5,3);
\draw (1.5,4) node {$\bT$};
\draw[->] (9,1.5)-- node[above]{$\bij$}(8,1.5);

\begin{scope}[shift={(13,0)}]
\fill[pink!50] (-3,-.5) rectangle (0,0);
\fill[pink!50] (1,3) rectangle (4,3.5);
\fill[pink] (-2,0) \north\east\north\east\north\east \east\east  \south\west\west\south\south\west \west\west;
 \draw[very thin,dotted] (-2,0)  grid (3,3);
 \draw[brown, very thick] (-3,-.5)--++(0,.5) \east \north\east\north\east\north\east -- ++(0,.5);
 \draw[violet, ultra thick] (4,3.5)-- ++(0,-.5)\west\south\west\west\south\south\west -- ++(0,-0.5);
 \draw (-.5,1.5) node {1}; 
 \draw (-1.5,.5) node {2}; \draw (-1.5+3,.5+3) node {\footnotesize 2};  
 \draw (-.5,.5) node {3}; \draw (-.5+3,.5+3) node {\footnotesize 3};  
 \draw (.5,2.5) node {4}; \draw (.5-3,2.5-3) node {\footnotesize 4}; 
 \draw (.5,1.5) node {5}; 
 \draw (1.5,2.5) node {6}; \draw (1.5-3,2.5-3) node {\footnotesize 6}; 
 \draw (2.5,2.5) node {7}; \draw (2.5-3,2.5-3) node {\footnotesize 7}; 
 \draw (.5,.5) node {8}; \draw (.5+3,.5+3) node {\footnotesize 8};
 \draw[violet] (1.5,1.5) node {$\balpha$};
 \draw[->] (0,3.5)--(0,-1);
\draw[->] (-3,3)--(4.5,3);
\draw (-.5,4) node {$\bP$};
\end{scope}
\end{scope}
\end{tikzpicture}
\caption{The complements of the tableaux $P$ and $T$ from Figure~\ref{fig:bij} satisfy $\bij(\bP)=\bT$.}
\label{fig:complement}
\end{figure}

One property of $\CRSK$, already observed by Neyman~\cite{neyman_cylindric_2015}, is that its inverse can be computed by simply applying $\CRSK$ to the pair of complement tableaux. We use cylindric growth diagrams to provide a simpler proof of this property.

\begin{corollary}[{\cite[Cor.\ 4.19]{neyman_cylindric_2015}}]\label{cor:involution}
If $\CRSK(T,U)=(P,Q)$ then $\CRSK(\bP,\bQ)=(\bT,\bU)$.
In particular, if $\bij(T)=P$, then $\bij(\bP)=\bT$. 
\end{corollary}

\begin{proof}
For any $T\in\SSCT_{d,L}(\alpha/\mu)$ with complement $\bT\in\SSCT_{d,L}(\bar\mu/\bar\alpha)$, the sequence $\ssh(\bT)$ is obtained by reversing $\ssh(T)$ and replacing each shape $\rho$ with its complement shape~$\bar\rho$. 
Additionally, rotating the square in equation~\eqref{eq:square} by $180^\circ$ and replacing each label $\rho$ by $\bar\rho$, the forward local rule in equation~\eqref{eq:forward} becomes the backward local rule in  equation~\eqref{eq:backward}. 

We deduce that the cylindric growth diagram that computes $\CRSK(T,U)=(P,Q)$, after replacing each shape by its complement and rotating the grid by $180^\circ$, becomes a cylindric growth diagram that computes $\CRSK(\bP,\bQ)=(\bT,\bU)$.
\end{proof}

Corollary~\ref{cor:involution} implies that the diagram in Figure~\ref{fig:complement} commutes, and that, for any standard cylindric tableau $P$, we have $\bij^{-1}(P)=\overline{\bij(\bP)}$.

\subsection{Evacuation of standard Young tableaux in terms of $\CRS$}

Evacuation is an involution on the set of linear extensions of a poset~\cite{stanley_promotion_2009}. 
It was introduced in the study of the Robinson--Schensted correspondence by Sch\"utzenberger~\cite{schutzenberger_quelques_1963}, who originally defined it as a map on standard Young tableaux. In this subsection we show that there is a simple description of this map in terms of the cylindric Robinson--Schensted correspondence.

For $d,L\ge n$, elements of $\SCT_{d,L}^n(\cdot/[0^d])$ can be viewed as standard Young tableaux of straight shape, since the cylindric setting does not impose any additional conditions on the entries of the tableau. Denote by $\SYT^n$ the set of standard Young tableaux of straight shape with $n$ cells.
By complementation, the set $\SCT_{d,L}^n([0^d]/\cdot)$ is in bijection with $\SCT_{d,L}^n(\cdot/[L^d])$, which in turn is in bijection with $\SCT_{d,L}^n(\cdot/[0^d])$ by subtracting $L$ from each integer, and hence in bijection with $\SYT^n$.
Thus, in the case $d,L\ge n$ and $\alpha=[0^d]$, we can interpret the map $P\mapsto\bij(\bP)$, where $\bij$ is the bijection from Corollary~\ref{cor:bij}, as an involution on $\SYT^n$. See Figure~\ref{fig:evacuation} for an example.

\begin{figure}[htb]
\centering
\begin{tikzpicture}[scale=0.5]
\fill[pink] (0,-4) \north\north\north\north \east\east\east\east \south\west\west\south\south\west\south\west;
 \draw[very thin,dotted] (0,-4)  grid (4,0);
 \draw[very thick] (4,0)\south\west\west\south\south\west\south\west;
  \draw[violet, ultra thick] (0,-4.5) --++(0,.5) \north\north\north\north \east\east\east\east;
   \draw[violet, ultra thick,dotted] (4,0) \east;
      \draw[violet, ultra thick] (5,0) \east\north\north\north\north --++(0,.5);;
   \draw (.5,-.5) node {1}; 
 \draw (.5,-1.5) node {2}; 
 \draw (1.5,-.5) node {3};
 \draw (.5,-2.5) node {4}; 
 \draw (2.5,-.5) node {5}; 
 \draw (3.5,-.5) node {6}; 
 \draw (1.5,-1.5) node {7};
 \draw (.5,-3.5) node {8}; 
 \draw (1.5,-2.5) node {9}; 
\draw[->] (8.5,0)-- node[above]{complement}(9.5,0);
\draw (1.5,1) node {$P$};

\begin{scope}[shift={(12,0)}]
\fill[pink] (2,0) \east\east\east\east \north\north\north\north \west\south\west\south\south\west\west\south;
 \draw[very thin,dotted] (2,0)  grid (6,4);
 \draw[very thick] (6,4) \west\south\west\south\south\west\west\south;
  \draw[violet, ultra thick] (0,-4.5)  --++(0,.5) \north\north\north\north\east;
  \draw[violet, ultra thick,dotted] (1,0) \east;
  \draw[violet, ultra thick] (2,0)  \east\east\east\east \north\north\north\north --++(0,.5);
  \draw (4.5,2.5) node {1}; 
 \draw (5.5,3.5) node {2}; 
 \draw (4.5,1.5) node {3};
 \draw (2.5,.5) node {4}; 
 \draw (3.5,.5) node {5}; 
 \draw (5.5,2.5) node {6}; 
 \draw (4.5,.5) node {7};
 \draw (5.5,1.5) node {8}; 
 \draw (5.5,.5) node {9}; 

\draw (1,3) node {$\bP$};
\draw[->] (7.5,0)-- node[above]{$\bij$}(8.5,0);
\end{scope}

\begin{scope}[shift={(22.5,0)}]
\fill[yellow] (0,-4) \north\north\north\north \east\east\east\east \south\west\west\south\south\west\south\west;
 \draw[very thin,dotted] (0,-4)  grid (4,0);
 \draw[very thick] (4,0)\south\west\west\south\south\west\south\west;
  \draw[violet, ultra thick] (0,-4.5) --++(0,.5) \north\north\north\north \east\east\east\east;
   \draw[violet, ultra thick,dotted] (4,0) \east;
      \draw[violet, ultra thick] (5,0) \east\north\north\north\north --++(0,.5);;
   \draw (.5,-.5) node {1}; 
 \draw (1.5,-.5) node {2}; 
 \draw (.5,-1.5) node {3};
 \draw (.5,-2.5) node {4}; 
 \draw (1.5,-1.5) node {5}; 
 \draw (2.5,-.5) node {6}; 
 \draw (1.5,-2.5) node {7};
 \draw (3.5,-.5) node {8}; 
 \draw (.5,-3.5) node {9}; 
\draw (1.5,1) node {$\bij(\bP)$};
\end{scope}

\end{tikzpicture}
\caption{The map $P\mapsto\bij(\bP)$ for $d,L\ge n$ and $\alpha=[0^d]$.}
\label{fig:evacuation}
\end{figure} 

We will see that this involution coincides with evacuation on standard Young tableaux, which we
denote by $\evac$. We will prove a more general result that applies to pairs $(P,Q)$.
For $d,L\ge n$ and $\alpha=\beta=[0^d]\in\CS_{d,L}$, we can interpret the map $(P,Q)\mapsto\CRS(\bP,\bQ)$ 
as a map on pairs of standard Young tableaux of the same shape. We know that this map is an involution by Corollary~\ref{cor:involution}. 

\begin{proposition}
For $d,L\ge n$ and $\alpha=\beta=[0^d]$, the map $(P,Q)\mapsto\CRS(\bP,\bQ)$, when interpreted as an involution on pairs of standard Young tableaux, coincides with the map $(P,Q)\mapsto(\evac(P),\evac(Q))$. 
\end{proposition}

\begin{proof}
Let us analyze the insertions that take place in the computation of $\CRS(\bP,\bQ)$, as described in Theorem~\ref{thm:CRS}. We say that an entry is {\em bumped} when it is removed from a row and inserted in the next row. A bump from row $d$ (equivalently, row $0$) to row $1$ will be called a {\em bump around the edge}.
As we successively apply the internal row insertion operations $R_{i_1},R_{i_2},\dots,R_{i_n}$ to $\bP$, where $i_j$ is the row of $\bQ$ containing $j$, each insertion causes exactly one entry to be bumped around the edge. Let $\sigma=\sigma_1\sigma_2\dots\sigma_n$ be this sequence of entries, which is a permutation in the symmetric group $\S_n$. In the example in Figure~\ref{fig:evacuation} (where $P=Q$), we have $\sigma=493157682\in\S_9$. The bumps that take place when applying $R_{i_j}$ can then be separated into {\em early bumps} (those occurring up until the bump of $\sigma_j$ around the edge) and {\em late bumps} (those occurring afterwards).

Interpreting $\CRS(\bP,\bQ)$ as a pair of standard Young tableaux, this pair is precisely the image of $\sigma$ under the usual RS algorithm. Indeed, since $d\ge n$, the late bumps that occur when applying $R_{i_j}$ mimic the insertion of $\sigma_j$ via RS. 

Now let $\pi\in\S_n$ be the reverse-complement of $\sigma$, that is, $\pi=(n+1-\sigma_n)\dots(n+1-\sigma_2)(n+1-\sigma_1)$. Interpreting $(P,Q)$ as a pair of standard Young tableaux, we claim that this pair is precisely the image of $\pi$ under RS. This is because, for $j$ from $1$ to $n$, the bumps caused by the insertion of $\pi_j=n+1-\sigma_{n+1-j}$ via RS correspond, under complementation, to the inverses of the early bumps that occur when applying $R_{i_{n+i-j}}$ to the pair $(\bP,\bQ)$.
A similar argument is used in \cite[Cor.~4.19]{neyman_cylindric_2015} to show that the inverse of $\CRS$ can be computed by applying $\CRS$ to the complement tableaux.

Finally, it is a well-known property of the Robinson--Schensted correspondence \cite[Thm.~A1.2.10]{stanley_enumerative_1999} that if $\pi$ and $\sigma$ are reverse-complements of each other and $\RS(\pi)=(P,Q)$, then $\RS(\sigma)=(\evac(P),\evac(Q))$.
\end{proof}

\section{Bijections for oscillating cylindric tableaux}
\label{sec:OCT}

In this section we use cylindric growth diagrams to describe bijections for oscillating cylindric tableaux, and to give a proof of Theorem~\ref{thm:osc_walks}. Recall that $\OCT_{d,L}^w(\alpha,\beta)$ is the set of oscillating cylindric tableaux of type $w$ with initial shape $\alpha$ and final shape~$\beta$.
Denote by $\WW_{m,n}\subseteq\{+,-\}^{m+n}$ the set of words consisting of exactly $m$ pluses and $n$ minuses.
The following theorem can be thought of as a cylindric analogue of \cite[Thm.\ 4.2.10]{roby_applications_1991}.

\begin{theorem}\label{thm:OCT}
Let $m,n\ge0$ and let $\alpha,\beta\in\CS_{d,L}$ such that $|\alpha|+m=|\beta|+n$. For any $w\in\WW_{m,n}$, there is a bijection between $\OCT_{d,L}^w(\alpha,\beta)$ and the set of cylindric growth diagrams on an $m\times n$ grid where vertex $(0,n)$ has label $\alpha$ and vertex $(m,0)$ has label $\beta$.
 In particular, $\card{\OCT_{d,L}^w(\alpha,\beta)}$ is independent of $w\in\WW_{m,n}$, and it equals
$$\sum_{\substack{\mu\subseteq\alpha,\beta\\ |\alpha/\mu|=n,|\beta/\mu|=m}} \card{\SCT_{d,L}(\alpha/\mu)} \card{\SCT_{d,L}(\beta/\mu)}= \sum_{\substack{\lambda\supseteq\alpha,\beta\\ |\lambda/\beta|=n,|\lambda/\alpha|=m}} \card{\SCT_{d,L}(\lambda/\beta)} \card{\SCT_{d,L}(\lambda/\alpha)}.$$
\end{theorem}

\begin{proof}
Fix $w\in\WW_{m,n}$, and note that we can view $w$ as a lattice path from $(0,n)$ to $(m,0)$ with steps $\e=(1,0)$ and $\s=(0,-1)$ by replacing each $+$ with an $\e$ and each $-$ with an $\s$.

Given a tableau in $\OCT_{d,L}^w(\alpha,\beta)$, label the vertices of the above lattice path by the cylindric shapes in the sequence. 
Note that  vertex $(0,n)$ has label $\alpha$ and vertex $(m,0)$ has label $\beta$. This labeled path can be uniquely extended to a cylindric growth diagram on the $m\times n$ grid by using the forward and backward local rules. 

Conversely, given a cylindric growth diagram on an $m\times n$ grid, where vertex $(0,n)$ has label $\alpha$ and vertex $(m,0)$ has label $\beta$, the labels of the vertices of the above lattice path determine an element of $\OCT_{d,L}^w(\alpha,\beta)$. This completes the description of the bijection.

In the special case $w=-^n+^m$, tableaux in $\OCT_{d,L}^w(\alpha,\beta)$ describe the labels of the vertices on the left and bottom boundaries of the grid, and they can be viewed as pairs $(T,U)\in\SCT_{d,L}(\alpha/\mu)\times\SCT_{d,L}(\beta/\mu)$ for some $\mu\subseteq\alpha,\beta$ with $|\alpha/\mu|=n$ and $|\beta/\mu|=m$. 

Similarly, oscillating tableaux of type $w=+^m-^n$ describe the labels of the vertices on the right and top boundaries of the grid, which can be viewed as pairs $(P,Q)\in\SCT_{d,L}(\lambda/\beta)\times\SCT_{d,L}(\lambda/\alpha)$ for some $\lambda\supseteq\alpha,\beta$ with $|\lambda/\beta|=n$ and $|\lambda/\alpha|=m$.
\end{proof}

\begin{example}\label{ex:OCT} Let $\alpha=[3,3,1],\beta=[2,2,2]\in\CS_{3,2}$, and let $n=5$, $m=4$. For $w={+}{-}{+}{+}{-}{-}{+}{-}{-}\allowbreak\in\WW_{4,5}$,
the above bijection takes the oscillating tableau
$$[3,3,1],[3,3,2],[3,2,2],[4,2,2],[4,3,2],[4,2,2],[3,2,2],[3,3,2],[3,2,2],[2,2,2],$$
which is an element of $\OCT_{3,2}^w(\alpha,\beta)$, to the cylindric growth diagram in Figure~\ref{fig:growth}.
\end{example}

It is particularly interesting to consider the restriction of the above construction to the symmetric case.
Suppose that $m=n$ and $\alpha=\beta$, and consider oscillating cylindric tableaux in $\OCT_{d,L}^{2n}(\alpha,\alpha)$ of the form $\rho^0,\rho^1,\dots,\rho^{2n}$ such that $\rho^k=\rho^{2n-k}$ for all $0\le k\le n$. We call these {\em symmetric} oscillating cylindric tableaux. The type of such a tableau is a word in $\WW_{n,n}$ that is invariant under reversing the word and switching signs.
The corresponding lattice path from $(0,n)$ to $(n,0)$, with vertices labeled by the cylindric shapes in the sequence, is symmetric with respect to reflection along $y=x$. It follows that, when extending this path to the $n\times n$ grid by using the forward and backward local rules, the resulting cylindric growth diagram has the same symmetry.

Conversely, given a cylindric growth diagram on an $n\times n$ grid that is symmetric with respect to reflection along $y=x$, if we let $\alpha$ be the label of vertices $(0,n)$ and $(n,0)$, then the labels of the vertices along a symmetric path from $(0,n)$ to $(n,0)$ determine a symmetric oscillating cylindric tableau.

In this case, because of the symmetry, the tableau is determined by the shapes $\rho^0,\rho^1,\dots,\rho^n$, the type is determined by the first $n$ signs, and the growth diagram is determined by the labels of the vertices $(x,y)$ with $x\le y$.

\begin{theorem}\label{thm:OCT-symmetric}
Let $\alpha\in\CS_{d,L}$ and $w\in\{+,-\}^n$ for some $n\ge0$. There is a bijection between $\OCT_{d,L}^w(\alpha,\cdot)$ and the set of symmetric cylindric growth diagrams on an $n\times n$ grid where vertices $(0,n)$ and $(n,0)$ have label $\alpha$.

In particular, $\card{\OCT_{d,L}^w(\alpha,\cdot)}$ is independent of $w\in\{+,-\}^n$, and it equals
$\card{\SCT^n_{d,L}(\alpha/\cdot)} = \card{\SCT^n_{d,L}(\cdot/\alpha)}$.
\end{theorem}

\begin{proof}
Let $\overleftarrow{w}\in\{+,-\}^n$ be obtained from $w$ by reversing it and switching signs. There is a simple bijection between tableaux in $\OCT_{d,L}^w(\alpha,\cdot)$ and symmetric tableaux in $\OCT_{d,L}^{w\overleftarrow{w}}(\alpha,\alpha)$, obtained by mapping $\rho^0,\rho^1,\dots,\rho^n$ to $\rho^0,\rho^1,\dots,\rho^n,\dots,\rho^1,\rho^0$. These symmetric tableaux, in turn, are in bijection with symmetric cylindric growth diagrams on $n\times n$ where vertices $(0,n)$ and $(n,0)$ have label~$\alpha$, as can be seen by restricting the bijection from Theorem~\ref{thm:OCT} to the symmetric case.

In the special case $w=-^n$, tableaux in $\OCT_{d,L}^w(\alpha,\cdot)$ describe the labels of the vertices on the left boundary of the grid, which determine tableaux in $\SCT^n_{d,L}(\alpha/\cdot)$. Similarly, oscillating tableaux of type $w=+^n$ describe the labels of the vertices on the top boundary of the grid, which determine tableaux in $\SCT^d_{d,L}(\cdot/\alpha)$.
\end{proof}

We can now use Theorem~\ref{thm:equivalence-oscillating} to translate the above result to the other settings.

\begin{corollary}\label{cor:ww'}
For any  $w,w'\in\{+,-\}^n$, there are bijections between the sets in Theorem~\ref{thm:equivalence-oscillating} of type $w$ and those of type $w'$.
\end{corollary}

\begin{proof}
For any $\alpha\in\CS_{d,L}$, Theorem~\ref{thm:OCT-symmetric} gives a bijection between $\OCT_{d,L}^w(\alpha,\cdot)$ and $\OCT_{d,L}^{w'}(\alpha,\cdot)$.
On the other hand, we can use the bijections from Theorem \ref{thm:equivalence-oscillating} to translate between oscillating cylindric tableaux of types $w$ and $w'$ starting at $\alpha$, oscillating walks in $\Delta_{d,L}$ of types $w$ and $w'$ starting at vertex $\x$, oscillating walks in $\E_{d,L}$ of types $w$ and $w'$ starting at state $u$,
and oscillating walks in $\N_{d,L}$ of types $w$ and $w'$ starting at state $\bracket{u}$, respectively.
\end{proof}

In particular, Corollary~\ref{cor:ww'} gives a bijection between $\OW^w_{d,L}(\x)$ and $\OW^{w'}_{d,L}(\x)$, which proves Theorem~\ref{thm:osc_walks}.
The original proof in \cite[Sec.~2.2]{courtiel_bijections_2021} applies a sequence of certain flips to the walks, while transforming $w$ into $w'$. In \cite[Sec.~2.3]{courtiel_bijections_2021}, the authors recast this construction in terms of tilings of a tilted square using 9 kinds of square tiles (for the $d=3$ case), and then use induction to prove that the tiling exists and is unique. 
Even though there is no mention of growth diagrams or cylindric tableaux in~\cite{courtiel_bijections_2021}, one can show that these tiles correspond to the edge local rules in Equation~\eqref{eq:edge_rules}. Hence, when generalizing the tiling from \cite{courtiel_bijections_2021} to arbitrary $d$, the bijection from $\OW^w_{d,L}(\x)$ to $\OW^{w'}_{d,L}(\x)$ that it produces is equivalent to the one that we obtain in Corollary~\ref{cor:ww'} by using growth diagrams. This is illustrated in Figure~\ref{fig:Courtiel}, which shows the cylindric growth diagram that corresponds to the example in \cite[Fig.\ 7,8]{courtiel_bijections_2021}.

\newcommand\labarr[3]{
\draw[postaction={decorate},shorten <=4.2mm,shorten >=4.2mm,purple] (#1,#2) -- node[above,scale=.8]{$#3$} ++(1,0);
\draw[postaction={decorate},shorten <=4.2mm,shorten >=4.2mm,purple] (#2,#1) -- node[left,xshift=1,yshift=1,scale=.8]{$#3$} ++(0,1);
}

\begin{figure}[htb]
\centering
\begin{tikzpicture}[scale=1.2,decoration={
    markings,
    mark=at position .55 with {\arrow{>}}}
    ] 
\draw[dotted] (0,0) grid (3,3);
\draw (0,0) node {$[1\mn1\mn2]$};
\draw (1,0) node {$[1 0 \mn2]$}; \draw (0,1) node {$[1 0 \mn2]$};
\draw (2,0) node {$[1 0 \mn1]$}; \draw (0,2) node {$[1 0 \mn1]$};
\draw (3,0) node[violet] {$[1 0 0]$}; \draw (0,3) node[violet] {$[1 0 0]$};
\draw (3.3,0) node[right,violet] {$=\alpha$}; \draw (-.3,3) node[left,violet] {$\alpha=$}; 
\draw (1,1) node[purple] {$[1 0 \mn1]$}; 
\draw (2,1) node {$[2 0 \mn1]$}; \draw (1,2) node[purple] {$[2 0 \mn1]$};
\draw (3,1) node {$[2 0 0]$}; \draw (1,3) node[purple] {$[2 0 0]$};
\draw (2,2) node {$[2 1 \mn1]$};
\draw (3,2) node {$[2 1 0]$}; \draw (2,3) node {$[2 1 0]$};
\draw (3,3) node {$[3 1 0]$};
\draw[->] (1,-.4)--node[below]{$T$} (2,-.4);
\draw[->] (-.7,1)--node[left]{$T$} (-.7,2);
\draw[->] (1,3.5)--node[above]{$P$} (2,3.5);
\draw[->] (3.7,1)--node[right]{$P$} (3.7,2);

\labarr{0}{0}{(2)}
\labarr{1}{0}{(3)}
\labarr{2}{0}{(3)}
\labarr{0}{1}{(3)}
\labarr{1}{1}{(1)}
\labarr{2}{1}{(3)}
\labarr{0}{2}{(1)}
\labarr{1}{2}{(2)}
\labarr{2}{2}{(3)}
\labarr{0}{3}{(1)}
\labarr{1}{3}{(2)}
\labarr{2}{3}{(1)}

\begin{scope}[scale=.417,shift={(-8,2)}]
\fill[yellow!25] (1,3) rectangle (3,3.5);
\fill[yellow] (-2,0) \north\east\north\east\east\north \south\west\south\south\west\west;
 \draw[very thin,dotted] (-2,0)  grid (3,3);
  \draw[thick,<->,olive] (-.2,0)-- (-.2,3);
 \draw[thick,<->,olive] (0,3.2)-- (3,3.2); 
 \draw[black, very thick] (-2,-.5)--++(0,.5) \north\east\north\east\east\north -- ++(0,.5);
 \draw[violet, ultra thick] (3,3.5)-- ++(0,-.5)\west\west \south\west\south\south\west\west -- ++(0,-0.5);
 \draw (-.5,1.5) node {1}; 
 \draw (-1.5,.5) node {2}; \draw (-1.5+3,.5+3) node {\footnotesize 2};  
 \draw (-.5,.5) node {3}; \draw (-.5+3,.5+3) node {\footnotesize 3};  
 \draw[violet] (.5,1.5) node {$\alpha$};
 \draw[->] (0,3.5)--(0,-1);
\draw[->] (-2,3)--(3.5,3);
\draw (-.5,4) node {$T$};
\end{scope}

\begin{scope}[scale=.417,shift={(13.5,2)}]
\fill[pink!50] (-2,-.5) rectangle (0,0);
\fill[pink] (0,0) \north\north\east\north\east\east \south\west\west\south\west\south;
 \draw[very thin,dotted] (-2,0)  grid (3,3);
 \draw[black, very thick] (3,3.5)-- ++(0,-.5)\south\west\west\south\west\south  -- ++(0,-0.5);
 \draw[violet, ultra thick] (3,3.5)-- ++(0,-.5)\west\west \south\west\south\south\west\west -- ++(0,-0.5);
  \draw[violet] (-.5,1.5) node {$\alpha$};
 \draw (1.5,2.5) node {1}; \draw (1.5-3,2.5-3) node {\footnotesize 1}; 
 \draw (.5,1.5) node {2}; 
 \draw (2.5,2.5) node {3}; \draw (2.5-3,2.5-3) node {\footnotesize 3}; 
 \draw[->] (0,3.5)--(0,-1);
\draw[->] (-2,3)--(3.5,3);
\draw (2,4) node {$P$};
\end{scope}
\end{tikzpicture}
\caption{The symmetric cylindric growth diagram for the oscillating tableau $\alpha=[100],\allowbreak[200],\allowbreak[20\mn1],\allowbreak[10\mn1]$ in $\OCT_{3,3}^3(\alpha,\cdot)$ of type $w={+}{-}{-}$, which corresponds to the oscillating walk $s_1 \bars_3 \bars_1$ in $\Delta_{3,3}$ starting at $\x=f(\alpha)=(2,1,0)$.}
\label{fig:Courtiel}
\end{figure}

\section{Further directions}\label{sec:further}

\subsection{Crossings and nestings in matchings}

Huh, Kim, Krattenthaler and Okada \cite{huh_bounded_2025} have recently found a surprising connection between standard cylindric tableaux and certain matchings.
In our notation, letting $[0^d]=[0,\dots,0]\in\Delta_{d,L}$, they prove in \cite[Cor.~8.9]{huh_bounded_2025} that $\card{\SCT^n_{d,L}(\cdot/[0^d])}$ equals the number of certain partial matchings on $n$ points.
Specifically, if $d=2h+1$ and $L=2w+1$ for some $h,w\ge1$, then $\card{\SCT^n_{d,L}(\cdot/[0^d])}$ is the number of partial matchings on $[n]$ with no $(h+1)$-crossing and no $(w+1)$-nesting (see~\cite{huh_bounded_2025} for definitions).
They also give similar identities when $d$ or $L$ are even.

When $d=3$, this result is equivalent to Theorem~\ref{thm:mortimer_number_2015}, by interpreting noncrossing matchings as Motzkin paths. And when $d\to\infty$ and $L=2w+1$, it is equivalent to the fact, which can be proved using the RS correspondence, that the number of standard Young tableaux with $n$ cells and at most $2w+1$ columns equals the number of partial matchings on $[n]$ with no $(w+1)$-nesting.

The proofs in~\cite{huh_bounded_2025} are computational, based on explicit determinantal formulas. In \cite[Prob.~8.10]{huh_bounded_2025}, the authors ask for a bijective proof of their equalities.  It is our hope that the cylindric growth diagrams that we introduced in this paper, or some generalization of them, may be helpful in finding the elusive bijection between standard cylindric tableaux and restricted partial matchings.

It is known~\cite{chen_crossings_2007,krattenthaler_growth_2006} that partial matchings on $[n]$ with no $(h+1)$-crossing and no $(w+1)$-nesting are in bijection with $n$-step {\em vacillating tableaux} (similar to oscillating tableaux, but allowing steps that do no not add or remove a cell) where the partitions have at most $h$ rows and at most $w$ columns.

\subsection{The operators $U$ and $D$}

As discussed in Section~\ref{sec:growth}, the poset $(\CS_{d,L},\subseteq)$ of cylindric shapes ordered by containment is not an $r$-differential poset in the sense of \cite[Def.~1.1]{stanley_differential_1988}, but it has similar properties if we allow $r=0$ and do not require the poset to have a minimal element. 
In particular, similarly to~\cite{stanley_differential_1988}, one can define linear transformations $U$ and $D$ on the vector space of linear combinations of elements of $\CS_{d,L}$. In our poset, these transformations satisfy $DU-UD=0$, i.e., they commute. It follows, using the notation from Theorem~\ref{thm:OCT}, that $\card{\OCT_{d,L}^w(\alpha,\beta)}$ is independent of $w\in\WW_{m,n}$.
One could ask if parts of the theory of $r$-differential posets, such as the enumerative properties studied in \cite[Sec.~3]{stanley_differential_1988}, have analogues for the poset of cylindric shapes.

\subsection{Promotion on standard cylindric tableaux}

By viewing standard cylindric tableaux as linear extensions of posets, one can define a promotion operation on $\SCT(\lambda/\mu)$ via the usual sliding procedure~\cite{stanley_promotion_2009}. The slides in this case are the cylindric analogue of Sch\"utzenberger's {\it jeu de taquin} for standard Young tableaux~\cite{schutzenberger_correspondance_1977}.
This promotion operation can also be described in terms of growth diagrams, following the approach from \cite[Sec.~5]{stanley_promotion_2009}, which generalizes Fomin's growth diagram description of {\it jeu de taquin} \cite[Sec.~A1.2]{stanley_enumerative_1999}.

One could ask if promotion on standard cylindric tableaux has any interesting properties. For example, one could look for shapes $\lambda/\mu$ for which the order of promotion on $\SCT(\lambda/\mu)$ is small (compared to the cardinality of the set), as is the case for standard Young tableaux of rectangular or staircase shapes.

\subsection*{Acknowledgments}
The author thanks Alex Postnikov, Darij Grinberg, Christian Krattenthaler and Tom Roby for helpful conversations about cylindric tableaux and growth diagrams, as well as an anonymous referee for useful suggestions.
This work was partially supported by Simons Collaboration Grant \#929653.

\bibliographystyle{plain}
\bibliography{walks_simplices}

\end{document}